\documentclass[amssymb,amsfonts,refcheck,12pt,verbatim,righttag]{amsart}

\usepackage{amssymb,mathrsfs}

\usepackage{graphicx}
\usepackage{graphics,color}

\usepackage{amsfonts}
\usepackage{color}
\usepackage{amssymb}

\numberwithin{equation}{section} % \renewcommand{\rm}{\normalshape} %
\theoremstyle{plain}

\newtheorem{thm}{Theorem}[section]
\newtheorem{lem}[thm]{Lemma}
\newtheorem{pro}[thm]{Proposition}
\newtheorem{cor}[thm]{Corollary}
\newtheorem{ex}[thm]{Example}
\newtheorem{de}[thm]{Definition}
\newtheorem{rem}[thm]{Remark}

\def\R {{\Bbb R}}
\def\N {{\Bbb N}}

\def\ba{{\bf a}}

\def\va{{\bf a}}
\def\bv{{\bf v}}
\def\s{{\rm span}}

\newcommand{\w}{\wedge}

\newcommand{\ldim}[1]{\underline{\dim}_{\mathrm{#1}}}
\newcommand{\udim}[1]{\overline{\dim}_{\mathrm{#1}}}

\DeclareMathOperator*{\esssup}{ess\,sup}
\DeclareMathOperator*{\essinf}{ess\,inf}

\usepackage{float}

\begin{document}
\baselineskip 15pt

\baselineskip 15pt

\baselineskip 13.7pt
\title{Dimensions of orthogonal projections of typical self-affine sets and measures}

\author{De-Jun FENG}
\address{
Department of Mathematics\\
The Chinese University of Hong Kong\\
Shatin,  Hong Kong\\
}
\email{djfeng@math.cuhk.edu.hk}

\author{Yu-Hao Xie}
\address{
Department of Mathematics\\
The Chinese University of Hong Kong\\
Shatin,  Hong Kong\\
}
\email{yhxie@math.cuhk.edu.hk}



%\thanks { The research of Feng  was partially supported by the HKRGC GRF grant}
%2000 {\it Mathematics Subject Classification}: 28A80, 37C45}

\date{}

\begin{abstract}
Let $\{T_i\}_{i=1}^m$ be a family of $d\times d$   invertible real matrices with $\|T_i\|<1/2$ for $1\leq i\leq m$.
 For $\ba=(a_1,\ldots, a_m)\in \R^{md}$, let  $\pi^{\ba}\colon \Sigma=\{1,\ldots, m\}^\N\to \R^d$ denote the coding map associated with the
  affine IFS $\{T_ix+a_i\}_{i=1}^m$, and let $K^\ba$ denote the attractor of this IFS. Let $W$ be a linear subspace of $\R^d$ and $P_W$ the orthogonal projection onto $W$.   We show that for $\mathcal L^{md}$-a.e.~$\ba\in \R^{md}$,   the Hausdorff and  box-counting dimensions of $P_W(K^\ba)$ coincide and are determined by the zero point of a certain pressure function associated with $T_1,\ldots, T_m$ and $W$. Moreover, for every ergodic $\sigma$-invariant measure $\mu$ on $\Sigma$ and for $\mathcal L^{md}$-a.e.~$\ba\in \R^{md}$, the local dimensions of $(P_W\pi^\ba)_*\mu$ exist almost everywhere, here $(P_W\pi^\ba)_*\mu$ stands for  the push-forward of $\mu$ under the map $P_W\pi^\ba$. However, as illustrated by examples, $(P_W\pi^\ba)_*\mu$ may not be  exact dimensional for $\mathcal L^{md}$-a.e.~$\ba\in \R^{md}$. Nevertheless, when $\mu$ is a Bernoulli product measure, or more generally, a supermultiplicative ergodic $\sigma$-invariant measure, $(P_W\pi^\ba)_*\mu$ is exact dimensional for $\mathcal L^{md}$-a.e.~$\ba\in \R^{md}$.
  \end{abstract}

\maketitle

\setcounter{section}{0}
\section{Introduction}\label{S1}
\setcounter{equation}{0}

The main purpose of this paper is to study the dimensions of orthogonal projections of  `typical' self-affine sets and measures along specific directions.

It is a fundamental question in fractal geometry and dynamical systems to investigate orthogonal projections of many concrete fractal sets and measures and go beyond general results such as Marstrand's theorem (see, e.g., \cite{FalconerFraserJin2015, Shmerkin2015, BSS2023}).  Before introducing the background and our results of this study, let us first provide some necessary notation and definitions regarding various dimensions of sets and measures. For a set $A\subset \R^d$, we use $\dim_{\rm H}A$, $\dim_{\rm P}A$, $\underline{\dim}_{\rm B}A$, and $ \overline{\dim}_{\rm B}A$ to denote the Hausdorff, packing, lower box-counting, and upper box-counting dimensions of $A$, respectively (see, e.g., \cite{Fal03, Mattila1995} for the definitions).  If $\underline{\dim}_{\rm B}A= \overline{\dim}_{\rm B}A$, we denote the common value by $\dim_{\rm B}A$ and refer to it as the {\it box-counting dimension} of $A$.

Recall that for a probability measure $\eta$ on $\R^d$,  the {\it local upper and lower dimensions} of $\eta$ at $x\in \R^d$ are defined as follows:
$$\overline{\dim}_{\rm loc}(\eta, x)=\limsup_{r\to 0}\frac{\log \eta (B(x,r))}{\log r},\quad \underline{\dim}_{\rm loc}(\eta, x)=\liminf_{r\to 0}\frac{\log \eta (B(x,r))}{\log r},$$
 where $B(x,r)$ denotes the closed ball of radius $r$ centered at $x$. If $\overline{\dim}_{\rm loc}(\eta, x)=\underline{\dim}_{\rm loc} (\eta, x)$, we denote the common value  as $\dim_{\rm loc}(\eta,x)$ and refer to it as the {\it local dimension} of $\eta$ at $x$. We say that  $\eta$  is  {\it exact dimensional}
if there exists a constant $C$ such that the  local dimension
$\dim_{\rm loc}(\eta, x)$
exists and equals $C$ for almost every $x$ with respect to $\eta$ (i.e., for $\eta$-a.e.~$x\in \R^d$). It is well known that if $\eta$ is  exact dimensional, then the lower and upper Hausdorff and packing dimensions of $\eta$  coincide and  equal the constant $C$ (see e.g., \cite{Fal-technique, Young1982}); in this case, we simply denote this constant as $\dim \eta$ and call it the dimension of $\eta$.  Recall that the lower and upper Hausdorff and packing dimensions of $\eta$ are defined as follows:
\begin{equation*}
\begin{split}
\ldim{H} \eta &= \essinf_{x\in {\rm spt}(\eta)} \underline{\dim}_{\rm loc} (\eta, x) , \quad
\udim{H} \eta = \esssup_{x\in {\rm spt}(\eta)} \underline{\dim}_{\rm loc} (\eta, x),\\
\ldim{P} \eta &= \essinf_{x\in {\rm spt}(\eta)}\overline{\dim}_{\rm loc} (\eta, x), \quad
\udim{P} \eta = \esssup_{x\in {\rm spt}(\eta)}\overline{\dim}_{\rm loc} (\eta, x).
\end{split}
\end{equation*}

Next, let us introduce the definition of self-affine sets. By an {\it affine iterated function system} (affine IFS) on $\R^d$  we mean a finite  family $\{f_i\}_{i=1}^m$ of affine mappings from $\R^d$ to $\R^d$,
taking  the form
\begin{equation*}
\label{e-form}
f_i(x)=T_ix+a_i,\qquad i=1,\ldots, m,
\end{equation*}
where $T_i$ are contracting $d\times d$ invertible real  matrices and $a_i\in \R^d$.  It is  well known \cite{Hut81} that, for any such IFS $\{f_i\}_{i=1}^m$,   there  exists a unique non-empty compact set $K\subset \R^d$ such that
$$
K=\bigcup_{i=1}^m f_i(K).
$$
We call $K$ the {\it attractor} of $\{f_i\}_{i=1}^m$, or the {\em self-affine set} generated by $\{f_i\}_{i=1}^m$.
In particular, if all the maps $f_i$ are contracting similitudes, we call $K$ a {\em self-similar set}.

In what follows, we let ${\bf T}=(T_1,\ldots, T_m)$ be a fixed tuple of contracting $d\times d$ invertible real  matrices. Let $(\Sigma,\sigma)$ be the one-sided full shift over the alphabet $\{1,\ldots, m\}$, that is, $\Sigma=\{1,\ldots, m\}^\N$ and $\sigma:\Sigma\to \Sigma$ is  the left shift map.
For $\va = (a_1, \ldots, a_m) \in \R^{md}$, let $\pi^\ba\colon \Sigma \to \R^d$ be the coding map associated with the IFS $\{ f_i^{\ba}(x) = T_ix + a_i\}_{i=1}^m$, here we write $f_i^\ba$ instead of $f_i$ to emphasize its dependence of $\ba$. That is,
\begin{equation}
\label{e-pia}
\pi^\ba (x) = \lim_{n \to \infty} f^{\va}_{x_1} \circ \cdots \circ f^{\va}_{x_n}(0)
\end{equation}
for $x =(x_n)_{n=1}^\infty\in \Sigma$. It is well known \cite{Hut81} that the image $\pi^\ba(\Sigma)$ of $\Sigma$ under $\pi^\ba$ is precisely the attractor of $\{f^\ba_i\}_{i=1}^m$. For convenience, we write $K^\ba=\pi^\ba(\Sigma)$. For an ergodic $\sigma$-invariant measure $\mu$ on $\Sigma$, let $\pi^\ba_*\mu$ denote the push-forward of $\mu$ by $\pi^\ba$, and we call it an {\it ergodic stationary measure} associated with the IFS $\{f_i^\ba\}_{i=1}^m$. It is known \cite{Feng2023} that every ergodic stationary measure associated with an affine  IFS is exact dimensional; see also \cite{BaranyKaenmaki2017,FengHu2009} for some earlier results.

In his seminal 1988 paper \cite{Fal88},  Falconer  introduced a quantity associated with the matrices $T_1,\ldots, T_m$, now commonly referred to as  the {\em affinity dimension}, denoted by
$\dim_{\rm AFF}(T_1,\ldots, T_m)$ (its definition will be presented shortly).  Falconer showed that this quantity always provides an upper bound for the upper box-counting dimension of $K^\ba$. Furthermore,  if
\begin{equation}
\label{e-norm}
\|T_i\|<\frac{1}{2} \quad \text{for all } 1\leq i\leq m,
\end{equation}
then    $$\dim_{\rm H}K^\ba=\dim_{\rm B}K^\ba=\min \{d, \dim_{\rm AFF}(T_1,\ldots, T_m)\}$$
holds  for ${\mathcal L}^{md}$-a.e.~translation vector $\ba$. In fact, Falconer initially proved this with $1/3$ in place of $1/2$;  it was later shown by Solomyak \cite{Sol98} that the bound $1/2$ suffices.

In 2007, Jordan, Pollicott and Simon \cite{JPS07} proved analogous results for ergodic stationary measures. For each ergodic $\sigma$-invariant measure $\mu$ on $\Sigma$, they showed that  $\dim \pi^\ba_*\mu$ is bounded above by the Lyapunov dimension $\dim_{\rm LY}(\mu, {\bf T})$  for every translation vector $\ba\in \R^{md}$  (see Definition \ref{de-Lyapunov} for the definition of Lyapunov dimension). Moreover,  under assumption \eqref{e-norm}, the equality
$$
\dim \pi^\ba_*\mu=\min\{d, \dim_{\rm LY}(\mu, {\bf T})\}
$$
holds  for ${\mathcal L}^{md}$-a.e.~$\ba\in \R^{md}$.
%We point out that there exists at least one ergodic $\sigma$-invariant measure $\mu$ such that $\dim_{\rm LY}(\mu, {\bf T})=\dim_{\rm AFF}({\bf T})$. This result is due to K\"aenm\"aki \cite{Kaenmaki2004}, who actually proved that that under the assumption of \eqref{e-norm}, there exists such an ergodic measure for which  $\dim_{\rm H}\pi^\ba_*\mu=\dim_{\rm H} K^\ba$ for almost every $\ba$.

Now, let us present the definition of  affinity dimension. Let ${\rm Mat}_d(\R)$ denote the collection of all $d\times d$  real matrices. For $A\in {\rm Mat}_d(\R)$, let $$\alpha_1(A)\geq\cdots\geq \alpha_d(A)$$ denote the  singular values of $A$, that is, the positive
square roots of the eigenvalues of the positive semi-definite matrix $A^*A$, where $A^*$ denotes the transpose of $A$.

	 Following Falconer's notation,  for $s\geq 0$  the {\it singular value function} $\varphi^s\colon {\rm Mat}_d(\R)\to [0,\infty)$ is defined by
\begin{equation}
\label{e-singular}
\varphi^s(A)=\left\{
\begin{array}
{ll}
\alpha_1(A)\cdots \alpha_j(A) \alpha_{j+1}^{s-j}(A) & \mbox{ if }0\leq s< d,\\
\det(A)^{s/d} & \mbox{ if } s\geq d,
\end{array}
\right.
\end{equation}
where $j=\lfloor s \rfloor$ is the integer part of $s$.
%The singular value function is sub-multiplicative, i.e.~ $\varphi^s(AB)\leq \varphi^s(A)\varphi^s(B)$; see \cite[Lemma 2.1]{Fal88}.
Then the affinity dimension of the tuple ${\bf T}=(T_1,\ldots, T_m)$ is defined by
\begin{equation}\label{e-aff}
\dim_{\rm AFF}({\bf T})=\inf\left\{s\geq 0\colon \lim_{n\to \infty}\frac{1}{n} \log\! \sum_{i_1,\ldots, i_n=1}^m\varphi^s(T_{i_1}\cdots T_{i_n})\leq 0\right\}.
\end{equation}

For a linear subspace $W\subset\R^d$, let $P_W$ denote the orthogonal projection onto $W$. In this paper, we aim to study the dimensions of  $P_W(K^\ba)$ and $(P_W\pi^\ba)_*\mu$ for almost every translation vector $\ba$, where $\mu$ is an ergodic $\sigma$-invariant measure on $\Sigma$ and  $(P_W\pi^\ba)_*\mu$ denotes the push-forward of $\mu$ under the map $P_W\pi^\ba$.

 To state our main results, we still need to introduce some notation and definitions. Write $\Sigma_0=\{\emptyset\}$, where $\emptyset$ stands for the empty word, and  $\Sigma_n=\{1,\ldots,m\}^n$  for $n\geq 1$. Set $\Sigma_*=\bigcup_{n=0}^\infty \Sigma_n$.
For a linear subspace $W$ of $\R^d$, define
 \begin{equation}
 \label{e-affine}
 \dim_{\rm AFF}({\bf T}, W)=\inf\left\{s\geq 0\colon \sum_{n=1}^\infty\sum_{I\in \Sigma_n}\varphi^s(P_WT_I)<\infty\right\},
 \end{equation}
where $T_I=T_{i_1}\cdots T_{i_n}$ for $I=i_1\ldots i_n$. It was recently  proved by Morris \cite[Theorem 1]{Morris2023} that $$\overline{\dim}_{\rm B}P_W(K^\ba)\leq \dim_{\rm AFF}({\bf T}, W)$$
for every $\ba\in \R^{md}$ and each linear subspace $W$ of $\R^d$.

For $k\in \{1,\ldots, d-1\}$,  the collection of all $k$-dimensional linear subspaces of $\R^d$ forms the Grassmann manifold $G(d,k)$. Let  $\gamma_{d,k}$ denote the natural invariant measure on  $G(d,k)$, which is locally  equivalent to the $k(d-k)$-dimensional Lebesgue measure; see \cite[p.~51]{Mattila1995} for a detailed definition.

 A set $\mathcal A\subset {\rm Mat}_d(\R)$ is said to be {\it irreducible} if  there is no proper nontrivial
linear subspace $V$ of $\R^d$ such that $A(V ) \subset V$ for all $A\in \mathcal A$; otherwise ${\mathcal A}$ is called {\it reducible}.  

For $T\in {\rm Mat}_d(\R)$ and $k\in \{1,\ldots, d\}$, let $T^{\wedge k}$ denote the $k$-fold exterior product of $T$; see Section~\ref{S-2.1} for the definition.
 For a pair of integers $n\geq q\geq 0$, let
$$
\binom{n}{q}=\frac{n\,!}{q\,!(n-q)\,!}
$$
denote the coefficient of the term $x^q$ in the expansion of $(1+x)^n$.

 The first result of this paper is the following.

\begin{thm}
\label{thm-1.1}
Let $k\in \{1,\ldots, d-1\}$. Then the following statements hold.
\begin{itemize}
\item[(i)] Let $\ell$ be the smallest integer not less than $\min\{k, \dim_{\rm AFF}({\bf T})\}$. Then, as $W$ ranges  over $G(d,k)$, $\dim_{\rm AFF}({\bf T}, W)$ can take only finitely many values; more precisely,  $$\#\{\dim_{\rm AFF}({\bf T}, W)\colon W\in G(d,k)\}\leq \min_{q\in \N\colon \ell \leq q\leq k} \binom{d}{q}-\binom{k}{q}+1,$$ where $\#$ stands for cardinality. Moreover, if  $\left\{T_i^{\wedge q}\colon i=1,\ldots, m\right\}$ is irreducible for some integer $q$ with $\ell\leq q\leq k$, then
     \begin{equation}
     \label{e-AFFW}
     \dim_{\rm AFF}({\bf T}, W)=\min\{k, \dim_{\rm AFF}({\bf T})\}
     \end{equation}
    for all $W\in G(d,k)$.
    \medskip

\item [(ii)] Equality \eqref{e-AFFW} holds for $\gamma_{d,k}$-a.e.~$W\in G(d,k)$. Here no irreducibility assumption on ${\bf T}$ is required.
    \medskip

\item[(iii)] Assume that $\|T_i\|< 1/2$ for all $1\leq i\leq m$. Let $W\in G(d,k)$. Then
    $$
    \dim_{\rm H}P_W(K^\ba)=\dim_{\rm B}P_W(K^\ba)=\dim_{\rm AFF}({\bf T}, W)
    $$
    for ${\mathcal L}^{md}$-a.e.~$\ba\in \R^{md}$.  Moreover, if
    \begin{equation}
    \label{e-lebcondition}
    \lim_{n\to \infty}\frac{1}{n}\log \sum_{I\in \Sigma_n}\varphi^k(P_WT_I)>0,
    \end{equation}
     then
    ${\mathcal H}^k (P_W(K^\ba))>0$ for ${\mathcal L}^{md}$-a.e.~$\ba\in \R^{md}$, where ${\mathcal H}^k$ stands for the $k$-dimensional Hausdorff measure.
\end{itemize}
\end{thm}

We remark that the limit in \eqref{e-lebcondition} always exists; see Proposition \ref{pro-equivalence} for a more general statement.

 To state our second result, let $\mu$ be an ergodic $\sigma$-invariant measure on $\Sigma$ and $W\in G(d,k)$.  For $x = (x_i)_{i =1}^\infty \in \Sigma$ and $n \in \N$,   define
\begin{equation}
\label{e-2.2}
S_n(\mu,{\bf T}, W,x)=\sup\left\{s\in [0,k]\colon \varphi^s(P_WT_{x|n})\geq  \mu ([x_1 \cdots x_n])\right\},
\end{equation}
where $x|n:=x_1\ldots x_n$,  $T_{x|n}:=T_{x_1}\cdots T_{x_n}$ and $$[x_1 \ldots x_n]: = \{ (y_i)_{i =1}^\infty\in \Sigma\colon y_i = x_i \text{ for } i = 1, \dots, n\}.$$
Equivalently,
\begin{equation}\label{e-equi}
S_n(\mu, \mathbf{T}, W, x) =
\begin{cases}
k, & \text{if } \varphi^k(P_W T_{x|n}) \geq \mu([x|n]), \\
s \in [0,k) \text{ with } \varphi^s(P_W T_{x|n}) = \mu([x|n]), & \text{otherwise}.
\end{cases}
\end{equation}
Next, define
\begin{equation}
\label{e-Smux}
S(\mu, {\bf T}, W, x)= \lim_{n\to \infty} S_n(\mu,{\bf T}, W,x),
\end{equation}
if the limit exists.

A probability measure $\eta$ on $\Sigma$ is said to be {\em supermultiplicative} if there exists $C>0$ such that
$$
\eta([IJ])\geq C\eta([I])\eta([J]) \qquad \text{for all }I,J\in \Sigma^*.
$$

Now we are ready to formulate our second result.

\begin{thm}
\label{thm-1.2}
 Let $\mu$ be an ergodic $\sigma$-invariant measure on $\Sigma$.  Then there exists a Borel subset $\Sigma'$ of $\Sigma$ with $\mu(\Sigma')=1$ such that the following properties hold.
\begin{itemize}
\item[(i)]  The limit in \eqref{e-Smux}  defining
$S(\mu, {\bf T}, W, x)$ exists for every $x\in \Sigma'$ and $W\in G(d,k)$. Moreover,
$$\#\{S(\mu, {\bf T}, W, x)\colon W\in G(d,k),\; x\in \Sigma'\}\leq \binom{d+\ell'-k}{\ell'}, $$
 where  $\ell'$ is the smallest integer not less than $\min\{k, \dim_{\rm LY}(\mu, {\bf T})\}$. Here  $\dim_{\rm LY}(\mu, {\bf T})$ denotes the Lyapunov dimension of $\mu$  with respect to ${\bf T}$;  see Definition \ref{de-Lyapunov}.
\medskip

 \item[(ii)] For every  $W\in G(d,k)$ and
  $\ba\in \R^{md}$,
\begin{align*}
    \overline{\dim}_{\rm loc}((P_W\pi^\ba)_*\mu, P_W\pi^\ba x)&\leq S(\mu, {\bf T}, W,x) \quad \mbox{for $\mu$-a.e.~$x\in \Sigma'$};
      \end{align*}
    and consequently,
      \begin{align*}
   \ldim{H} (P_W\pi^\ba)_*\mu &\leq   \underline{S}(\mu,{\bf T}, W),\quad
\udim{H} (P_W\pi^\ba)_*\mu \leq  \overline{S}(\mu,{\bf T}, W),\\
\ldim{P}(P_W\pi^\ba)_*\mu &\leq  \underline{S}(\mu, {\bf T},W),\quad
\udim{P} (P_W\pi^\ba)_*\mu \leq \overline{S}(\mu,{\bf T}, W),
\end{align*}
where
\begin{equation}
\label{e-Gammamu}
\underline{S}(\mu,{\bf T}, W)  := \essinf_{x\in {\rm spt}\mu} S(\mu, {\bf T}, W, x), \quad
\overline{S}(\mu,{\bf T},W)  := \esssup_{x\in {\rm spt}\mu} S(\mu, {\bf T}, W, x).
\end{equation}
\item[(iii)] Assume  additionally that $\mu$ is fully supported and supermultiplicative. Then $\underline{S}(\mu,{\bf T}, W)=\overline{S}(\mu, {\bf T},W)$ for all $W\in G(d,k)$.  If furthermore $\{T_i^{\wedge q}\}_{i=1}^m$ is irreducible for some integer $q$ such that $\ell'\leq q\leq k$, where  $\ell'$ is the smallest integer not less than $\min\{k, \dim_{\rm LY}(\mu, {\bf T})\}$, then
\begin{equation}
\label{e-smuw}
\underline{S}(\mu, {\bf T},W)=\overline{S}(\mu, {\bf T},W)=\min\{k,\dim_{\rm LY}(\mu, {\bf T})\}
\end{equation}
 for all $W\in G(d,k)$.
\medskip
\item[(iv)] Assume that $\|T_i\|<1/2$ for $1\leq i\leq m$. Let $W\in G(d,k)$. Then for $\mathcal L^{md}$-a.e.~$\ba\in \R^{md}$,
    \begin{align*}
    {\dim}_{\rm loc}((P_W\pi^\ba)_*\mu, P_W\pi^\ba x)&= S(\mu,{\bf T},W, x)
    \end{align*}
for $\mu$-a.e.~$x\in \Sigma$. Consequently, for $\mathcal L^{md}$-a.e.~$\ba\in \R^{md}$,
   \begin{align*}
\ldim{H} (P_W\pi^\ba)_*\mu &=\ldim{P} (P_W\pi^\ba)_*\mu=   \underline{S}(\mu,{\bf T}, W),\\
\udim{H} (P_W\pi^\ba)_*\mu &= \udim{P}(P_W\pi^\ba)_*\mu= \overline{S}(\mu,{\bf T}, W).
\end{align*}
Moreover, \eqref{e-smuw} holds for $\gamma_{d,k}$-a.e.~$W\in G(d,k)$.
\end{itemize}
\end{thm}

According to Theorem \ref{thm-1.2}(iv), under the assumption of \eqref{e-norm}, for every ergodic $\sigma$-invariant measure $\mu$ on $\Sigma$ and every linear subspace $W$ of $\R^d$,  the local dimensions of   $(P_W\pi^\ba)_*\mu$ exist almost everywhere for $\mathcal L^{md}$-a.e.~$\ba\in \R^{md}$. However,  as  demonstrated by Example \ref{not exact dimensional}, $(P_W\pi^\ba)_*\mu$ may  fail to be exact dimensional for almost all $\ba$. Furthermore, if ${\bf T}$ is a tuple of $2\times 2$ antidiagonal matrices, then under a mild assumption on ${\bf T}$, we can construct a compact $\sigma$-invariant set $X\subset \Sigma$ with positive topological entropy such that $$\overline{S}(\mu,{\bf T},W)\neq \underline{S}(\mu,{\bf T},W)$$ for every ergodic $\sigma$-invariant measure $\mu$ supported on $X$ with $h_\mu(\sigma)>0$, where $W$ is either the $x$-axis or the $y$-axis in $\R^2$; see Proposition \ref{pro-antidiagonal} and its proof.           This phenomenon is both intriguing and unexpected.  Nevertheless, by Theorem \ref{thm-1.2}(iii)-(iv), for any tuple ${\bf T}$ satisfying \eqref{e-norm}, if $\mu$ is fully supported and supermultiplicative (e.g., when $\mu$ is a Bernoulli product measure),  then $(P_W\pi^\ba)_*\mu$ is exact dimensional for almost all $\ba$.  We note that  the assumption of $\mu$ being fully supported can be dropped; see Remark \ref{rem-5.8}.

In addition to  Theorems \ref{thm-1.1}-\ref{thm-1.2},  we  present further results concerning the quantities $\dim_{\rm AFF}({\bf T}, W)$,  $\overline{S}(\mu,{\bf T}, W)$  and $\underline{S}(\mu,{\bf T}, W)$ in  specific cases in Section \ref{S-7}.  For instance, when $d=2$, we provide
\begin{itemize}
\item
a simple verifiable criterion for $\dim_{\rm AFF}({\bf T}, W)$ to be strictly less than \\ $\min\{1, \dim_{\rm AFF}({\bf T})\}$,
\item a verifiable criterion for $\overline{S}(\mu,{\bf T}, W)$ to be strictly less than \\
$\min\{1, \dim_{\rm LY}(\mu, {\bf T})\}$, and
    \item
    a necessary and sufficient condition for which $\overline{S}(\mu,{\bf T}, W)>\underline{S}(\mu,{\bf T}, W)$;
\end{itemize}
 see Propositions \ref{pro-planar case}-\ref{pro-planar case-measure}.

It is worth noting that when $T_i=\rho_iO_i$ for all $1\leq i\leq m$, with $0<\rho_i<1$ and $O_i$ being orthogonal, both \eqref{e-AFFW} and \eqref{e-smuw} hold for all $W\in G(d,k)$;
see Section \ref{S-9}. Therefore, in this case, there is no exceptional dimension drop for orthogonal projections of $K^\ba$ and $\pi^\ba_*\mu$ for almost all $\ba$, provided that $\rho_i<1/2$ for all $i$.

We emphasize that our results can be readily extended to general linear projections, rather than being restricted to orthogonal projections.  To illustrate this, let $L\colon \R^d \to \R^d$ be a singular linear transformation of rank $k$, and let $W=L^*(\R^d)$, where $L^*$ denotes the transpose of $L$. Then, there exists an invertible transformation $M$ on $\R^d$ such that $L=MP_W$; see Lemma~\ref{lem-simple}(ii).  Thus,  the dimensional properties of projected sets and measures under the linear projection $L$ are the same as those under the orthogonal projection $P_W$.

Rather than focusing on typical self-affine sets and measures, there are existing results in the literature concerning  orthogonal projections of specific self-similar and self-affine sets and measures along all or particular directions (see, e.g., the surveys \cite{FalconerFraserJin2015, Shmerkin2015}, and the book \cite{BSS2023}). Notably, the following result was established by Peres and Shmerkin \cite{PeresShmerkin2009}  in the plane and by Hochman and Shmerkin \cite{HochmanShmerkin2012} in  higher dimensions: Let $K\subset \R^d$ be a self-similar set generated by an IFS $\{f_i(x)=\rho_iO_ix+a_i\}_{i=1}^m$ of similitudes with dense rotations (i.e., the rotation group generated by  $O_1,\ldots, O_m$ is dense in the full rotation group $SO(d,\R)$). Suppose that  the IFS $\{f_i\}_{i=1}^m$ additionally satisfies the strong separation condition (i.e., $f_i(K)$ are pairwise disjoint).\, Then
\begin{equation}
\label{e-1.2''}
\dim_{\rm H} P_W K= \min\{\dim_{\rm H} K, k\} \quad \mbox{for all } W\in G(d,k),
\end{equation}
and the above equality also holds if $K$ is replaced by any self-similar measure associated with the IFS. Later, Farkas \cite{Farkas2016} and Falconer and Jin \cite{FalconerJin2014} showed that \eqref{e-1.2''} and its variant for self-similar measures still hold without any separation condition on the IFS. Regarding projections of self-similar measures,  Algom and Shmerkin \cite{AlgomShmerkin2024} further weakened the denseness assumption on the rotation group. Specifically,  they obtained the exact sharp condition on the rotation group under which the measure variant of \eqref{e-1.2''} holds  when $k=1$ or $d-1$. It is worth noting that for any planar self-similar set or measure, the set of  exceptional directions $W\in G(2,1)$ for which \eqref{e-1.2''} or its measure variant fails to hold is at most countable. This result was recently established by Wu \cite{Wu2025}, improving a previous result of Hochman \cite{Hochman2014} that the set of such exceptional directions has zero packing dimension.

%a significant refinement of the work of \cite{HochmanShmerkin2012} and \cite{FalconerJin2014} on the projections of self-similar measures was achieved by Algom and Shmerkin \cite{AlgomShmerkin2024}:

In \cite{FergusonJordanShmerkin2010}, Ferguson, Jordan and Shmerkin proved that under a suitable irrationality assumption,  for several classes of self-affine carpets  $K$ in the plane (including Bedford-McMullen, Gatzouras-Lalley, and Baranski carpets), it holds that
\begin{equation}
\label{e-e1'}
\dim_{\rm H} P_W K =\min\{\dim_{\rm H} K, 1\},
\end{equation}
for all $W\in G(2,1)$ except for the $x$-axis and the $y$-axis. This extends the earlier result of Peres and Shmerkin \cite{PeresShmerkin2009} on sums of Cantor sets.  Recently, for general planar diagonal affine IFSs under a suitable irrationality assumption,  a version of \eqref{e-e1'},  where $K$ is replaced by any self-affine measure, was proved for all $W\in G(2,1)$ except for the $x$-axis and the $y$-axis by B\'{a}r\'{a}ny {\it et al.}~\cite[Theorem 1.6]{BaranyKPW2023} and Py\"{o}r\"{a}l\"{a} \cite{Pyorala2024}; see also \cite{FergusonFraserSahlsten2015} for an earlier result. We note that this also holds for a special class of ergodic measures on product-like planar self-affine sets--more precisely, products of two ergodic stationary measures supported respectively on two homogeneous self-similar sets--see Hochman and Shmerkin \cite{HochmanShmerkin2012} and Bruce and Jin \cite{BruceJin2023}.   In \cite{FalconerKempton2017}, Falconer and Kempton studied  planar affine IFSs consisting of maps whose linear parts are given by matrices with strictly positive entries. They showed that the dimension of a self-affine measure $\eta$ for such an IFS is preserved for all projections with a possible exception of one direction, provided that the dimension of $\eta$ is equal to its Lyapunov dimension.   In 2019, Barany, Hochman and Rapaport \cite{BaranyHochmanRapaport2019} achieved breakthrough results on planar self-affine sets and measures.  They  proved that for a planar self-affine set $K$, if the generating IFS of $K$ satisfies the strong open set condition, and the linear parts of the IFS are strong irreducible and proximal, then the Hausdorff dimension of $K$ equals its affinity dimension, and moreover, \eqref{e-e1'} and its version for self-affine measures hold for all subspaces $W$. More recently, Rapaport \cite{Rapaport2024} obtained analogous results for every affine IFS $\{T_ix+a_i\}_{i=1}^m$ in $\R^3$ that satisfies the same assumptions as in \cite{BaranyHochmanRapaport2019}.

Let us briefly outline some strategies employed in our proofs of Theorem~\ref{thm-1.1}(iii) and Theorem~\ref{thm-1.2}(iii)--(iv). By adapting an idea of Falconer \cite{Fal88}, one can show that, under the assumption \eqref{e-norm} and for a given $W \in G(d,k)$,
\[
\dim_{\mathrm H} P_W(K^\ba) = t_0(W)
\quad\text{for }\mathcal L^{md}\text{-a.e. }\ba \in \mathbb{R}^{md},
\]
where
\[
t_0(W) = \inf \{ s \ge 0 :\ \mathcal M^s_{W,\infty} = 0 \},
\]
and
\[
\mathcal M^s_{W,\infty}
   = \inf \sum_{n=1}^\infty \varphi^s(P_W T_{I_n}),
\]
with the infimum taken over all countable covers $\{[I_n]\}_n$ of $\Sigma$ consisting of cylinder sets.

However, this approach only establishes the constancy of $\dim_{\mathrm H} P_W(K^\ba)$ for almost all $\ba$. Since the mapping
\[
I \mapsto \varphi^s(P_W T_I)
\]
from $\Sigma_*$ to $(0,\infty)$ is in general neither submultiplicative nor supermultiplicative for $s \in (0,k]$, it is difficult to mimic Falconer’s argument (see the proof of \cite[Theorem~5.4]{Fal88}) to establish the coincidence of $\dim_{\mathrm B} P_W(K^\ba)$ and $\dim_{\mathrm H} P_W(K^\ba)$, or to establish possible connections between $\mathcal M^s_{W,\infty}$ and other dynamically defined pressure functions.

To overcome this difficulty, we adopt a different approach. We begin by studying the limiting behavior of
$$\frac1n\log \varphi^s(T_{x|n}^*P_W) \quad \mbox{and}\quad  \frac{1}{n}\log \sup_{J\in \Sigma_*} \varphi^s(T_{x|n}^*P_{T_J^*W})$$
for  $x\in \Sigma$, $W\in G(d,k)$ and $s\in [0,k]$. Applying Oseledets's multiplicative ergodic theorem, we prove a key result (see Proposition~\ref{pro-key}) stating that  for each ergodic $\sigma$-invariant measure $\mu$ on $\Sigma$, there exists a measurable subset $\Sigma' \subset \Sigma$ with $\mu(\Sigma') = 1$ such that, for every $x \in \Sigma'$, the limits
\[
\lim_{n\to\infty} \frac{1}{n}\log \varphi^{s}(T_{x|n}^* P_W)
\quad \text{and} \quad
\lim_{n\to\infty} \frac{1}{n}\log \sup_{J\in\Sigma_*} \varphi^{s}\!\left(T_{x|n}^* P_{T_J^* W}\right)
\]
exist simultaneously for all $W \in G(d,k)$ and all $s \in [0,k]$. Moreover, their values depend on $W$, $s$, and $x$, and  are equal to suitable weighted sums of the Lyapunov exponents of the matrix cocycle $\Sigma \to \mathrm{Mat}_d(\mathbb{R})$, $y \mapsto T_{y_1}^*$, with respect to $\mu$; see Proposition~\ref{pro-key} for details.  In particular,  for every $x\in \Sigma'$, $W \in G(d,k)$ and  $s \in [0,k]$,
\begin{equation}
\label{e-vartxn}
\sup_{J\in \Sigma_*}\lim_{n\to\infty} \frac{1}{n}\log \varphi^{s}(T_{x|n}^* P_{T_J^*W})
=
\lim_{n\to\infty} \frac{1}{n}\log \sup_{J\in\Sigma_*} \varphi^{s}\!\left(T_{x|n}^* P_{T_J^* W}\right).
\end{equation}

Proposition~\ref{pro-key} plays a key role in the proofs of our main results. It enables us to prove Theorem \ref{thm-1.2}(iv) by using a projection analogue of \cite[Lemma 3.1]{Fal88} and  adapting the arguments in the proofs of \cite[Theorem 2.1(ii)]{FLM2023} and \cite[Theorem 4.1]{HuntKaloshin1997}.

To prove Theorem  \ref{thm-1.2}(iii), an essential step is to  show that if $\mu$ is fully supported and supermultiplicative, then for every  $s\in [0,k]$ and $W\in G(d,k)$, the first limit in \eqref{e-vartxn} is constant for $\mu$-a.e.~$x$. It turns out that it suffices to prove this  statement when $s$ is an integer. To this end, we establish a special property of such measures $\mu$ (see Lemma~\ref{lem-density}):  for every measurable set $A\subset \Sigma$  with $\mu(A)>0$, and every positive integer $j$, there exist $I\in \Sigma_*$ and a measurable subset $A'\subset A$ with $\mu(A')>0$  such that
$$
IJy\in A
$$
for all $J\in \Sigma_*$ with $|J|\leq j$ and all $y\in A'$. We then prove the desired statement by combining this special property with a certain invariance property of Oseledets filtrations, together with a contradictory argument.

 To prove Theorem \ref{thm-1.1}(iii), we need another key ingredient.  For $s\in [0,k]$ and $W\in G(d,k)$, define $\psi^s_W\colon \Sigma_*\to (0, \infty)$ by $$\psi_W^s(I)=\sup_{J\in \Sigma_*} \varphi^s(T_{I}^*P_{T_J^*W}).$$
It turns out that $\psi_W^s$ is submultiplicative (see Lemma \ref{lem-subm}). Moreover, using Proposition \ref{pro-key}, we demonstrate that the quantity $\dim_{\rm AFF}({\bf T}, W)$ corresponds to the zero point of the topological pressure function associated with the subadditive potential $\{\log \psi^s_W(\cdot|n)\}_{n=1}^\infty$ (see Proposition \ref{pro-equivalence} and Lemma \ref{lem-Affd}(ii)). Among other applications, this result allows us to construct  suitable measures $\widetilde{\mu}$ on $\Sigma$ from the equilibrium measures of these subadditive potentials, ensuring that the dimension of $(P_W\pi^{\ba})_*\widetilde{\mu}$  approximates $\dim_{\rm H}P_W(K^\ba)$ from below for almost all $\ba$.

We note that, simultaneously and independently of this work, Morris and Sert \cite{MorrisSert2025}
obtained  alternative proofs of variants of Theorems \ref{thm-1.1}(iii) and \ref{thm-1.2}(iv). Similar to Example \ref{not exact dimensional}, they also constructed examples showing that for an ergodic invariant measure $\mu$, the projection $(P_W\pi^\ba)_*\mu$ may fail to be exact dimensional for almost every $\ba$. We have recently become aware that in a paper \cite{AKPST2025+} in progress, Allen {\it et al}.~independently constructed a planar box-like self-affine set for which a certain ergodic $\sigma$-invariant measure has orthogonal projections which are not exact-dimensional.

The paper is organized as follows. In Section \ref{S-2}, we provide some preliminaries on linear algebra, subadditive thermodynamic formalism, singular value functions, and  Lyapunov dimensions. In Section \ref{S-3}, we prove Proposition \ref{pro-key} by applying Oseledets's multiplicative ergodic theorem.  In Section \ref{S-4}, we present some properties of $\psi_W^s$ and the corresponding topological pressures. In Section \ref{S-5}, we prove Theorem \ref{thm-1.2}. In Section \ref{S-6}, we prove Theorem \ref{thm-1.1}. In Section \ref{S-7}, we give some additional results  about $\dim_{\rm AFF}({\bf T}, W)$, $\overline{S}(\mu,{\bf T}, W)$ and $\underline{S}(\mu,{\bf T}, W)$ in specific  cases.
In Section \ref{S-8}, we construct a concrete example (see Example \ref{not exact dimensional}) for which $\overline{S}(\mu,{\bf T}, W)\neq \underline{S}(\mu,{\bf T}, W)$; we also give a simple criterion to check whether there exist such counter examples for a given tuple ${\bf T}$ of $2\times 2$ antidiagonal matrices. In Section \ref{S-9}, we give some final remarks.
\section{Preliminaries}
\label{S-2}

\subsection{Exterior algebra}
\label{S-2.1}

As usual in the study of matrix cocycles, we often make use of the exterior algebra generated by the $k$-alternating forms, which we denote $(\R^d)^{\wedge k}$. It is endowed with an inner product $(\cdot|\cdot)$, with the property that
\[
(v_1\wedge\cdots\wedge v_k|w_1\wedge\cdots\wedge w_k) = \det(\langle v_a,  w_b\rangle)_{1\le a,b\le k},
\]
where $\langle v_a, w_b\rangle$ is the usual inner product on $\R^d$. The norm of $v_1\wedge\cdots\wedge v_k$ is equal to the  $k$-dimensional  volume of the parallelotope formed by $v_1,\ldots, v_k$ (see, e.g.~ \cite[p.~220]{ShafarevichRemizov2013}). It follows that
\begin{equation}
\|v_1\wedge\cdots\wedge v_k\|\leq \prod_{i=1}^k \|v_k\|
\end{equation}
and
\begin{equation}
\label{e-F-3.2'}
\|v_1\wedge\cdots\wedge v_{k+p}\|\leq \|v_1\w \cdots \w v_k\|\cdot \|v_{k+1}\w \cdots\w v_{k+p}\| \mbox{ for }k, p\geq 1.
\end{equation}
Moreover, in the special case when $k=d$,
\begin{equation}
\label{e-F-3.2''}
\|v_1\wedge\cdots\wedge v_{d}\|=|\det(v_1,\ldots, v_d)|,
\end{equation}
where $(v_1,\ldots, v_d)$ stands for the $d\times d$ matrix with column vectors $v_1,\ldots, v_d$.

For a linear subspace $V$ of $\R^d$, let $V^\perp$ denote the orthogonal complement of $V$ in $\R^d$, and let $P_V:\R^d\to V$ be the orthogonal projection onto $V$. The following result also follows from the volume interpretation of the norm of exterior products
 (see, e.g.~ \cite[p.~220]{ShafarevichRemizov2013}).
\begin{lem}
\label{lem-F-3.0}
 Let $w, v_1,\ldots, v_k\in \R^d$ so that $v_1,\ldots, v_k$ are  linearly independent. Set $V=\s(v_1,\ldots, v_k)$. Then
$$
d(w,V)=\|P_{V^\perp}(w)\|=\frac{\|w \w v_1\wedge\cdots\wedge v_k\|}{\|v_1\wedge\cdots\wedge v_k\|},
$$
where $d(w, V):=\inf\{\|w-v\|\colon v\in V\}$.
\end{lem}

If $\{ v_i\}_{i=1}^d$ is an orthonormal basis of $\R^d$, then $\{v_{i_1}\wedge\cdots\wedge v_{i_k}\colon 1\leq i_1<\ldots<i_k\leq d\}$ is an orthonormal basis of $(\R^d)^{\wedge k}$.  Let ${\rm Mat}_d(\R)$ denote the set of real $d\times d$ matrices. For $A\in{\rm Mat}_d(\R)$, we recall that the $k$-fold exterior product $A^{\wedge k}$ of $A$ is defined by the condition
\[
A^{\wedge k}(v_1\wedge\cdots\wedge v_k) = A v_1\wedge\cdots \wedge A v_k.
\]

The following properties are well known (see e.g.~\cite[Chap.~3.2]{Arnold1998} for parts (i)--(iv). Part (v) follows from \eqref{e-F-3.2''}).
\begin{lem}
\label{lem-F-3.1}
Let $A, B\in {\rm Mat}_d(\R)$ and $1\leq k<d$. Then the following properties hold.
\begin{itemize}
\item[(i)] $(AB)^{\w k}=A^{\w k} B^{\w k}$, and in particular, $\| (AB)^{\w k}\|\le \| A^{\w k}\| \| B^{\w k}\|$.
\item[(ii)] $\| A^{\w k}\|=\alpha_1(A)\cdots \alpha_k(A)$, where $\alpha_1(A)\geq \cdots\geq \alpha_d(A)$ are the singular values of $A$, and in particular, $\| A^{\w k}\|\le \| A\|^k$.
\item[(iii)] $(A^*)^{\w k}=(A^{\w k})^*$, where $A^*$ stands for the transpose of $A$.
\item[(iv)] $\det(A^{\w k})=\det(A)^{\binom{d-1}{k-1}}$.
\item[(v)] If $\{ v_i\}_{i=1}^d$ is a basis of $\R^d$, then
$$
\frac{\|A v_1\w \cdots \w A v_d\|}{\|v_1\w \cdots\w v_d\|}=|\det(A)|.
$$
\end{itemize}
\end{lem}

\subsection{Angles between linear subspaces}

Here we define the (minimal) angle between two linear subspaces $V, W$ of $\R^d$.  For $x,y\in \R^d\setminus \{0\}$, let $\measuredangle (x,y)$ denote the angle between the lines $\ell_x$ and $\ell_y$, where $\ell_x$ represents the line in $\R^d$ passing through the origin and the point $x$. In this definition,  we always have  $\measuredangle (x,y)\in [0,\pi/2]$ and
$$
\sin (\measuredangle (x,y))=\frac{(\|x\|^2\|y\|^2-\langle x, y\rangle^2)^{1/2}}{\|x\|\|y\|}.
$$
 Given two linear subspaces $V$ and $W$ of $\R^d$,  the angle $\measuredangle(V, W)$ ($0\leq \measuredangle(V, W)\leq \pi/2$) between  $V$ and $W$ is defined by
$$
\sin (\measuredangle(V, W))=\inf_{x\in V\setminus\{0\},\; y\in W\setminus\{0\}}\sin (\measuredangle (x,y)).
$$

It is known that for two nontrivial linear subspaces $V$ and $W$ of $\R^d$, $V\cap W=\{0\}$ if and only if $\measuredangle(V, W)>0$
(see e.g.~\cite[Proposition 13.2.1]{Gohberg2006}).

Below we give some useful results about the angles between linear subspaces.
\begin{de}
\label{de-2.1}
Let ${\bf v}=\{v_i\}_{i=1}^d$ be an ordered basis  of  $\R^d$. Define
\begin{equation}
\label{e-F-thetabv}
\alpha({\bf v})=\inf \measuredangle\Big({\rm span}\{v_i\colon i\in I\},\; {\rm span}\{v_j\colon j\in J\}\Big),
\end{equation}
where the infimum is taken over all disjoint nonempty subsets $I,J$ of $\{1,\ldots, d\}$. We call $\alpha({\bf v})$ the smallest angle generated by ${\bf v}$.
\end{de}

\begin{lem}
\label{lem-F-3.2}
Let $\bv=\{v_i\}_{i=1}^d$ be  an ordered basis  of  $\R^d$ and let $\alpha({\bf v})$ be smallest angle generated by ${\bf v}$.
Then the following properties hold:
\begin{itemize}
\item[(i)] For $a_1,\ldots, a_d\in \R$, $$\left\|\sum_{i=1}^d a_i v_i\right\|\geq |a_j|\|v_j\|\sin (\alpha(\bv))$$ for each $1\leq j\leq d$.
\item[(ii)] If $w$ is a unit vector in $\R^d$ with $w=\sum_{i=1}^d a_{i}v_i$. Then
$$|a_j|\leq \frac{1}{\|v_j\|\sin (\alpha(\bv))}$$ for each $1\leq j\leq d$.
\item[(iii)] For $w\in \R^d$ and $I\subset \{1,\ldots, d\}$,
$$
\sin \left( \measuredangle(\s\{w\},\; \s\{v_i\colon i\in I\})\right)=\frac{\|w\wedge (\bigwedge_{i\in I}v_i)\|}{\|w\|\|\bigwedge_{i\in I}v_i\|}.
$$
\item[(iv)] $\displaystyle (\sin (\alpha(\bv)))^{d-1}\leq \frac{\|\bigwedge_{i=1}^d v_i\|}{\prod_{i=1}^d\|v_i\|}\leq d \sin (\alpha(\bv))$.
\end{itemize}
\end{lem}
\begin{proof}
We first prove (i). Without loss of generality we show that $$\left\|\sum_{i=1}^d a_i v_i\right\|\geq |a_1|\|v_1\|\sin (\alpha(\bv)).$$
 To see this, set $W=\s\{v_2,\ldots, v_d\}$. By Lemma \ref{lem-F-3.0},
$$
\left\|\sum_{i=1}^d a_i v_i\right\|\geq d(a_1v_1, W) =\|P_{W^\perp}(a_1v_1)\|=\|a_1v_1\|\sin (\alpha)=|a_1|\|v_1\|\sin(\alpha),
$$
where $\alpha$ is the angle between ${\rm span}\{v_1\}$ and $W$. Since $\alpha\geq \alpha(\bv)$,  it follows that
$\|\sum_{i=1}^d a_i v_i\|\geq |a_1|\|v_1\|\sin(\alpha(\bv))$. This proves (i).

Next we prove (ii). By (i),  we have $$1=\|w\|=\left\|\sum_{i=1}^d a_{i}v_i\right\|\geq |a_j|\|v_j\|\sin(\alpha(\bv))$$ for each $1\leq j\leq d$, from which (ii) follows.

Part (iii) is a direct consequence of Lemma \ref{lem-F-3.0}.
To see (iv), we may assume that $v_1,\ldots, v_d$ are all unit vectors. Applying (iii) repeatedly yields
$$
\|v_1\w \cdots\w v_d\|\geq \sin (\alpha(\bv))\|v_2\w\cdots\w v_d\|\geq \ldots\geq (\sin (\alpha(\bv)))^{d-1}.
$$
This proves the first inequality in (iv). To see the other inequality in (iv), notice that there exist nonempty $I, J\subset \{1,\ldots, d\}$ with $I\cap J= \emptyset$ so that
$$
\measuredangle(\s\{v_i\colon i\in I\},\; \s\{v_j\colon j\in J\})=\alpha(\bv).
$$
Hence there exists a unit vector $u\in \s\{v_i\colon i\in I\}$ so that $$\measuredangle(\s\{u\},\; \s\{v_j\colon j\in J\})=\alpha(\bv).$$
Notice that $u=\sum_{i\in I} t_i v_i$ for some $t_i\in \R$. Since $1=\|u\|\leq \sum_{i\in I}|t_i|$, there exists an element in $I$, say $i_1$, such that $|t_{i_1}|\geq \frac{1}{\#(I)}\geq \frac{1}{d}$. Note that
\begin{equation}
\label{e-F-wedge}
\bigwedge_{i\in I} v_i=\pm \frac{1}{t_{i_1}} u \w \left(\bigwedge_{i\in I\setminus \{i_1\}} v_i\right).
\end{equation}
Hence
we have
\begin{align*}
\|v_1\w \cdots \w v_d\|&\leq \left\| \left(\bigwedge_{i\in I} v_i\right)\w  \left(\bigwedge_{j\in J} v_j\right) \right\|
 \qquad (\mbox{by \eqref{e-F-3.2'}})\\
&= \frac{1}{|t_{i_1}|} \left\| u \w \left(\bigwedge_{i\in I\backslash \{i_1\}} v_i\right) \w  \left(\bigwedge_{j\in J} v_j\right) \right\|\qquad (\mbox{by \eqref{e-F-wedge}})\\
&\leq  d \left\| u \w   \left(\bigwedge_{j\in J} v_j\right) \right\|\qquad (\mbox{by \eqref{e-F-3.2'}})\\
&=d \sin (\alpha(\bv)) \|u\| \left\|\bigwedge_{j\in J} v_j \right\| \qquad \mbox{(by (iii))}\\
&\leq d \sin (\alpha(\bv)).
\end{align*}
This proves the second inequality in (iv).
\end{proof}

%\begin{lem}
%Let ${\bf v}=\{v_i\}_{i=1}^d$ be an ordered basis  of  $\R^d$. Then
%$$
%\sin (\alpha({\bf v})) \geq \frac{|\det(v_1,\ldots, v_d)|}{d\prod_{i=1}^d\|v_i\|}.
%$$
%\end{lem}
%\begin{proof}
%To be completed.
%\end{proof}

\subsection{Pivot position vectors of  linear subspaces with respect to an ordered basis}
\label{S-LS}
In this subsection, we will introduce the definition of pivot position vector of a linear subspace of $\R^d$ with respect to an ordered basis of $\R^d$.   For this purpose, let us first recall the concept of row-reduced echelon matrix. The reader is referred to \cite[Chap.~1]{HoffmanKunze1961} for more details.
\begin{de}
A $k\times d$  matrix $M$ is called a row-reduced echelon matrix  if:
\begin{itemize}
\item[(a)]  the first non-zero entry in each non-zero row of $M$ is equal to $1$;
\item[(b)] each column of $M$ which contains the leading non-zero entry of some
row has all its other entries $0$.
\item[(c)] every row of $M$  which has all its entries 0 occurs below every row which has a non-zero entry;
\item[(d)]  if rows $1, \ldots, r$ are the non-zero rows of $M$, and if the leading nonzero
entry of row $i$ occurs in column $p_i$, $i = 1,\ldots, r$, then $p_1<p_2 <\cdots<p_r$.
\end{itemize}
\end{de}

One can also describe a $k\times d$  row-reduced echelon matrix $M$ as
follows. Either every entry in $M$ is 0, or there exists a positive integer $r$,
$1\leq r\leq k$,  and $r$ positive integers $p_1,\ldots, p_r$  with $1\leq p_1<\cdots <p_r\leq d$ and
\begin{itemize}
\item[(a)] $M_{ij}=0$ for $i > r$, and $M_{ij} = 0$ if $j < p_i$.
\item[(b)] $M_{ip_j} =\delta_{ij}$ for $1\leq i\leq r$ and $1\leq j\leq r$, where $\delta_{ij}=1$ if $i=j$ and $0$ otherwise.
\end{itemize}
We call the vector $(p_1,\ldots, p_r)$ the {\it pivot position vector} of $M$.

Now let $W$ be a linear subspace of $\R^d$ with $\dim W=k$, where $1\leq k\leq d$. Let ${\bf v}=\{v_1,\ldots, v_d\}$ be an ordered basis of $\R^d$.  It is well known (see e.g. \cite[Chap.~2, Theorem~11]{HoffmanKunze1961}) that  there is a  precise one $k\times d$ row-reduced echelon matrix $M=(M_{ij})$ with rank $k$ such that
\begin{equation}
\label{e-Wspan}
W={\rm span} \left\{\sum_{j=1}^d M_{ij}v_j\colon i=1,\ldots, k\right\}.
\end{equation}
Let $(p_1,\ldots, p_k)$ be the  pivot position vector of $M$.
\begin{de}
\label{de-pivot}
Let $W$ be a linear subspace of $\R^d$ with $\dim W=k$ and let ${\bf v}=\{v_i\}_{i=1}^d$ be an ordered basis of $\R^d$. Let  $(p_1,\ldots, p_k)$  be defined as above. We call $(p_1,\ldots, p_k)$ the  pivot position vector of $W$ with respect to ${\bf v}$, and denote it as ${\bf p}(W,{\bf v})$.
\end{de}

\begin{rem}
\label{rem-pivot}
From the standard procedure of finding row-reduced echelon form, it is readily checked that the mapping $(W,{\bf v})\mapsto {\bf p}(W,{\bf v})$ is Borel measurable with respect to the natural topology on $G(d,k)\times \left\{{\bf v}=\{v_i\}_{i=1}^d\colon \det(v_1,\ldots, v_d)\neq 0\right\}$.
\end{rem}
In the following we give a useful lemma which will be used in the proofs of Proposition \ref{pro-key} and Theorem \ref{thm-1.1}(i).

\begin{lem}
\label{lem-2.5'}
Let $(p_1,\ldots, p_k)$ be the pivot position vector of $W$ with respect to ${\bf v}$, where $W$ is a  $k$-dimensional linear subspace  of $\R^d$, and ${\bf v}=\{v_i\}_{i=1}^d$ is an ordered basis  of $\R^d$.  Let $\ell\in \{1,\ldots, k\}$. Then the following properties hold.
\begin{itemize}
\item[(i)]  The linear subspace $W^{\wedge \ell}$ of $(\R^d)^{\wedge \ell}$ satisfies
\begin{equation}
\label{e-eWl}
\begin{split}
W^{\wedge \ell}\subset\; & {\rm span}\left\{v_{i_1}\wedge \cdots \wedge v_{i_\ell}\colon 1\leq i_1<\cdots<i_\ell\leq d \mbox{ and }\right.\\
&\qquad \qquad  \qquad \qquad \qquad \left.  i_m\geq p_m \mbox{ for all }1\leq m\leq \ell \right\},
\end{split}
\end{equation}
and $W^{\wedge \ell}\not\subset H$, where $H$ is a linear subspace of $(\R^d)^{\wedge \ell}$ defined by
\begin{equation}
\label{e-H*}
\begin{split}
H=\; & {\rm span}\left\{v_{i_1}\wedge \cdots \wedge v_{i_\ell}\colon 1\leq i_1<\cdots<i_\ell\leq d \mbox{ and }\right.\\
&\qquad \qquad  \qquad \qquad \qquad \left. (i_1,\ldots, i_\ell)\neq (p_1,\ldots, p_\ell) \right\}.
\end{split}
\end{equation}
\item[(ii)] Let $V\in G(d,k)$.  Suppose that   $V^{\wedge \ell}\not\subset H$.
Let $(q_1,\ldots, q_k)$ be the pivot position vector of $V$ with respect to ${\bf v}$. Then $q_i\leq p_i$ for all $1\leq i\leq \ell$.
\end{itemize}
\end{lem}
\begin{proof}
Let $M=(M_{ij})$ be the unique $k\times d$ row-reduced echelon matrix with rank $k$ such that \eqref{e-Wspan} holds. Write
$$
w_i=\sum_{j=1}^d M_{ij}v_j,\qquad i=1,\ldots, k.
$$
Since $(p_1,\ldots, p_k)$ is the pivot position vector of $M$,
\begin{equation}
\label{e-wi}
w_i=v_{p_i}+\sum_{j=p_i+1}^d M_{ij}v_j,\qquad i=1,\ldots, k.
\end{equation}
Recall that $W=\s\{w_1,\ldots, w_k\}$. It follows that
$$
W^{\w \ell}={\rm span}\left\{w_{j_1}\wedge \cdots \wedge w_{j_\ell}\colon 1\leq j_1<\cdots<j_\ell\leq k\right\}.
$$

By \eqref{e-wi}, for each $1\leq j_1<\cdots<j_\ell\leq k$,
\begin{align*}w_{j_1}\wedge \cdots \wedge w_{j_\ell}&\in {\rm span}\left\{v_{i_1}\wedge \cdots \wedge v_{i_\ell}\colon p_{j_m}\leq  i_m\leq d \mbox{ for all }1\leq m\leq \ell \right\}\\
&\subset {\rm span}\left\{v_{i_1}\wedge \cdots \wedge v_{i_\ell}\colon 1\leq i_1<\cdots<i_\ell\leq d \mbox{ and }\right.\\
&\qquad \qquad  \qquad \qquad \qquad \left.  i_m\geq p_{m} \mbox{ for all }1\leq m\leq \ell \right\},
\end{align*}
where in the last inclusion we use the fact that $j_m\geq m$ and so $p_{j_m}\geq p_m$ for each $1\leq m\leq \ell$.
This proves \eqref{e-eWl}. Again by \eqref{e-wi}, we see that the term $v_{p_1}\wedge \cdots \w v_{p_\ell}$ appears in the linear expansion of
$w_1\w \cdots \w w_\ell$ relative to the basis
$$
{\bv}^{(\ell)}:=\left\{v_{i_1}\wedge \cdots \wedge v_{i_\ell}\colon 1\leq i_1<\cdots<i_\ell\leq d\right\}
$$
of $(\R^d)^{\w \ell}$.
 It follows that $w_1\w \cdots \w w_\ell\not\in H$, where $H$ is defined as in \eqref{e-H*}.
Hence $W^{\w \ell}\not\subset H$. This completes the proof of part (i).

To see part (ii), applying part (i) to $V^{\w \ell}$ gives
\begin{equation}
\label{e-eWl*}
\begin{split}
V^{\wedge \ell}\subset\; & {\rm span}\left\{v_{i_1}\wedge \cdots \wedge v_{i_\ell}\colon 1\leq i_1<\cdots<i_\ell\leq d \mbox{ and }\right.\\
&\qquad \qquad  \qquad \qquad \qquad \left.  i_m\geq q_m \mbox{ for all }1\leq m\leq \ell \right\}.
\end{split}
\end{equation}
Meanwhile since $V^{\wedge \ell}\not\subset H$, it follows that there exists $u\in V^{\w \ell}$ such that a nonzero scalar multiple of
$v_{p_1}\w \cdots \w v_{p_\ell}$ appears in the linear expansion of $u$ relative to the basis
${\bv}^{(\ell)}$.
Therefore, by \eqref{e-eWl*}, we have $p_m\geq q_m$ for all $1\leq m\leq \ell$.
\end{proof}
\subsection{Variational principle for subadditive pressure}
\label{S-subadditive}

 Let $(\Sigma, \sigma)$ be the full shift over a finite alphabet $\{1,\ldots, m\}$.    A sequence $\mathcal F=\{\log f_n\}_{n=1}^\infty$ of functions on $\Sigma$ is said to be a {\em subadditive potential} if
$$0\leq f_{n+m}(x)\le f_n(x)f_m(\sigma^n x)$$ for all $x\in \Sigma$ and $n,m\in\N$. The {\em topological pressure} of a subadditive potential $\mathcal F$ is defined as
$$P(\sigma, \mathcal F)= \lim_{n\to\infty} \frac{1}{n}\log\left( \sum_{I\in \Sigma_n} \sup_{y\in [I]} f_n(y) \right).
$$
The limit exists by using a standard subadditivity argument.

 For a $\sigma$-invariant measure $\mu$ on $\Sigma$,  let $h_\mu(\sigma)$ denote the measure-theoretic entropy of $\mu$ (cf.~\cite{Walters1982}). Let $\mathcal E(\Sigma, \sigma)$ denote the collection of ergodic $\sigma$-invariant measures on $\Sigma$. Our proof of  Theorem \ref{thm-1.1}  depends on the following  variational principle for subadditive potentials. Although in \cite{CFH2008} this was proved for potentials on an arbitrary continuous dynamical system on a compact space, we state it only for fullshifts.
\begin{thm}[{\cite[Theorem 1.1]{CFH2008}}] \label{thm:subadditive-VP}
Let  $\mathcal{F}=\{\log f_n\}$ be a subadditive potential on $\Sigma$. Assume that  $f_n$ is continuous on $\Sigma$ for each $n$. Then
\begin{equation}
\label{variational-principle}
P(\sigma, \mathcal{F}) = \sup\left\{ h_\mu(\sigma) + \lim_{n\to\infty} \frac{1}{n} \int \log(f_n(x)) d\mu(x)\colon \mu\in \mathcal E(\Sigma, \sigma) \right\}.
\end{equation}
\end{thm}
Particular cases of the above, under stronger assumptions on the potentials, were previously obtained by many authors, see for example \cite{Falconer1988, Barreira1996, Kaenmaki2004, Mummert2006} and references therein.

Measures that achieve the supremum in \eqref{variational-principle}  are called {\em ergodic equilibrium measures} for the potential $\mathcal F$. The existence of ergodic equilibrium measures follows from the upper semi-continuity of the entropy map $\mu\mapsto h_\mu(\sigma)$ for fullshifts (see, e.g., \cite[Propostion 3.5]{Feng2011} and the remark therein).

\subsection{Singular value functions}

 Recall that for $A\in {\rm Mat}_d(\R)$, $\alpha_1(A)\geq\cdots \geq \alpha_d(A)$ stand for the singular values of $A$. It is well known that $\alpha_i(A)=\alpha_i(A^*)$ for each $i$. For $s\geq 0$, let $\varphi^s$ denote the singular value function; see  \eqref{e-singular} for the definition. Here we collect several lemmas about $\varphi^s$.

\begin{lem}[\cite{Fal88}]\
\label{lem-singinequality}
\begin{itemize}
\item[(i)]
$\varphi^s (AB) \leq \varphi^s (A) \varphi^s (B)$ for all $A, B \in {\rm Mat}_d(\R)$ and $s \geq 0$.
\item[(ii)] $\varphi^s (A) (\alpha_d(A))^t \leq \varphi^{s+t} (A) \leq \varphi^s (A) \| A \|^t$ for all $A \in {\rm Mat}_d(\R)$ and $s, t \geq 0$.
\end{itemize}
\end{lem}

\begin{lem}
\label{lem-phis}
Let $A\in {\rm Mat}_d(\R)$ and $s\in [0,d]$. Set $k=\lfloor s\rfloor$. Then
$$
\varphi^s(A)=\|A^{\w k}\|^{k+1-s}\|A^{\w (k+1)}\|^{s-k}.
$$
Moreover, $\varphi^s(A)=\varphi^s(A^*)$.

\end{lem}
\begin{proof}
It follows directly from the definition of $\varphi^s(A)$ (see \eqref{e-singular}) and Lemma \ref{lem-F-3.1}(ii).
\end{proof}

\begin{lem}
\label{lem-singular1}
Let $A\in {\rm GL}_d(\R)$, $W\in G(d,k)$ and $s\geq 0$. Then
\begin{itemize}
\item[(i)] $\alpha_k(AP_W)\geq \alpha_d(A)$ and $\alpha_{k+1}(AP_W)=0$.
\item[(ii)] If $s>k$, then $\varphi^s(AP_W)=0$.
\item[(iii)] If $s\in [0,k]$, then $(\alpha_d(A))^s\leq \varphi^s(AP_W)\leq \varphi^s(A)$.
\end{itemize}
\end{lem}
\begin{proof}
Clearly (ii) and (iii) follow from (i) and the definition of $\varphi^s$.

To prove (i), we need the following analog of the Courant-Fisher theorem for singular values:  for every $M\in {\rm Mat}_{d}(\R)$ and $i\in \{1,\ldots, d\}$,
 \begin{equation}
 \label{e-alpha*}
 \alpha_i(M)=\max_{\dim(V)=i}\min_{x\in V \atop{\|x\|=1}} \|Mx\|,
 \end{equation}
 where in this expression, $V$ is a subspace of $\R^d$; see e.g.~\cite[Theorem 8.6.1]{Golub2013} or \cite[Theorem 3.1.2]{HoJo91}.

 Taking $i=k$, $M=AP_W$ and $V=W$ in \eqref{e-alpha*} gives
 $$
 \alpha_k(AP_W)\geq \min_{x\in W \atop{\|x\|=1}} \|AP_Wx\|=\min_{x\in W \atop{\|x\|=1}} \|Ax\|\geq \min_{x\in \R^d \atop{\|x\|=1}}\|Ax\|=\alpha_d(A).
 $$

 Meanwhile if $V$ is a subspace with $\dim (V)=k+1$, then $\dim(V)+\dim(W^\perp)>d$, and consequently, $V\cap W^{\perp}\neq \{0\}$.
 Hence for each subspace $V$  with $\dim (V)=k+1$,
  $$
  \min_{x\in V \atop{\|x\|=1}} \|AP_Wx\|\leq \min_{x\in V\cap W^\perp \atop{\|x\|=1}} \|AP_Wx\|=0.
  $$
 Then taking   $i=k+1$ and $M=AP_W$  in \eqref{e-alpha*} gives $\alpha_{k+1}(AP_W)=0$.
\end{proof}

\subsection{Lyapunov dimension}
\label{S-Lyapunov}
Let ${\bf T}=(T_1,\ldots, T_m)$ be a tuple of $d\times d$ invertible real matrices with $\|T_i\|<1$ for $1\leq i\leq m$, and let $\mu$ be an ergodic $\sigma$-invariant measure on $\Sigma=\{1,\ldots, m\}^\N$.
\begin{de}
\label{e-Lambda}
For $i\in \{1,\ldots,d\}$, the $i$-th Lyapunov exponent of $\mu$ with respect to ${\bf T}$ is defined by
\begin{equation}
\label{e-Lambdai}
\Lambda_i=\lim_{n\to \infty} \frac{1}{n} \int \log(\alpha_i(T_{x|n})) \; d\mu(x),
\end{equation}
where $\alpha_i(A)$ stands for the $i$-th singular value of $A$.
\end{de}
The existence of the limit in defining $\Lambda_i$ follows from \cite[Theorem 3.3.3]{Arnold1998}. Following \cite{JPS07}, below we present the definition of Lyapunov dimension of $\mu$ with respect to ${\bf T}$.
\begin{de}
\label{de-Lyapunov}
The Lyapunov dimension of $\mu$ with respect to ${\bf T}$, written as $\dim_{\rm LY} (\mu, {\bf T})$, is the unique non-negative value $s$ for which
$$
h_\mu(\sigma)+\mathcal G^s_*(\mu)=0,
$$
where ${\mathcal G}^s_*(\mu):=\lim_{n\to \infty} (1/n)\int \log (\varphi^s(T_{x|n}))\; d\mu(x)$, and $\varphi^s$ is the singular value function defined as in \eqref{e-singular}.
\end{de}

It follows from the definition of $\varphi^s$ and Definition \ref{e-Lambda} that
\begin{equation}
\label{e-Gs}
\mathcal G^s_*(\mu)=\left\{\begin{array}{ll}
\Lambda_1+\cdots+\Lambda_{\lfloor s\rfloor}+(s-\lfloor s \rfloor) \Lambda_{\lfloor s \rfloor +1} &\mbox{ if } s< d,\\
[5pt]
\frac{s}{d} (\Lambda_1+\cdots+\Lambda_d)&\mbox{ if } s\geq  d.
\end{array}
\right.
\end{equation}

\section{Oseledets's multiplicative ergodic theorem and a key proposition}
\label{S-3}

Throughout this section, let ${\bf T}=(T_1,\ldots, T_m)$ be a fixed tuple of $d\times d$ invertible real matrices, and let  $\mu$ be an ergodic $\sigma$-invariant measure on $\Sigma$. The main result of this section is Proposition \ref{pro-key}, which describes the asymptotic properties of
$$\frac1n\log \varphi^s(T_{x|n}^*P_W) \quad \mbox{and}\quad  \frac{1}{n}\log \sup_{J\in \Sigma_*} \varphi^s(T_{x|n}^*P_{T_J^*W})$$ for $\mu$-a.e.~$x\in \Sigma$, uniformly in $W$ and $s$.  It plays a key role in the proofs of Theorems \ref{thm-1.1} and \ref{thm-1.2}.

In order to state and prove this result, we require the following theorem, in which part (6) is due to the Shannon-McMillan-Breiman theorem (see e.g. \cite[p.~261]{Petersen1989}), while the other parts  are due to  Oseledets's multiplicative ergodic theorem (see e.g.~ \cite[Theorem 4.1]{FLQ2010} and \cite[Theorem 5.3.1]{Arnold1998}).

\begin{thm} \label{thm:oseledets}
 There exists a measurable set $\Sigma'\subset \Sigma$ with $\sigma(\Sigma')\subset \Sigma'$ and $\mu(\Sigma')=1$, such that there are an integer $r\in \{1,\ldots, d\}$, real numbers $\lambda_1>\cdots>\lambda_r$, and positive integers  $d_1,\ldots, d_r$ with $\sum_{i=1}^r d_i=d$ so that for every $x=(x_n)_{n=1}^\infty\in \Sigma'$, there is a splitting $\R^d=\bigoplus_{i=1}^r E_i(x)$ which satisfies  the following properties.
\begin{enumerate}
\item $\dim E_i(x)=d_i$;
\item $T_{x_1}^*E_i(x)= E_i(\sigma x)$;
\item For all $v\in E_i(x)\setminus\{0\}$,
\[
\lim_{n\to\infty} \frac{\log \|T_{x|n}^*v\|}{n}=\lambda_i.
\]
\item $\displaystyle \lim_{n\to\infty} \frac{1}{n}\log |\det(T_{x|n}^*)|=\sum_{i=1}^rd_i\lambda_i$;
\item The mappings $x\mapsto E_i(x)$ are measurable on $\Sigma'$.
\item $\lim_{n\to \infty} \frac{1}{n}\log \mu([x|n])=-h_\mu(\sigma)$.
\end{enumerate}
Moreover, let $$\Lambda_1\geq \Lambda_2\geq \cdots\geq \Lambda_d$$ be the list of the $\lambda_i$ where $\lambda_i$ appears $d_i$ times.  Choose a measurable ordered basis ${\bf v}(x)=\{v_1(x),\ldots, v_d(x)\}$ adapted to the splitting $\bigoplus_{i=1}^r E_i(x)$, $x \in \Sigma'$, i.e.~such that the first $d_1$ vectors are in $E_1(x)$,\ldots, the last $d_r$ vectors in $E_r(x)$. Then for each $\ell\in \{1,\ldots, d-1\}$,  $1\leq i_1<\cdots<i_\ell\leq d$ and $x\in \Sigma'$, the following properties hold.
\begin{itemize}
\item[(7)] $
\lim_{n\to \infty}\frac{1}{n}\log  \left\|(T_{x|n}^*)^{\wedge \ell}(v_{i_1}(x)\wedge\cdots\wedge v_{i_\ell}(x))\right\|=\Lambda_{i_1}+\cdots+\Lambda_{i_\ell}.
$
\item[(8)]
$
\lim_{n\to \infty}\frac{1}{n}\log \sin(\alpha_n(x))=0,
$
where $\alpha_n(x)$ denotes the smallest angle generated by the basis
$$\left\{(T_{x|n}^*)^{\wedge \ell}(v_{i_1}(x)\wedge\cdots\wedge v_{i_\ell}(x))\colon 1\leq i_1<\cdots<i_\ell\leq d\right\}
$$
of $(\R^d)^{\wedge \ell}$; see Definition \ref{de-2.1}.
\end{itemize}
\end{thm}

\begin{rem}
\label{rem-3.2}
\begin{itemize}
\item[(i)]
  The numbers $\Lambda_1,\ldots, \Lambda_d$ are called the Lyapunov exponents (counting multiplicity) of the matrix cocycle $x\mapsto T_{x_1}^*$ with respect to $\mu$. They can be alternatively defined by \eqref{e-Lambda}; see e.g.~\cite[Theorem 3.3.3]{Arnold1998} for a proof.
\item[(ii)] Set $V_i(x)=\bigoplus_{j=i+1}^r E_j(x)$ for $x\in \Sigma'$ and $i=0,1,\ldots, r$. Then
$$
\R^d=V_0(x)\supsetneqq V_1(x)\supsetneqq \cdots \supsetneqq  V_r(x)=\{0\},
$$
which is called the associated Oseledets filtration with the matrix cocycle $x\mapsto T_{x_1}^*$  and $\mu$.  Moreover,
$$\lim_{n\to \infty}\frac1n\log \|T_{x|n}^*v\|=\lambda_{i+1}$$ for all $i\in \{0,1,\ldots, r-1\}$ and $v\in V_{i}(x)\backslash V_{i+1}(x)$. By Theorem \ref{thm:oseledets}(2),  $T_{x_1}^* V_i(x)=V_i(\sigma x)$ for every $x\in \Sigma'$.
\item[(iii)] The Lyapunov exponents $\widetilde{\Lambda}_i$, for $1\leq i\leq \binom{d}{\ell}$,  of  the  cocycle $x\mapsto (T_{x_1}^*)^{\w \ell}$ with respect to $\mu$, are simply the rearrangement (in  decreasing order) of  $\Lambda_{i_1}+\cdots+\Lambda_{i_\ell}$, where $1\leq i_1<\cdots <i_\ell\leq d$. See e.g.~\cite[Theorem 5.3.1]{Arnold1998}.
\end{itemize}
\end{rem}

It appears that part (8) of Theorem \ref{thm:oseledets} is not explicitly stated in Oseledets's multiplicative ergodic theorem; therefore  we provide a proof below.

\begin{proof}[Proof of Theorem \ref{thm:oseledets}(8)] Let $x\in \Sigma'$ and $\ell\in \{1,\ldots, d-1\}$. Set
$$
{\mathcal I}_\ell:=\left\{(i_1,\ldots, i_\ell)\in {\Bbb N}^\ell\colon 1\leq i_1<\cdots<i_\ell\leq d\right\}.
$$
For  $(i_1,\ldots, i_\ell)\in \mathcal I_\ell$, write
\begin{align*}
v^{(0)}_{i_1,\ldots,i_\ell}&:=v_{i_1}(x)\wedge\ldots\wedge v_{i_\ell}(x)\quad \mbox{ and }\\
v^{(n)}_{i_1,\ldots,i_\ell}&:=(T_{x|n}^*)^{\wedge \ell}(v_{i_1}(x)\wedge\cdots\wedge v_{i_\ell}(x))\quad \mbox{ for } n\geq 1.
\end{align*}
Clearly, for each $n\geq 0$,
$
\left\{
v^{(n)}_{i_1,\ldots,i_\ell}
\right\}_{(i_1,\ldots, i_\ell)\in {\mathcal I}_{\ell}}$ is a basis of $(\R^d)^{\w \ell}$. By Lemma \ref{lem-F-3.2}(iv),
\begin{equation}
\label{e-t3.1}
\sin(\alpha_n(x))\geq \binom{d}{\ell}^{-1}\cdot \frac{\left\| \bigwedge_{(i_1,\ldots, i_\ell)\in {\mathcal I}_{\ell}}v^{(n)}_{i_1,\ldots,i_\ell} \right\|}{\prod_{(i_1,\ldots, i_\ell)\in {\mathcal I}_{\ell}} \left\|v^{(n)}_{i_1,\ldots,i_\ell}\right\|} \qquad \mbox{ for }n\geq 0.
\end{equation}
Observe that
\begin{equation*}
\begin{split}
\left\| \bigwedge_{(i_1,\ldots, i_\ell)\in {\mathcal I}_{\ell}}v^{(n)}_{i_1,\ldots,i_\ell} \right\|&=\left\| \bigwedge_{(i_1,\ldots, i_\ell)\in {\mathcal I}_{\ell}} (T_{x|n}^*)^{\w \ell} v^{(0)}_{i_1,\ldots,i_\ell} \right\|\\
&=\left|\det\left((T_{x|n}^*)^{\w \ell}\right)\right|\left\| \bigwedge_{(i_1,\ldots, i_\ell)\in {\mathcal I}_{\ell}}  v^{(0)}_{i_1,\ldots,i_\ell} \right\|\quad \mbox{(by Lemma \ref{lem-F-3.1}(v))}\\
&=\left|\det\left(T_{x|n}^*\right)\right|^{\binom{d-1}{\ell-1}}c(x) \quad\qquad \qquad \qquad  \mbox{(by Lemma \ref{lem-F-3.1}(iv))},
\end{split}
\end{equation*}
where $c(x):=\left\| \bigwedge_{(i_1,\ldots, i_\ell)\in {\mathcal I}_{\ell}}  v^{(0)}_{i_1,\ldots,i_\ell} \right\|$ is positive and independent of $n$. It follows from (4) that
\begin{equation}
\label{e-t3.2}
\begin{split}
\lim_{n\to \infty} \frac1n\log \left\| \bigwedge_{(i_1,\ldots, i_\ell)\in {\mathcal I}_{\ell}}v^{(n)}_{i_1,\ldots,i_\ell} \right\|&=\binom{d-1}{\ell-1}\lim_{n\to \infty} \frac1n \log \left|\det\left(T_{x|n}^*\right)\right|\\
&=\binom{d-1}{\ell-1}\sum_{i=1}^rd_i\lambda_i\\
&=\binom{d-1}{\ell-1}\sum_{i=1}^d\Lambda_i.
\end{split}
\end{equation}
Meanwhile, by (7),
\begin{equation}
\label{e-t3.3}
\begin{split}
\lim_{n\to \infty} \frac{1}{n} \log \prod_{(i_1,\ldots, i_\ell)\in {\mathcal I}_{\ell}} \left\|v^{(n)}_{i_1,\ldots,i_\ell}\right\|&=\sum_{(i_1,\ldots, i_\ell)\in {\mathcal I}_{\ell}}\lim_{n\to \infty} \frac{1}{n}\log\left\|v^{(n)}_{i_1,\ldots,i_\ell}\right\|\\
&=\sum_{(i_1,\ldots, i_\ell)\in {\mathcal I}_{\ell}}(\Lambda_{i_1}+\cdots+\Lambda_{i_\ell})\\
&=\binom{d-1}{\ell-1}\sum_{i=1}^d\Lambda_i,
\end{split}
\end{equation}
where in the last equality we use the simple fact that for each $j\in \{1,\ldots, d\}$,
$$
\#\left\{(i_1,\ldots, i_\ell)\in {\mathcal I}_\ell\colon j\in \{i_1,\ldots, i_\ell\} \right\}=\binom{d-1}{\ell-1}.
$$
Combining \eqref{e-t3.1}, \eqref{e-t3.2} and \eqref{e-t3.3} yields that
$$
\liminf_{n\to \infty}\frac1n \log \sin(\alpha_n(x))\geq 0.
$$
Since $\alpha_n(x)\in (0,\pi/2]$, this implies (8).
\end{proof}

In the remaining part of this subsection, let $\Sigma'$, $r$, $\Lambda_1,\ldots, \Lambda_d$, $\bigoplus_{i=1}^r E_i(x)$ ($x\in \Sigma'$) be given as in Theorem \ref{thm:oseledets}, and also let  ${\bf v}(x)=\{v_1(x),\ldots, v_d(x)\}$ be a measurable ordered basis adapted to the splitting $\bigoplus_{i=1}^r E_i(x)$, $x \in \Sigma'$.

Recall the definition of the pivot position vector for a linear subspace with respect to an ordered basis; see  Definition \ref{de-pivot}.  Now we are ready to state the main result of this section.
\begin{pro}
\label{pro-key}
Let $k\in \{1,\ldots, d-1\}$, $W\in G(d,k)$ and $s\in [0,k]$. For $x\in \Sigma'$, let $\left(p_1(W,x),\ldots, p_k(W,x)\right)$ denote the pivot position vector of $W$ with respect to the ordered basis ${\bf v}(x)=\{v_i(x)\}_{i=1}^d$. Then for every $x\in \Sigma'$,
 \begin{equation}
 \label{e-sw}
 \lim_{n\to \infty}\frac{1}{n}\log \varphi^s(T_{x|n}^*P_W)=\sum_{j=1}^{\lfloor s \rfloor} \Lambda_{p_j(W,x)}+(s-\lfloor s \rfloor)\Lambda_{p_{\lfloor s \rfloor+1}(W,x)},
 \end{equation}
 and
 \begin{equation}
 \label{e-supw}
 \lim_{n\to \infty}\frac{1}{n}\log \sup_{J\in \Sigma_*} \varphi^s(T_{x|n}^*P_{T_J^*W})=\sup_{J\in \Sigma_*}\sum_{j=1}^{\lfloor s \rfloor} \Lambda_{p_j(T^*_JW,x)}+(s-\lfloor s \rfloor)\Lambda_{p_{\lfloor s \rfloor+1}(T^*_JW,x)}.
 \end{equation}
  \end{pro}
\begin{proof}
Fix $x\in \Sigma'$. To simplify our notation, we write ${\bf v}=\{v_i\}_{i=1}^d$ for the ordered basis $\{v_1(x),\ldots, v_d(x)\}$ of $\R^d$. For $W\in G(d,k)$, let $(p_1(W),\ldots, p_k(W))$ denote the pivot position vector of $W$ with respect to the ordered basis ${\bf v}$.

Let $\epsilon>0$. We will show that for each $\ell \in \{1,\ldots, k\}$, there exist $N_\ell\in \N$ and $C_\ell>0$ such that for all $n\geq N_\ell$,
\begin{equation}
\label{e-upper1}
 \left\|(T_{x|n}^*P_W)^{\wedge \ell}\right\| \leq C_\ell \exp\Big(n\Big(\epsilon+\sum_{j=1}^\ell \Lambda_{p_j(W)}\Big)\Big) \quad \mbox{ for all }W\in G(d,k),
 \end{equation}
and
\begin{equation}
\label{e-lower1}
 \left\|(T_{x|n}^*P_W)^{\wedge \ell}\right\| \geq D_{\ell,W} \exp\Big(n\Big(-2\epsilon+\sum_{j=1}^\ell \Lambda_{p_j(W)}\Big)\Big) \quad \mbox{ for all }W\in G(d,k),
\end{equation}
where $D_{\ell,W}>0$ depends on $\ell$ and $W$, and is independent of $n$.
To prove the above inequalities, fix $\ell \in \{1,\ldots, k\}$. For $n\in \N$, let $\alpha_{n}^{(\ell)}$ denote the smallest angle generated by the basis
$$\left\{(T_{x|n}^*)^{\wedge \ell}(v_{i_1}\wedge\cdots\wedge v_{i_\ell})\colon 1\leq i_1<\cdots<i_\ell\leq d\right\}
$$
of $(\R^d)^{\wedge \ell}$; see Definition \ref{de-2.1}.  Set
$$
{\mathcal I}_\ell:=\left\{(i_1,\ldots, i_\ell)\in {\Bbb N}^\ell\colon 1\leq i_1<\cdots<i_\ell\leq d\right\}.
$$
For $n\geq 0$ and $(i_1,\ldots, i_\ell)\in \mathcal I_\ell$, write
$$
v^{(n)}_{i_1,\ldots,i_\ell}:=(T_{x|n}^*)^{\wedge \ell}(v_{i_1}\wedge\cdots\wedge v_{i_\ell}).
$$
By Theorem \ref{thm:oseledets}, there exists $N_\ell\in \N$ such that for all $n\geq N_\ell$ and $(i_1,\ldots, i_\ell)\in \mathcal I_\ell$,
\begin{equation}
\label{e-3.1a}
\left|\frac{1}{n}\log \|v^{(n)}_{i_1,\ldots,i_\ell}\|-\sum_{j=1}^\ell \Lambda_{i_j}\right|<\epsilon \quad\mbox{ and }\quad \sin \left(\alpha_{n}^{(\ell)}\right)>e^{-n\epsilon}.
\end{equation}

 For $W\in G(d,k)$, write
 $$
 \mathcal I_{\ell,W}:=\left\{(i_1,\ldots, i_\ell)\in \mathcal I_\ell\colon  i_j\geq p_j(W) \mbox{ for }1\leq j\leq \ell\right\}.
 $$
 By Lemma \ref{lem-2.5'}(i),
 \begin{equation}
 \label{e-3.2b}
 W^{\wedge \ell}\subset {\rm span}\left(\left\{v_{i_1}\wedge \cdots\wedge v_{i_\ell}\colon (i_1,\ldots,i_\ell)\in \mathcal I_{\ell,W}\right\}\right)
 \end{equation}
 for each $W\in G(d,k)$.  In what follows we estimate the growth rate of $\|(T_{x|n}^*P_W)^{\wedge \ell}\|$. It is readily checked that
 \begin{equation}
 \label{e-txn}
 (T_{x|n}^*P_W)^{\wedge \ell}=(T_{x|n}^*)^{\wedge \ell}(P_W)^{\wedge \ell}=(T_{x|n}^*)^{\wedge \ell} P_{W^{\wedge \ell}}.
 \end{equation}

 Now let $W\in G(d,k)$ and let $u$ be a unit vector in $W^{\wedge \ell}$. By  \eqref{e-3.2b}, $u$ can be expanded as
 $$
 u=\sum_{(i_1,\ldots, i_\ell)\in \mathcal I_{\ell, W}} a_{i_1,\ldots, i_\ell}\; v_{i_1}\wedge \cdots\wedge v_{i_\ell}
 $$
 with $a_{i_1,\ldots, i_\ell}\in \R$.
 By Lemma \ref{lem-F-3.2}(ii),
\begin{equation}
\label{e-aa}
 |a_{i_1,\ldots, i_\ell}|\leq \frac{1}{\|v_{i_1}\wedge \cdots\wedge v_{i_\ell}\|\sin\left(\alpha_{0}^{(\ell)}\right)}
 \end{equation}
 for each $(i_1,\ldots, i_\ell)\in \mathcal I_{\ell,W}$. It follows that for all $n\geq N_\ell$,
 \begin{align*}
 \|(T_{x|n}^*)^{\wedge \ell}u\|&=\Big\|\sum_{(i_1,\ldots, i_\ell)\in \mathcal I_{\ell,W}} a_{i_1,\ldots, i_\ell}\; v_{i_1,\ldots,i_\ell}^{(n)}\Big\|\\
 &\leq \sum_{(i_1,\ldots, i_\ell)\in \mathcal I_{\ell,W}} |a_{i_1,\ldots, i_\ell}|\; \left\|v_{i_1,\ldots,i_\ell}^{(n)}\right\|\\
 &\leq \sum_{(i_1,\ldots, i_\ell)\in \mathcal I_{\ell,W}} \frac{ \exp\Big(n\Big(\epsilon+\sum_{j=1}^\ell \Lambda_{i_j}\Big)\Big)}{\|v_{i_1}\wedge \cdots\wedge v_{i_\ell}\|\sin\left(\alpha_{0}^{(\ell)}\right)}\qquad \mbox{(by \eqref{e-3.1a} and \eqref{e-aa})}\\
 &\leq C_\ell\exp\left(n\left(\epsilon+\sum_{j=1}^\ell \Lambda_{p_j(W)}\right)\right),
 \end{align*}
 where
 $$
 C_\ell:=
 \binom{d}{\ell}
 \max_{ (i_1,\ldots, i_\ell)\in {\mathcal I_\ell} } \frac{1}{ \| v_{i_1}\wedge \cdots\wedge v_{i_\ell}\| \sin\left(\alpha_0^{(\ell)}\right)}.
 $$
 Since $u$ is an arbitrarily taken unit vector in $W^{\wedge\ell}$,  this proves \eqref{e-upper1}.
  %Hence for $n\geq N$,
%\begin{align}
 %\left\|(T_{x|n}^*P_W)^{\wedge \ell}\right\|&=\left\|(T_{x|n}^*)^{\wedge \ell} P_{W^{\wedge \ell}}\right\| \nonumber\\
 %&=\sup_{u\in W^{\wedge \ell} \colon \|u\|=1} \left\|(T_{x|n}^*)^{\wedge \ell}u\right\|\nonumber\\
 %&\leq C\exp\Big(n\Big(\epsilon+\sum_{j=1}^\ell \Lambda_{p_j(W)}\Big)\Big). \label{e-upper}
%\end{align}
%Since $N$ and $C$ is independent of $W$, it follows that for $n\geq N$,
%\begin{equation}
%\label{e-upper'}
%\sup_{J\in \Sigma_*}\left\|(T_{x|n}^*P_{T^*_JW})^{\wedge \ell}\right\| \leq C\exp\Big(n\Big(\epsilon+\sup_{J\in \Sigma_*}\sum_{j=1}^\ell \Lambda_{p_j(T^*_JW)}\Big)\Big).
%\end{equation}

 To obtain a lower bound of  $\|(T_{x|n}^*P_W)^{\wedge \ell}\|$, let $M=(M_{i,j})$ be the unique $k\times d$ row-reduced echelon matrix $M=(M_{ij})$ with rank $k$ such that
$$W={\rm span} \left\{\sum_{j=1}^d M_{ij}v_j\colon i=1,\ldots, k\right\};
$$
see Section \ref{S-LS} for the details. Set $$w_i=\sum_{j=1}^d M_{ij}v_j\qquad \mbox{  for }i=1,\ldots, \ell.$$
Notice that $w_1,\ldots, w_\ell\in W$ and they are linearly independent. Hence, $w_1\wedge\ldots\wedge w_\ell$ is a nonzero element of $W^{\wedge \ell}$.
Since the pivot position vector of $M$ is equal to $(p_1(W),\ldots, p_k(W))$, the vector $w_1\wedge\ldots\wedge w_\ell$ can be expanded as
  $$
  w_1\wedge\cdots\wedge w_\ell=\sum_{(i_1,\ldots, i_\ell)\in \mathcal I_{\ell, W}}b_{i_1,\ldots, i_\ell}\;v_{i_1}\wedge \cdots\wedge v_{i_\ell}
  $$
with   $b_{p_1(W),\ldots, p_\ell(W)}=1$. It follows that for $n\in \N$,
\begin{align*}
 (T_{x|n}^*P_W)^{\wedge \ell}(w_1\wedge\cdots\wedge w_\ell)&=(T_{x|n}^*)^{\wedge \ell}P_{W^{\w \ell}}(w_1\wedge\cdots\wedge w_\ell)\qquad\quad\mbox{ (by \eqref{e-txn})}\\
 &=(T_{x|n}^*)^{\wedge \ell}(w_1\wedge\cdots\wedge w_\ell)\\
 &=\sum_{(i_1,\ldots, i_\ell)\in \mathcal I_{\ell,W}}b_{i_1,\ldots, i_\ell}\;v_{i_1,\ldots, i_\ell}^{(n)}.
\end{align*}
By Lemma \ref{lem-F-3.2}(i), for $n\geq N_\ell$,
\begin{align*}
\left\|(T_{x|n}^*P_W)^{\wedge \ell}(w_1\wedge\cdots\wedge w_\ell)\right\|&\geq \left|b_{p_1(W),\ldots, p_\ell(W)}\right|\;\left\|v_{p_1(W),\ldots, p_\ell(W)}^{(n)}\right\|\sin\left(\alpha_n^{(\ell)}\right)\\
&=\left\|v_{p_1(W),\ldots, p_\ell(W)}^{(n)}\right\|\sin\left(\alpha_n^{(\ell)}\right)\\
& \geq \exp\Big(n\Big(-2\epsilon+\sum_{j=1}^\ell \Lambda_{p_j(W)}\Big)\Big)\qquad \mbox{(by \eqref{e-3.1a})}.
\end{align*}
Hence
\begin{align*}
\left\|(T_{x|n}^*P_W)^{\wedge \ell}\right\| &\geq \frac{\left\|(T_{x|n}^*P_W)^{\wedge \ell}(w_1\wedge\cdots\wedge w_\ell)\right\|}{\|w_1\wedge\cdots\wedge w_\ell\|}\nonumber \\
&\geq \frac{\exp\Big(n\Big(-2\epsilon+\sum_{j=1}^\ell \Lambda_{p_j(W)}\Big)\Big)}{\|w_1\wedge\cdots\wedge w_\ell\|}.
\label{e-lower}
\end{align*}
This proves \eqref{e-lower1}  by taking
$$
D_{\ell, W}:=\frac{1}{\|w_1\wedge\cdots\wedge w_\ell\|}.
$$

Next let $s\in [0,k]$. Take $$N:=\max_{1\leq \ell\leq k}N_\ell, \qquad C:=\max_{1\leq \ell\leq k}C_\ell,$$
 and
$$D_W:=\min_{1\leq \ell\leq k} D_{\ell, W} \quad \mbox{ for } W\in G(d,k).$$
 By  Lemma \ref{lem-phis}, \eqref{e-upper1} and \eqref{e-lower1}, we see that for all $n\geq N$ and $W\in G(d,k)$,
\begin{equation}
\label{e-upper2}
 \varphi^s(T_{x|n}^*P_W) \leq C \exp\left(n\left(\epsilon+\Big(\sum_{j=1}^{\lfloor s \rfloor} \Lambda_{p_j(W)}\Big)+(s-\lfloor s\rfloor)\Lambda_{p_{\lfloor s\rfloor +1}(W)}\right)\right),
 \end{equation}
and
\begin{equation}
\label{e-lower2}
\varphi^s(T_{x|n}^*P_W) \geq D_W \exp\left(n\left(-2\epsilon+\Big(\sum_{j=1}^{\lfloor s \rfloor} \Lambda_{p_j(W)}\Big)+(s-\lfloor s\rfloor)\Lambda_{p_{\lfloor s\rfloor +1}(W)}\right)\right).
\end{equation}

Clearly \eqref{e-sw} follows from \eqref{e-upper2} and \eqref{e-lower2} since $\epsilon$ is arbitrarily taken. To see \eqref{e-supw}, let $W\in G(d,k)$ and write
$$
\Gamma:=\sup_{J\in \Sigma_*} \sum_{j=1}^{\lfloor s \rfloor} \Lambda_{p_j(T^*_JW)}+(s-\lfloor s\rfloor)\Lambda_{p_{\lfloor s\rfloor +1}(T^*_JW)}.
$$
Since $p_j(\cdot)$, $1\leq j\leq k$, take values in $\{1,\ldots, d\}$, the above supremum in defining $\Gamma$ is attained at some $J_0\in \Sigma_*$.
  By \eqref{e-upper2} and \eqref{e-lower2} we see that for $n\geq N$,
$$
D_{T^*_{J_0}W}\exp(n(-2\epsilon+\Gamma))\leq \sup_{J\in \Sigma_*} \varphi^s(T_{x|n}^*P_{T_J^*W})\leq C \exp(n(\epsilon+\Gamma)).
$$
Since $\epsilon$ is arbitrarily taken, this proves \eqref{e-supw}.
\end{proof}

In the remaining part of this section,  we give two direct applications of Proposition \ref{pro-key}.

\begin{cor}
\label{cor-2.8}
Under the conditions of Proposition \ref{pro-key}, there exist a measurable $A\subset \Sigma'$ with $\mu(A)>0$, and $J\in \Sigma_*$ such that for each $x\in A$,
$$
\lim_{n\to \infty}\frac{1}{n}\log \varphi^s(T_{x|n}^*P_{T_J^*W})=\lim_{n\to \infty}\frac{1}{n}\log \psi_W^s(x|n)=\Theta,
$$
where $\psi_W^s(x|n):=\sup_{J\in \Sigma_*} \varphi^s(T_{x|n}^*P_{T_J^*W})$ and $$\Theta=\lim_{n\to \infty}\frac1n\int \log \psi_W^s(x|n)\; d\mu(x).$$
\end{cor}
\begin{proof}Since $\psi^s_W$ is submultiplicative by Lemma \ref{lem-subm}, it follows from the Kingman's subadditive ergodic theorem that there exists a measurable set $\Sigma''\subset \Sigma'$ with $\mu(\Sigma'')=\mu(\Sigma')=1$ such that
\begin{equation}
\label{e-psiTheta}
\lim_{n\to \infty}\frac{1}{n}\log \psi_W^s(x|n)=\Theta \quad\mbox{ for all }x\in \Sigma''.
\end{equation}
Meanwhile by Remark \ref{rem-pivot}, $p_j(T_J^*W,x)$ is measurable in $x$ and takes value in $\{1,\ldots, d\}$ for each $1\leq j\leq k$ and $J\in \Sigma_*$. By \eqref{e-supw}, there exists a measurable $J'\colon \Sigma''\to \Sigma_*$ such that for each $x\in \Sigma''$, the supermum in the righthand side of \eqref{e-supw} is attained at $J=J'(x)$. Since $\Sigma_*$ is countable, there exists $J_0\in \Sigma_*$ such that $A:=\{x\in \Sigma''\colon J'(x)=J_0\}$ has positive $\mu$ measure. Now for each $x\in A$,
\begin{align*}
\Theta&=\lim_{n\to \infty}\frac{1}{n}\log \psi_W^s(x|n)\qquad\qquad\qquad\qquad\qquad\qquad\mbox{(by \eqref{e-psiTheta})}\\
&=\sum_{j=1}^{\lfloor s \rfloor}\Lambda_{p_j(T^*_{J_0}W,x)}+(s-\lfloor s \rfloor)\Lambda_{p_{\lfloor s \rfloor+1}(T^*_{J_0}W,x)}\quad\qquad\qquad\mbox{(by \eqref{e-supw})}\\
&=\lim_{n\to \infty}\frac{1}{n}\log \varphi^s(T_{x|n}^*P_{T_{J_0}^*W}) \;\;\quad \qquad\qquad\qquad\qquad\mbox{(by \eqref{e-sw})}.
\end{align*}
This completes the proof.
\end{proof}
%{\color{red} {\bf Remark}:  the righthand sides of \eqref{e-sw}-\eqref{e-supw} are not well defined for $s=0$. A adjustment is needed.}

Recall the definitions of $S_n(\mu, {\bf T},W,x)$ and $S(\mu, {\bf T},W,x)$; see \eqref{e-2.2}, \eqref{e-equi} and \eqref{e-Smux}.
\begin{lem}
\label{lem-smuw}
Let $k\in \{1,\ldots, d-1\}$, $W\in G(d,k)$ and $x\in \Sigma'$. Let $(p_1,\ldots, p_k)$ denote the pivot position vector of $W$ with respect to the ordered basis ${\bf v}(x)$. Set
\begin{equation}
\label{e-Gammas} \Gamma(s)=\sum_{j=1}^{\lfloor s \rfloor} \Lambda_{p_j}+(s-\lfloor s \rfloor)\Lambda_{p_{\lfloor s \rfloor+1}}\quad \mbox{ for }s\in [0,k].
\end{equation}
Then the limit $\lim_{n\to \infty}S_n(\mu, {\bf T}, W, x)$ in defining $S(\mu, {\bf T}, W, x)$ exists. Moreover,
\begin{equation}
\label{e-eqsmu}
S(\mu, {\bf T}, W, x)=\left\{
\begin{array}{ll}
k, & \mbox{ if }h_\mu(\sigma)+\Gamma(k)\geq 0,\\
s\in [0,k) \mbox{ with  } h_\mu(\sigma)+\Gamma(s)=0,& \mbox{ otherwise}.
\end{array}
\right.
\end{equation}
\end{lem}
\begin{proof}
Let $S$ be the  largest $s\in [0,k]$ such that
$h_\mu(\sigma)+ \Gamma(s)\geq 0$.
 Since $\Gamma$ is strictly decreasing and bi-Lipschitz on $[0,k]$, it follows that either $S=k$  and $h_\mu(\sigma)+\Gamma(S)\geq 0$, or $S\in [0,k)$ and $h_\mu(\sigma)+\Gamma(S)=0$. That is, $S$ is given by the righthand side of \eqref{e-eqsmu}.

Next we  prove that
 \begin{equation}
 \label{e-SnS}
\lim_{n\to \infty} S_n(\mu,{\bf T}, W,x)=S.
\end{equation}
 To this end, let $\epsilon>0$. We need to  show that
\begin{equation}
\label{e-Sn}
S-\epsilon\leq S_n(\mu, {\bf T}, W,x)\leq S+\epsilon
\end{equation}
for large enough $n$.

To prove the first inequality in  \eqref{e-Sn}, we may assume that $S-\epsilon\geq 0$. As $\Gamma$ is strictly decreasing,  by the definition of $S$, $h_\mu(\sigma)+\Gamma(S-\epsilon)>0$. Notice that by Theorem \ref{thm:oseledets}(6) and Proposition \ref{pro-key},
\begin{equation}
\label{e-entropy}
\lim_{n\to \infty}\frac{1}{n} \log \mu([x|n])=-h_\mu(\sigma),\quad \lim_{n\to \infty} \frac{1}{n}\log \varphi^s(P_WT_{x|n})=\Gamma(s) \mbox{ for } s\in [0,k].
\end{equation}
 Combining them with the inequality $h_\mu(\sigma)+\Gamma(S-\epsilon)>0$  yields that
$$\varphi^{S-\epsilon}(P_WT_{x|n})> \mu([x|n])\quad \mbox{ for large enough } n,$$
which implies the first inequality in \eqref{e-Sn} for large enough $n$.

To prove the second inequality in \eqref{e-Sn}, we may assume $S+\epsilon<k$; otherwise there is nothing to prove. Since $S<k$, by definition it follows that
$h_\mu(\sigma)+\Gamma(S)=0$, and thus $h_\mu(\sigma)+\Gamma(S+\epsilon)<0$.
Combining this with  \eqref{e-entropy} yields that
$$\varphi^{S+\epsilon}(P_WT_{x|n})< \mu([x|n])\quad \mbox{ for large enough } n,$$
which implies the second inequality in \eqref{e-Sn} for large enough $n$. This completes the proof of \eqref{e-Sn}.
\end{proof}
\section{A special family of  sub-additive pressures}
\label{S-4}

Throughout this section, let ${\bf T}=(T_1,\ldots, T_m)$ be a fixed tuple  of $d \times d$  invertible real matrices with $\|T_i\|<1$ for $1\leq i\leq m$. For each $W\in G(d,k)$, we are going to introduce a parametrized family of subadditive potentials and prove that  their topological pressures coincide with the following limit
$$\lim_{n\to \infty}\frac1n\log \sum_{I\in \Sigma_n}\varphi^s(T_I^*P_W);$$
see Proposition \ref{pro-equivalence}. This result plays an important  role in the proof of Theorem \ref{thm-1.1}.

Recall that for a linear subspace $V$ of $\R^d$, $P_V$ stands for the orthogonal projection from $\R^d$ onto $V$. For $s\geq 0$ and $W\in G(d,k)$, we define $\psi_W^s\colon \Sigma_*\to [0,\infty)$ by
\begin{equation}
\label{e-psiw}
\psi_W^s(I)=\sup_{K\in \Sigma_*}\varphi^s(T_I^*P_{T_K^*W}).
\end{equation}
The introduction of $\psi_W^s$ and the proof of the following lemma were inspired by \cite[Theorem 4]{BochiMorris2018}.

\begin{lem}
\label{lem-subm}
         For any $s \geq 0$ and $W\in G(d,k)$, $\psi^{s}_{W}$ is submultiplicative in the sense that $\psi^{s}_{W}(IJ)\leq \psi^{s}_{W}(I)\psi^{s}_{W}(J)$ for all $I, J\in \Sigma_*$.
    \end{lem}
    \begin{proof}
        Let $I, J, K\in \Sigma_*$. Notice that
       $
       T_I^* P_{T^*_KW}(\R^d)=  T_I^* T^*_KW=T_{KI}^*W.
       $
      It follows that
      $$
      T_I^* P_{T^*_KW}=P_{T_{KI}^*W}T_I^* P_{T^*_KW}.
      $$
     Hence
     $$
     T_{IJ}^*P_{T^*_KW}=T_{J}^*T_I^*P_{T^*_KW}=T_{J}^*P_{T_{KI}^*W}T_I^* P_{T^*_KW}=\left(T_{J}^* P_{T_{KI}^*W}\right)\left(T_I^* P_{T^*_KW}\right).
     $$
     Since $\varphi^s$ is submultiplicative, it follows that
     $$
     \varphi^s\left(T_{IJ}^*P_{T^*_KW}\right)\leq \varphi^s\left(T_{J}^* P_{T_{KI}^*W}\right)\varphi^s\left(T_I^* P_{T^*_KW}\right)\leq  \psi^{s}_{W}(J)\psi^{s}_{W}(I).
     $$
     Taking supremum over $K\in \Sigma_*$ gives    $\psi^{s}_{W}(IJ)\leq \psi^{s}_{W}(I)\psi^{s}_{W}(J)$.
        \end{proof}

Next we collect some elementary properties of $\psi^s_W$.
\begin{lem}
\label{lem-4.3}
Let $s\geq 0$ and $W\in G(d,k)$. Then the following statements hold.
\begin{itemize}
\item[(i)]  If $s>k$, then $\psi^s_W(I)=0$ for any $I\in \Sigma_*$.
\item[(ii)] If $s\in [0,k]$, then $$(\alpha_-)^{sn}\leq \psi_W^s(I)\leq (\alpha_+)^{sn}$$ for each $n\in \N$ and $I\in \Sigma_n$, where
$\alpha_-$ and $\alpha_+$ are defined by
\begin{equation}
\label{e-alpha}
\alpha_-=\min_{1\leq i\leq m}\alpha_d(T_i),\qquad \alpha_+=\max_{1\leq i\leq m}\|T_i\|.
\end{equation}
\item[(iii)] For $0\leq s<t\leq k$ and $I\in \Sigma_*$,  $$\psi^{s}_W(I)(\alpha_-)^{(t-s)|I|}\leq \psi^{t}_W(I)\leq  \psi^{s}_W(I)(\alpha_+)^{(t-s)|I|}.$$
\end{itemize}
\end{lem}
\begin{proof}
Parts (i) and (ii) follow directly  from the definition of $\psi^s_W$ and Lemma \ref{lem-singular1}(ii)-(iii). The second inequality in Part (iii) follows directly from Lemma \ref{lem-singinequality}(ii). For the first  inequality in Part (iii), since $s<t\leq k$, it follows that
\begin{align*}
\psi^{t}_W(I)&=\sup_{K\in \Sigma_*}\varphi^t(T_I^*P_{T_K^*W})\\
&\geq \sup_{K\in \Sigma_*}\varphi^s(T_I^*P_{T_K^*W})\alpha_k(T_I^*P_{T_K^*W})^{t-s}\\
&\geq  \sup_{K\in \Sigma_*}\varphi^s(T_I^*P_{T_K^*W})(\alpha_-)^{(t-s)|I|}\qquad \qquad (\mbox{by Lemma \ref{lem-singular1}(i)})\\
&=\psi^{s}_W(I)(\alpha_-)^{(t-s)|I|}.
\end{align*}
This completes the proof.
\end{proof}

Now for $s\geq 0$ and $W\in G(d,k)$, define
\begin{equation}
\label{e-ptws}
P({\bf T}, W,s)=\lim_{n\to \infty} \frac{1}{n}\log \sum_{I\in \Sigma_n}\psi_W^s(I).
\end{equation}
Due to Lemma \ref{lem-subm},  $P({\bf T}, W,s)$ is the topological pressure of the subadditive potential $\{\log \psi^s_W(\cdot|n)\}_{n=1}^\infty$; see Section \ref{S-subadditive} for the definition of the topological pressure of a subadditive potential. The following result is a direct consequence of Lemma \ref{lem-4.3}.

\begin{lem}
\label{lem-Pws}
\begin{itemize}
\item[(i)] $P({\bf T}, W,s)\in \R$ for all $0\leq s\leq k$, and $P({\bf T}, W,s)=-\infty$ if $s>k$.
\item[(ii)] For  all $0\leq t_1<t_2\leq k$,
$$(t_2-t_1)\log (1/\alpha_+)\leq P({\bf T}, W,t_1)-P({\bf T}, W,t_2)\leq (t_2-t_1)\log \left(1/\alpha_-\right),$$
where $\alpha_-,\alpha_+$ are defined as in \eqref{e-alpha} and are less than 1.
\end{itemize}
\end{lem}

The main result of this section is the following, which will  be derived from Corollary \ref{cor-2.8} and the subadditive variational principle.

\begin{pro}
\label{pro-equivalence}
For $s\geq 0$ and $W\in G(d,k)$,
\begin{equation}
\label{e-varpw}
P({\bf T}, W,s)=\lim_{n\to \infty} \frac{1}{n}\log \sum_{I\in \Sigma_n} \varphi^s(T_I^*P_W).
\end{equation}
\end{pro}

To prove this result, we need the following simple facts in linear algebra.

\begin{lem}
\label{lem-simple}
\begin{itemize}
\item[(i)] Let $A, B\in {\rm Mat}_d(\R)$ with $A(\R^d)=B(\R^d)$. Then $A=BD$ for some $D\in {\rm GL}_d(\R)$.
\item[(ii)] Let $L\in {\rm Mat}_d(\R)$ with rank $k$. Set $W=L^*(\R^d)$. Then there exists $M\in {\rm GL}_d(\R)$ such that
$L=MP_W$.
\end{itemize}
\end{lem}

\begin{proof}
Part (i) is standard; see e.g.~\cite[p.~56]{HoffmanKunze1961}. To see (ii), let $W=L^*(\R^d)$. By (i), there exists $D\in {\rm GL}_d(\R)$ such that $L^*=P_WD$. Taking transpose gives $L=D^*P_W$.
\end{proof}

\begin{proof}[Proof of Proposition \ref{pro-equivalence}] Let $s\geq 0$ and $W\in G(d,k)$.  By \eqref{e-psiw},  $\psi_W^s(I)\geq \varphi^s(T_I^*P_W)$ for every $I\in \Sigma_*$. It follows that
$$
P({\bf T}, W,s)\geq \limsup_{n\to \infty} \frac{1}{n}\log \sum_{I\in \Sigma_n} \varphi^s(T_I^*P_W).
$$
Below, we prove that
\begin{equation}
\label{e-reverse}
P({\bf T}, W,s)\leq \liminf_{n\to \infty} \frac{1}{n}\log \sum_{I\in \Sigma_n} \varphi^s(T_I^*P_W).
\end{equation}

Let $\mu$ be an ergodic equilibrium measure for the  potential $\{\log \psi^s_W(\cdot|n)\}_{n=1}^\infty$. Then
$$
P({\bf T}, W,s)=h_\mu(\sigma)+\Theta,
$$
where $\Theta:=\lim_{n\to \infty}(1/n)\int \log \psi^s_W(x|n)\; d\mu(x)$.
By the Shannon-McMillan-Breiman theorem (see e.g.~\cite[p.~261]{Petersen1989}),
\begin{equation}
\label{e-SMB}
\lim_{n\to \infty}\frac{1}{n}\log \mu([x|n])=-h_\mu(\sigma) \quad \mbox{ for $\mu$-a.e.~$x\in \Sigma$}.
\end{equation}
Meanwhile by Corollary \ref{cor-2.8}, there exist a measurable set $A\subset \Sigma'$ with $\mu(A)>0$ and a word $J\in \Sigma_*$ such that
\begin{equation}
\label{e-psiw1}
\lim_{n\to \infty}\frac{1}{n}\log \varphi^s(T_{x|n}^*P_{T_J^*W})=\lim_{n\to \infty}\frac{1}{n}\log \psi_W^s(x|n)=\Theta\quad \mbox{ for }x\in A.
\end{equation}

Let $\epsilon>0$. By \eqref{e-SMB} and \eqref{e-psiw1}, there exist $N\in \N$ and $A_1\subset A$ with $\mu(A_1)>0$ such that for all $x\in A_1$ and $n\geq N$,
\begin{equation}
\label{e-special'}
\varphi^s(T_{x|n}^*P_{T_J^*W})\geq e^{n(\Theta-\epsilon)},\qquad \mu([x|n])\leq e^{-n(h_\mu(\sigma)-\epsilon)}.
\end{equation}

Write for $n\in\N$, $$\mathcal A_n:=\{I\in \Sigma_n\colon [I]\cap A_1\neq \emptyset\}.$$
By  \eqref{e-special'}, for each $n\geq N$ and $I\in \mathcal A_n$,
$$
\varphi^s(T_{I}^*P_{T_J^*W})\geq e^{n(\Theta-\epsilon)}\quad\mbox{ and }\quad  \mu([I])\leq e^{-n(h_\mu(\sigma)-\epsilon)}.
$$
It follows that  for $n\geq N$,
$$
\mu(A_1)\leq \sum_{I\in \mathcal A_n}\mu([I])\leq \#(\mathcal A_n) \cdot e^{-n(h_\mu(\sigma)-\epsilon)},
$$
which implies that $\#(\mathcal A_n)\geq \mu(A_1)e^{n(h_\mu(\sigma)-\epsilon)}$. Hence for $n\geq N$,
\begin{equation}
\label{e-special1}
\sum_{I\in \Sigma_n}\varphi^s(T_{I}^*P_{T_J^*W})\geq \sum_{I\in \mathcal A_n}\varphi^s(T_{I}^*P_{T_J^*W})\geq \#(\mathcal A_n)e^{n(\Theta-\epsilon)}\geq \mu(A_1)e^{n(h_\mu(\sigma)+\Theta-2\epsilon)}.
\end{equation}

Finally, notice that $P_{T_J^*W}(\R^d)=T_J^*P_W(\R^d)$.  By Lemma \ref{lem-simple}(i), there exists $M\in {\rm GL}_d(\R)$ such that
$P_{T_J^*W}=T_J^*P_WM.$ It follows that
$$
\sum_{I\in \Sigma_n}\varphi^s(T_{I}^*P_{T_J^*W})=\sum_{I\in \Sigma_n}\varphi^s(T_{I}^*T_J^*P_{W}M)\leq \varphi^s(M)\sum_{I\in \Sigma_n}\varphi^s(T_{I}^*T_J^*P_{W}),
$$
where the last inequality follows from the submultiplicativity of $\varphi^s$. Combining this with \eqref{e-special1} gives
\begin{align*}
\sum_{I\in \Sigma_{n+|J|}} \varphi^s(T_I^*P_W)&\geq \sum_{I\in \Sigma_n}\varphi^s(T_{I}^*T_J^*P_{W})\\
&\geq (\varphi^s(M))^{-1}\sum_{I\in \Sigma_n}\varphi^s(T_{I}^*P_{T_J^*W})\\
&\geq(\varphi^s(M))^{-1} \mu(A_1)e^{n(h_\mu(\sigma)+\Theta-2\epsilon)}.
\end{align*}
It implies that
$$
\liminf_{n\to \infty} \frac{1}{n}\log \sum_{I\in \Sigma_n} \varphi^s(T_I^*P_W)\geq h_\mu(\sigma)+\Theta-2\epsilon=P({\bf T}, W,s)-2\epsilon.
$$
Letting $\epsilon\to 0$ yields \eqref{e-reverse}.
\end{proof}

Recall that  $\dim_{\rm AFF}({\bf T})$ and $\dim_{\rm AFF}({\bf T}, W)$ are defined as in \eqref{e-aff} and \eqref{e-affine}, respectively. Below we will illustrate the relations between $\dim_{\rm AFF}({\bf T}, W)$, $\dim_{\rm AFF}({\bf T})$ and $P({\bf T}, W,s)$.
\begin{lem}
\label{lem-Affd}
Let $W\in G(d,k)$. Then the following statements hold.
\begin{itemize}
\item[(i)]
 $\dim_{\rm AFF}({\bf T}, W)\leq \min\{k, \dim_{\rm AFF}({\bf T})\}$.
 \item[(ii)] $\dim_{\rm AFF}({\bf T}, W)=\sup\{s\in [0,k]\colon P({\bf T}, W,s)\geq 0\}$.
\item[(iii)] Setting $t=\dim_{\rm AFF}({\bf T}, W)$, we have
\begin{equation}
\label{e-Pwt}
\left\{
\begin{array}{ll} P({\bf T}, W,t)\geq 0 &\mbox{  if } t=k,\\
P({\bf T}, W,t)=0 &\mbox{  if }t<k.
\end{array}
\right.
\end{equation}
\end{itemize}
\end{lem}
\begin{proof}
To prove (i), let $s>\min\{k, \dim_{\rm AFF}({\bf T})\}$. Then either $s>k$, or $s>\dim_{\rm AFF}({\bf T})$. In the case when $s>k$, by Lemma \ref{lem-singular1}(ii),  $\phi^s(P_WT_I)=0$ for each $I\in \Sigma_*$, and consequently,
$$
\sum_{n=1}^\infty\sum_{I\in \Sigma_n}\varphi^s(P_WT_I)=0.
$$
In the other case when $s>\dim_{\rm AFF}({\bf T})$,
$$
\sum_{n=1}^\infty\sum_{I\in \Sigma_n}\varphi^s(P_WT_I)\leq \sum_{n=1}^\infty\sum_{I\in \Sigma_n}\varphi^s(T_I)<\infty.
$$
From the definition of $\dim_{\rm AFF}({\bf T}, W)$ (see \eqref{e-affine}),  we conclude that in both cases $\dim_{\rm AFF}({\bf T}, W)\leq s$. This proves (i).

To see (ii), by Proposition \ref{pro-equivalence} and Lemma \ref{lem-Pws}(i), for all $0\leq s\leq k$,
$$
\lim_{n\to \infty} \frac{1}{n}\log \sum_{I\in \Sigma_n} \varphi^s(T_I^*P_W)=P({\bf T}, W,s)\in \R.
$$
Combining this with the definition of $\dim_{\rm AFF}({\bf T}, W)$ (see \eqref{e-affine}) yields (ii).

Since  $P({\bf T}, W,s)$ is monotone decreasing and continuous in $s$ on $[0,k]$ as stated in Lemma \ref{lem-Pws}(ii), we can conclude \eqref{e-Pwt} from (ii).
\end{proof}
\section{ Proof of Theorem \ref{thm-1.2}}
\label{S-5}

Throughout this section, let ${\bf T}=(T_1,\ldots, T_m)$ be a tuple  of $d \times d$  invertible real matrices with $\|T_i\|<1$ for $1\leq i\leq m$. For $\ba=(a_1,\ldots, a_m)\in \R^{md}$, let $\pi^\ba:\Sigma\to \R^d$ be the coding map associated with the IFS $\{f_i^\ba(x)=T_ix+a_i\}_{i=1}^m$ (see \eqref{e-pia}). For short we write $f_I^\ba:=f^\ba_{i_1}\circ \cdots\circ f^\ba_{i_n}$ and $T_I:=T_{i_1}\cdots T_{i_n}$ for $I=i_1\cdots i_n\in \Sigma_n:=\{1,\ldots, m\}^n$. Let $\mu$ be a fixed ergodic $\sigma$-invariant measure on $\Sigma$, and let $\Sigma'$, $\Lambda_1,\ldots, \Lambda_d$, ${\bf v}(x)=\{v_i(x)\}_{i=1}^d$ $(x\in \Sigma'$) be given as in Theorem \ref{thm:oseledets}.

\begin{proof}[Proof of Theorem \ref{thm-1.2}(i)]
By Lemma \ref{lem-smuw},  the limit $\lim_{n\to \infty} S_n(\mu, {\bf T},W,x)$ in defining
$S(\mu, {\bf T}, W, x)$ exists for every $W\in G(d,k)$ and $x\in \Sigma'$.

 Let $\ell'$ be the smallest integer not less than $\min\{k, \dim_{\rm LY}(\mu, {\bf T})\}$. We need to prove that
\begin{equation}
\label{e-num}
 \#\{S(\mu, {\bf T}, W, x)\colon W\in G(d,k),\; x\in \Sigma'\}\leq \binom{d+\ell'-k}{\ell'}.
\end{equation}

By Lemma \ref{lem-smuw}, for $x\in \Sigma'$ and $W\in G(d,k)$, $S(\mu, {\bf T}, W,x)$ only depends on the (integral) pivot position vector
$
(p_i=p_i(W,x))_{i=1}^k
$
of $W$ with respect to ${\bf v}(x)$.
Since
\begin{equation}
\label{e-pwx}
1\leq p_1<p_2<\cdots <p_k\leq d,
\end{equation}
this vector can take at most $\binom {d}{k}$ different values when $(W,x)$ runs over $G(d,k)\times \Sigma'$, so we get the upper bound
\begin{equation}
\label{e-card}
 \#\{S(\mu, {\bf T}, W, x)\colon W\in G(d,k),\; x\in \Sigma'\}\leq \binom{d}{k}.
\end{equation}
This proves \eqref{e-num} in the case when $\ell'= k$.

Next we assume that $\ell'<k$. In this case,  $\dim_{\rm LY}(\mu, {\bf T})< k$. Write $s_0:=\dim_{\rm LY}(\mu, {\bf T})$. By  Definition \ref{de-Lyapunov} and \eqref{e-Gs},
\begin{equation}
\label{e-hLY}
h_\mu(\sigma)+\sum_{i=1}^{\lfloor s_0\rfloor}\Lambda_{i}+(s_0-\lfloor s_0\rfloor)\Lambda_{{\lfloor s_0\rfloor +1}}=0.
\end{equation}
For $x\in \Sigma'$ and $W\in G(d,k)$,  we obtain from \eqref{e-pwx} that
$$
p_i:=p_i(W,x)\geq i,\quad i=1,\ldots, k.
$$
 It follows that
 \begin{align*}
 h_\mu&(\sigma)+\sum_{i=1}^{\lfloor s_0\rfloor}\Lambda_{p_i}+(s_0-\lfloor s_0\rfloor)\Lambda_{p_{\lfloor s_0\rfloor +1}}\\
 &\leq h_\mu(\sigma)+\sum_{i=1}^{\lfloor s_0\rfloor}\Lambda_{i}+(s_0-\lfloor s_0\rfloor)\Lambda_{{\lfloor s_0\rfloor +1}}\\
 &=0\qquad\qquad \mbox{(by \eqref{e-hLY})}.
 \end{align*}
 That is, $h_\mu(\sigma)+\Gamma(s_0)\leq 0$, where $\Gamma$ is defined as in \eqref{e-Gammas}.  Since $s_0<k$, by Lemma \ref{lem-smuw},  $S(\mu, {\bf T}, W,x)=s\leq s_0$, where $s$ is the unique number in $[0,k]$ such that
 $h_\mu(\sigma)+\Gamma(s)=0$. Since $s\leq s_0$ and $\ell'=\lceil s_0\rceil$,  $s$ is uniquely determined by $p_1, \ldots, p_{\ell'}$.
 Since $p_{\ell'}<p_{\ell'+1}<\cdots <p_k\leq d$, it follows that $$p_{\ell'}\leq d-(k-\ell')=d+\ell'-k.$$ Hence the vector
 $(p_i)_{i=1}^{\ell'}$
  can take at most $\binom {d+\ell'-k}{\ell'}$ different values when $(W,x)$ runs over $G(d,k)\times \Sigma'$, so we get the upper bound
$$
 \#\{S(\mu, {\bf T}, W, x)\colon W\in G(d,k),\; x\in \Sigma'\}\leq \binom{d+\ell'-k}{\ell'}.
$$
This completes the proof of \eqref{e-num}.
\end{proof}

Recall that for $z\in \R^d$ and $r>0$,  $B(z,r)$ stands for the closed ball of radius $r$ centred at $z$.  To prove part (ii) of Theorem \ref{thm-1.2}, we need the following result.
\begin{lem} \label{LemJ}
Let $\ba\in \R^{md}$ and $W\in G(d, k)$. Define $g:\Sigma\to \R^d$ by $g=P_W\pi^{\ba}$, and let $\eta=g_*\mu$ be the push-forward of $\mu$ by $g$.  Then there is a positive constant $c>0$ which depends on $\ba$ and ${\bf T}$ such that the following property holds.  For every $\epsilon\in (0,1)$ and  $\ell\in \{0,1,\ldots, k-1\}$,  we have for $\mu$-a.e.~$x \in \Sigma$,
\begin{equation}
\label{e-lemJ}
\eta \Big( B\big(g(x), c \alpha_{\ell+1}(P_WT_{x|n})\big) \Big) \ge (1-\epsilon)^n \dfrac{ \mu([x|n])}{N_{\ell}(x|n)}\quad  \mbox{ for large enough $n$},
\end{equation}
where
\begin{equation}
\label{e-Nxn}
N_{\ell}(x|n) := \alpha_1(P_WT_{x|n}) \cdots \alpha_{\ell}(P_WT_{x|n}) \alpha_{\ell+1}^{-\ell} (P_WT_{x|n}).
\end{equation}
\end{lem}
\begin{proof}
The statement of the lemma and its proof are  slightly  modified from an unpublished note \cite{Jor11} of Jordan; see  \cite[Lemma 2.2]{FLM2023} for Jordan's arguments. For the  reader's convenience, below we include a detailed proof of the lemma.

Let $\ba=(a_1,\ldots, a_m)\in \R^{md}$ and $W\in G(d,k)$. Take a large $R>0$ such that
$$
f^\ba_i(B(0,R))\subset B(0, R)  \qquad \mbox{ for } i=1,\ldots,m,
$$
where $f^\ba_i(x):=T_ix+a_i$.  Clearly, the attractor $\pi^{\ba}(\Sigma)$ of the IFS $\{f^\ba_i\}_{i=1}^m$ is contained in $B(0, R)$, which implies that $P_W\pi^{\ba}(\Sigma)\subset B(0, R)$.   Take $c=4 R\sqrt{d}$. Below, we show that the statement of the lemma holds for this choice of $c$.

Let $\epsilon\in (0,1)$ and $\ell\in \{0,\ldots, k-1\}$. For $n\in \N$, let $\Lambda_n$ denote the set of points $x\in \Sigma$ such that
$$
\eta \Big( B\big(g (x), c \alpha_{\ell+1}(P_WT_{x|n})\big) \Big) < (1-\epsilon)^n \dfrac{ \mu([x|n])}{N_{\ell}(x|n)}.
$$
To prove that \eqref{e-lemJ} holds for $\mu$-a.e.~$x\in \Sigma$,   by the Borel-Cantelli lemma it suffices to show that
\begin{equation}
\label{e-teta}
\sum_{n=1}^\infty \mu(\Lambda_n)<\infty.
\end{equation}

To prove \eqref{e-teta}, let $n\in \N$ and $I\in \Sigma_n$. Notice that  $P_Wf^\ba_I(B(0, R))$ is an ellipsoid of semi-axes
$$
R\alpha_1(P_WT_{I})\geq \cdots\geq R\alpha_d(P_WT_{I}),
$$
so it can be covered by $$2^\ell \prod_{i=1}^\ell \frac{\alpha_i(P_WT_{I})}{\alpha_{\ell+1}(P_WT_{I})}$$ balls of radius $2R\sqrt{d}\alpha_{\ell+1}(P_WT_{I})$. Since $g([I])=P_W\pi^\ba([I])$ is contained in $P_Wf^\ba_I(B(0, R))$,  it follows that there exists a nonnegative  integer $L$ satisfying
\begin{equation}
\label{e-JL}
L\leq 2^\ell \prod_{i=1}^\ell \frac{\alpha_i(P_WT_{I})}{\alpha_{\ell+1}(P_WT_{I})}=2^\ell N_{\ell}(I)
\end{equation}
 such that $g(\Lambda_n\cap [I])$ can be covered by $L$ balls of radius $2R\sqrt{d}\alpha_{\ell+1}(P_WT_{I})$, denoted as $B_1,\ldots, B_L$. We may assume that $g(\Lambda_n\cap [I])\cap B_i\neq \emptyset$ for each $1\leq i\leq L$. Hence for each $i$, we may pick $x^{(i)}\in \Lambda_n \cap [I]$ such that $g(x^{(i)}) \in B_i$. Clearly
\begin{equation}
\label{e-Bi}
B_i\subset B\left(g (x^{(i)}), 4R\sqrt{d}\alpha_{\ell+1}(P_WT_{I})\right)=B\left(g( x^{(i)}), c\alpha_{\ell+1}(P_WT_{I})\right).
\end{equation}
Since $x^{(i)}\in \Lambda_n \cap [I]$, by the definition of $\Lambda_n$ we obtain
\begin{equation}
\label{e-pi}
\eta \Big( B\big(g (x^{(i)}), c \alpha_{\ell+1}(P_WT_{I})\big) \Big)< (1-\epsilon)^n\frac{\mu([I])}{N_{\ell}(I)}.
\end{equation}
It follows that
\begin{eqnarray*}
\mu(\Lambda_n\cap [I])&\leq& \mu\circ g^{-1}(g(\Lambda_n\cap [I]))\\
&\leq & \eta\left(\bigcup_{i=1}^LB_i\right)\\
&\leq & \eta\left(\bigcup_{i=1}^LB\left(g( x^{(i)}), c\alpha_{\ell+1}(P_WT_{I})\right)\right)\qquad \mbox{(by \eqref{e-Bi})}\\
&\leq & L (1-\epsilon)^n\frac{\mu([I])}{N_{\ell}(I))}\qquad \mbox{(by \eqref{e-pi})}\\
&\leq & 2^\ell(1-\epsilon)^n\mu([I])\qquad \mbox{(by \eqref{e-JL})}.
\end{eqnarray*}
Summing over $I\in \Sigma_n$ yields that $\mu(\Lambda_n)\leq 2^\ell (1-\epsilon)^n$, which implies \eqref{e-teta}.
\end{proof}

  \begin{proof}[Proof of  Theorem \ref{thm-1.2}(ii)]

Let $\ba\in \R^{md}$ and $W\in G(d,k)$. We need to show that  for $\mu$-a.e.~$x\in \Sigma'$,
\begin{equation*}
\label{e-q1}
\overline{\dim}_{\rm loc}\big((P_W\pi^\ba)_*\mu, P_W\pi^\ba x\big)\leq S(\mu,{\bf T},W, x).
\end{equation*}
For this purpose, it is enough to show that for every $\delta>0$,
\begin{equation}
\label{e-q1'}
\overline{\dim}_{\rm loc}\big((P_W\pi^\ba)_*\mu, P_W\pi^\ba x\big)\leq S(\mu,{\bf T},W,x)+\delta \quad \mbox{ for $\mu$-a.e.~$x\in \Sigma'$}.
\end{equation}
To this end, let $\delta>0$. Pick $\epsilon \in (0,1)$ such that
\begin{equation}
\label{e-to2}
\frac{\log (1-\epsilon)}{\log \alpha_+}<\delta,
\end{equation} where $\alpha_+:=\max\{\|T_i\|\colon i=1,\ldots, m\}$.

Set $$A_p:=\{x\in \Sigma'\colon p\leq S(\mu,W, x)<p+1\},  \quad p=0,\ldots, k-1.
$$ Since $\overline{\dim}_{\rm loc}\big((P_W\pi^\ba)_*\mu, P_W\pi^\ba x\big)\leq k$ for $\mu$-a.e.~$x\in \Sigma'$,  it suffices to show that \eqref{e-q1'} holds for $\mu$-a.e.~$x\in \bigcup_{p=0}^{k-1}A_p$.

Fix $p\in \{0,\ldots, k-1\}$. By Lemma \ref{LemJ}, there exists $A'_p\subset A_p$ with $\mu(A_p\backslash A_p')=0$ such that for each $x\in A_p'$,
\begin{equation}
\label{e-lemJ'}
(P_W\pi^\ba)_*\mu \Big( B\big(P_W\pi^\ba x, c \alpha_{p+1}(P_WT_{x|n})\big) \Big) \ge (1-\epsilon)^n \dfrac{ \mu([x|n])}{N_{p}(x|n)}\quad  \mbox{ for large enough $n$},
\end{equation}
where $N_{p}(x|n)$ is defined as in \eqref{e-Nxn}.

Now let $x\in A_p'$. Let $\gamma\in (0, p+1-S(\mu, {\bf T}, W, x))$. Recall that $$\lim_{n\to \infty} S_{n}(\mu,{\bf T},W,x)=S(\mu,{\bf T},W,x),$$ as proved in part (i) of the theorem. Hence there exists $N\in \N$  such that
\begin{equation}
\label{e-to1}
p\leq S_{n}(\mu,{\bf T},W,x)+\gamma<p+1
\end{equation}
for all $n\geq N$.
Observe that
\begin{align*}
\dfrac{ \mu([x|n])}{N_{p}(x|n)}=\dfrac{ \varphi^{S_{n}(\mu,{\bf T},W,x)}(P_WT_{x|n})}{\varphi^p(P_WT_{x|n})\alpha_{p+1}^{-p}(P_WT_{x|n})}
&\geq \dfrac{ \varphi^{S_{n}(\mu,{\bf T},W,x)+\gamma}(P_WT_{x|n})}{\varphi^p(P_WT_{x|n})\alpha_{p+1}^{-p}(P_WT_{x|n})}\\
&=\alpha_{p+1}^{S_{n}(\mu,{\bf T},W,x)+\gamma}(P_WT_{x|n}),
\end{align*}
where in the last equality we have used \eqref{e-to1}. Hence by \eqref{e-lemJ'},
\begin{align*}
\overline{\dim}_{\rm loc}((P_W\pi^\ba)_*\mu, P_W\pi^\ba x)&\leq \limsup_{n\to \infty} \frac{\log (P_W\pi^\ba)_*\mu \Big( B\big(P_W\pi^\ba x, c \alpha_{p+1}(P_WT_{x|n})\big) \Big)}{\log  \alpha_{p+1}(P_WT_{x|n}) }\\
& \leq \limsup_{n\to \infty} \frac{\log \left((1-\epsilon)^{n}\dfrac{ \mu([x|n])}{N_{p}(x|n)}\right)}{\log  \alpha_{p+1}(P_WT_{x|n}) }\\
&\leq \limsup_{n\to \infty} \left(S_{n}(\mu,{\bf T},W,x)+\gamma+ \frac{n\log (1-\epsilon)}{\log  \alpha_{p+1}(P_WT_{x|n}) }\right)\\
&\leq S(\mu,{\bf T},W,x)+\gamma+\delta,
\end{align*}
where in the last inequality we use that $\alpha_{p+1}(P_WT_{x|n})\leq (\alpha_+)^{n}$ and \eqref{e-to2}.
Since $\gamma$ is arbitrarily taken in $(0, p+1-S(\mu,{\bf T},W,x))$,
$$\overline{\dim}_{\rm loc}\big((P_W\pi^\ba)_*\mu, P_W\pi^\ba x\big)\leq S(\mu,{\bf T},W,x)+\delta.$$
That is,  \eqref{e-q1'} holds for every $x\in A_p'$,  so it holds  for $\mu$-a.e.~$x\in A_p$, as desired.
\end{proof}

Next we turn to the proof of part (iv) of Theorem \ref{thm-1.2}. We need several lemmas.
\begin{lem}[\cite{SY97}]
\label{lem-SY}
Let $\nu$ be a Borel probability measure on $\R^d$ with compact support and $x\in \R^d$. Then
$$
\underline{\dim}_{\rm loc}(\nu,x)=\sup\left\{s\geq 0\colon \int \|x-y\|^{-s}\; d\nu(y)<\infty\right\}.
$$
\end{lem}
\begin{proof}
The equality was first observed in \cite{SY97}. The reader is referred to \cite[Theorem 3.4.2]{BishopPeres17} for an implicit proof.
\end{proof}

For $x,y\in \Sigma$, let $x\wedge y$ denote the common initial segment of $x$ and $y$.
\begin{lem}
\label{lem-Smu}
Let $W\in G(d,k)$ and $x\in \Sigma'$. Then
\begin{equation}
\label{e-smu}
S(\mu,{\bf T},W,x)=\sup\left\{s\geq 0\colon \int \frac{1}{\varphi^s(P_WT_{x\wedge y})}\; d\mu(y)<\infty\right\}.
\end{equation}

\end{lem}
\begin{proof}
Notice that  both sides of \eqref{e-smu} are not greater than $k$, and that $$\lim_{n\to \infty}S_n(\mu, {\bf T},W, x)=S(\mu,{\bf T}, W,x)$$
by Theorem \ref{thm-1.2}(i).

Now we first show that if $s>S(\mu,{\bf T},W,x)$, then $\int \frac{1}{\varphi^s(P_WT_{x\wedge y})}\; d\mu(y)=\infty$.  This conclusion holds trivially whenever $s>k$,  so we only need to consider the case when $s\leq k$.  Assume that $k\geq s> S(\mu,{\bf T},W,x)$. Choose $\delta>0$ so that $s-\delta>S(\mu,{\bf T},W,x)$. Then there exists $n_0$ such that  $S_{n}(\mu,{\bf T},W,x)<s-\delta$ for all $n\geq n_0$, which implies that for $n\geq n_0$,
$$
\mu([x|n])=\varphi^{S_{n}(\mu,{\bf T},W,x)}(P_WT_{x|n})\geq \varphi^{s-\delta}(P_WT_{x|n})\geq \varphi^{s}(P_WT_{x|n})(1/\alpha_+)^{n\delta},
$$
where $\alpha_+=\max\{\|T_i\|\colon i=1,\ldots, m\}$, and we use Lemma \ref{lem-singinequality}(ii) in the last inequality. It follows that for $n\geq n_0$,
$$\int \frac{1}{\varphi^s(P_WT_{x\wedge y})} \;d\mu(y)\geq \int_{[x|n]} \frac{1}{\varphi^s(P_WT_{x|n})} \;d\mu(y)= \frac{\mu([x|n])}{\varphi^s(P_WT_{x|n})} \geq (1/\alpha_+)^{n\delta}.
$$
Letting $n\to \infty$ gives $\int \frac{1}{\varphi^s(P_WT_{x\wedge y})}\; d\mu(y)=\infty$.

Next we show that $\int \frac{1}{\varphi^s(P_WT_{x\wedge y})} \;d\mu(y)<\infty$ for $0\leq s<S(\mu,{\bf T},W,x)$. Choose $\delta>0$ such that $s+\delta<S(\mu,{\bf T},W,x)$. Then there exists $n_1$ such that $S_n(\mu,{\bf T},W,x)>s+\delta$ for all $n\geq n_1$. It follows that for $n\geq n_1$,
$$\mu([x|n])\leq \varphi^{S_{n}(\mu,{\bf T},W,x)}(P_WT_{x|n})<\varphi^{s+\delta}(P_WT_{x|n})\leq \varphi^{s}(P_WT_{x|n}) (\alpha_+)^{n\delta},$$
which implies, in particular, that $\mu(\{x\})=0$. Hence
\begin{align*}
\int \frac{1}{\varphi^s(P_WT_{x\wedge y})}\; d\mu(y)&=\sum_{n=0}^\infty \frac{1}{\varphi^s(P_WT_{x|n})}(\mu([x|n])-\mu([x|n+1]))\\
&\leq \sum_{n=0}^\infty \frac{\mu([x|n])}{\varphi^s(P_WT_{x|n})}\\
&\leq \sum_{n=0}^{n_1-1} \frac{\mu([x|n])}{\varphi^s(P_WT_{x|n})}+\sum_{n=n_1}^\infty(\alpha_+)^{n\delta}<\infty.
\end{align*}
This completes the proof.
\end{proof}

\begin{lem}
\label{lem-2.5}
Let $\rho>0$. If $s$ is non-integral with $0<s<d$ and $\|T_i\|<1/2$ for $1\leq i\leq m$, then there exists a number $c=c(\rho, s,T_1,\ldots, T_m)>0$ such that for every non-zero linear subspace $W$ of $\R^d$,
\begin{equation}
\label{e-Falconer}
\int_{B_\rho}\frac{d\ba}{\|P_W\pi^\ba x-P_W\pi^\ba y\|^s}\leq \frac{c}{\varphi^s(P_WT_{x\wedge y})}
\end{equation}
 for all distinct $x,y\in \Sigma$, where $B_\rho$ denotes the closed ball in $\R^{md}$  of radius $\rho$ centred at the origin.
\end{lem}
\begin{proof}
It is a slight and  trivial modification of  the proofs of \cite[Lemma 3.1]{Fal88} and \cite[Proposition 3.1]{Sol98}.
\end{proof}
Now we are ready to prove part (iv) of Theorem \ref{thm-1.2}.

\begin{proof}[Proof of Theorem \ref{thm-1.2}(iv)] Let $W\in G(d,k)$.  We first prove that for $\mathcal L^{md}$-a.e.~$\ba\in \R^{md}$,
    \begin{align}
    \label{e-A1}
    {\dim}_{\rm loc}\big((P_W\pi^\ba)_*\mu, P_W\pi^\ba x\big)&= S(\mu,{\bf T},W, x)\quad \mbox{ for $\mu$-a.e.~$x\in \Sigma$.}
    \end{align}

According to part (ii) of the theorem, we only need to show that for $\mathcal L^{md}$-a.e.~$\ba\in \R^{md}$,
$$
 \underline{\dim}_{\rm loc}\big((P_W\pi^\ba)_*\mu, P_W\pi^\ba x\big)\geq  S(\mu,{\bf T},W,x)\quad \mbox{ for $\mu$-a.e.~$x\in \Sigma'$}.
 $$

 To this end, we adapt the arguments in the proofs of \cite[Theorem 4.1]{HuntKaloshin1997} and \cite[Theorem 2.1(ii)]{FLM2023}.  Let $\rho>0$.  For a given non-integral $s\in (0,k)$ and a positive integer $N$, let $\Omega_N$ be the set of $x$ for which
 $$
 \int_\Sigma\frac{1}{\varphi^s(P_WT_{x\wedge y})}\; d\mu(y)<N.
 $$
 Notice that by Lemma \ref{lem-Smu}, the set of all $x\in \Sigma'$ for which $S(\mu,{\bf T},W,x)>s$ is contained in the union of $\Omega_N$ for $N\geq 1$.  Appying Fubini's theorem,
  \begin{align*}
 \int_{B_\rho}\int_{\Omega_N} \int_{\R^d} \frac{d(P_W\pi^\ba)_*\mu(z)}{\|P_W\pi^\ba x-z\|^s}\; d\mu(x) d\ba&= \int_{B_\rho}\int_{\Omega_N} \int_{\Sigma} \frac{1}{\|P_W\pi^\ba x-P_W\pi^\ba y\|^s}\; d\mu(y)d\mu(x) d\ba\\
&=\int_{\Omega_N}  \int_{\Sigma} \int_{B_\rho} \frac{1}{\|P_W\pi^\ba x-P_W\pi^\ba y\|^s}\;d\ba
 d\mu(y)d\mu(x) \\
 &\leq \int_{\Omega_N}  \int_{\Sigma}  \frac{c}{\varphi^s(P_WT_{x\wedge y})}\;d\mu(y)d\mu(x) \quad\mbox{(by \eqref{e-Falconer})}\\
 &\leq cN.
  \end{align*}
 It follows that for $\mathcal L^{md}$-a.e.~$\ba\in B_\rho$,
 $\displaystyle \int_{\Omega_N} \int_{\R^d} \frac{d(P_W\pi^\ba)_*\mu(z)}{\|P_W\pi^\ba x-z\|^s}\; d\mu(x)<\infty$ and hence
 $$
 \int_{\R^d} \frac{d(P_W\pi^\ba)_*\mu(z)}{\|P_W\pi^\ba x-z\|^s}<\infty \quad\mbox{ for $\mu$-a.e.~$x\in \Omega_N$}.
 $$
 Taking the union over $N$, we have for $\mathcal L^{md}$-a.e.~$\ba\in B_\rho$,
 $$
 \int_{\R^d} \frac{d(P_W\pi^\ba)_*\mu(z)}{\|P_W\pi^\ba x-z\|^s}<\infty \quad\mbox{ for $\mu$-a.e.~$x$ with $S(\mu,{\bf T},W,x)>s$}.
 $$
It follows from  Lemma \ref{lem-SY} that  for $\mathcal L^{md}$-a.e.~$\ba\in B_\rho$,
 $$
 \underline{\dim}_{\rm loc}\big((P_W\pi^\ba)_*\mu, P_W\pi^\ba x\big)\geq s \quad\mbox{ for $\mu$-a.e.~$x$ with $S(\mu,{\bf T},W,x)>s$}.
 $$
Thus we have shown that for all non-integral $s\in (0,k)$,
$$
\mu\left(\left\{x\in \Sigma'\colon S(\mu,{\bf T},W,x)>s> \underline{\dim}_{\rm loc}\big((P_W\pi^\ba)_*\mu, P_W\pi^\ba x\big)\right\}\right)=0
$$
for $\mathcal L^{md}$-a.e.~$\ba\in B_\rho$.  Taking the union over all non-integral rational $s$ in $(0,k)$, we conclude that
for $\mathcal L^{md}$-a.e.~$\ba\in B_\rho$,
$$\mu\left(\left\{x\in \Sigma'\colon S(\mu,{\bf T},W,x)> \underline{\dim}_{\rm loc}\big((P_W\pi^\ba)_*\mu, P_W\pi^\ba x\big)\right\}\right)=0.$$
Hence for $\mathcal L^{md}$-a.e.~$\ba\in \R^{md}$,
$$
 \underline{\dim}_{\rm loc}\big((P_W\pi^\ba)_*\mu, P_W\pi^\ba x\big)\geq  S(\mu,{\bf T},W,x)$$  for $\mu$-a.e.~$x\in \Sigma'$.

Next we prove that
\begin{equation}
     \label{e-AFFW**}
     \underline{S}(\mu,  {\bf T},W)=\overline{S}(\mu,{\bf T}, W)=\min\{k, \dim_{\rm LY}(\mu, {\bf T})\} \quad \mbox{ for $\gamma_{d,k}$-a.e.~$W\in G(d,k)$}.
     \end{equation}

 To see this, notice that $ \pi^\ba_*\mu$ is exact dimensional for each $\ba\in \R^{md}$; see \cite{Feng2023}.  By \cite[Theorem~1.7]{JPS07},
 \begin{equation}\label{e-A2}
 \dim_{\rm H} \pi^\ba_*\mu=\min\{d, \dim_{\rm LY}(\mu, {\bf T})\} \qquad \mbox{ for $\mathcal L^{md}$-a.e.~$\ba\in \R^{md}$}.
\end{equation}
Meanwhile it is known (see e.g.~\cite[Theorem 4.1]{HuntKaloshin1997}) that for a given exact dimensional Borel probability measure $\eta$ on $\R^d$, for $\gamma_{d,k}$-a.e.~$W\in G(d,k)$, $(P_W)_*\eta$ is exact dimensional with dimension given by
$$
\dim_{\rm H} (P_W)_*\eta=\min\{ k, \dim_{\rm H}\eta\}.
$$
Taking $\eta=\pi^\ba_*\mu$ and applying \eqref{e-A2}
 yield that for $\mathcal L^{md}$-a.e.~$\ba\in \R^{md}$ and  $\gamma_{d,k}$-a.e.~$W\in G(d,k)$,
$$
\dim_{\rm loc} \big((P_W\pi^\ba)_*\mu, P_W\pi^\ba x)\big)=\min\{k, \dim_{\rm LY}(\mu, {\bf T})\} \qquad \mbox{ for $\mu$-a.e.~$x\in \Sigma$}.
$$
Applying the Fubini theorem, we see that for $\gamma_{d,k}$-a.e.~$W\in G(d,k)$,
$$
\underline{\dim}_{\rm H} P_W(\pi^\ba_*\mu)=\overline{\dim}_{\rm H} P_W(\pi^\ba_*\mu)=\min\{k, \dim_{\rm LY}(\mu, {\bf T})\} \quad \mbox{ for $\mathcal L^{md}$-a.e.~$\ba\in \R^{md}$}.
$$
Combining this with \eqref{e-A1} yields  \eqref{e-AFFW**}.
\end{proof}

We  prove  Theorem \ref{thm-1.2}(iii) in the remaining part of this section. We first give several lemmas.

For $I\in \Sigma_n$ and $y=(y_i)_{i=1}^\infty\in \Sigma$, let $Iy$ denote the unique point $z=(z_i)_{i=1}^\infty \in \Sigma$ such that $z|n=I$ and $z_{n+i}=y_i$ for all $i\geq 1$.

\begin{lem}
\label{lem-density}
Assume that $\mu$ is a  fully supported and supermultiplicative Borel probability measure on $\Sigma=\{1,\ldots,m\}^\N$.  Let $A\subset \Sigma$ be measurable with $\mu(A)>0$. Then, for every positive integer $j$, there exist $I\in \Sigma_*$ and a measurable subset $A'$ of $A$ with $\mu(A')>0$  such that
$$
IJy\in A
$$
for all $J\in \Sigma_*$ with $|J|\leq j$ and $y\in A'$.
\end{lem}
\begin{proof}
Let $j\in \N$.  Since $\mu$ is supermultiplicative, there exists $C>0$ such that
$$
\mu([IJ])\geq C\mu([I])\mu([J])\qquad \mbox{ for all }I,J\in \Sigma_*.
$$
A standard approximation argument shows that for every $I\in \Sigma_*$ and every Borel measurable set $E\subset \Sigma$,
\begin{equation}
\label{e-ZZ2}
\mu\left([I]\cap \sigma^{-|I|}(E)\right)\geq C\mu([I])\mu(E).
\end{equation}

Set
\begin{equation}
\label{e-epsilon}
\epsilon=C^2\mu(A)(j+1)^{-1}m^{-j}\min_{J\in \Sigma_j}\mu([J]).
\end{equation}
Since $\mu$ is fully supported, we have $\epsilon>0$. By the Borel density lemma, for $\mu$-a.e.~$x\in A$,
$$
\lim_{n\to \infty}\frac{\mu([x|n]\cap A)}{\mu([x|n])}=1.
$$
Hence we can find  $x\in A$ and $n\in \N$ such that
$$
\frac{\mu([x|n]\cap A)}{\mu([x|n])}> 1-\epsilon,
$$
which implies that
\begin{equation}
\label{e-ZZ3}
\frac{\mu\big([x|n]\cap (\Sigma\backslash A)\big)}{\mu([x|n])}< \epsilon.
\end{equation}

Now set $I=x|n$. Write
$$
\Omega_J=\{z\in \Sigma\colon IJz\in A\} \quad \mbox{for  $J\in \bigcup_{p=0}^j\Sigma_p$}.
$$
Clearly, if $z\in \Sigma\backslash \Omega_J$, then $IJz\in \Sigma\backslash A$.
Hence for each $J\in \bigcup_{p=0}^j\Sigma_p$,
\begin{equation}
\label{e-ZZ4}
[IJ]\cap \sigma^{-n-|J|}(\Sigma\backslash \Omega_J)\subset \Sigma\backslash A.
\end{equation}
Now we claim that
\begin{equation}
\label{e-ZZ1}
\mu(\Sigma\backslash \Omega_J)< (j+1)^{-1}m^{-j}\mu(A)\quad \mbox{ for all } J\in \bigcup_{p=0}^j\Sigma_p.
\end{equation}
Suppose on the contrary that \eqref{e-ZZ1} does not hold, that is, there exists   $J\in \bigcup_{p=0}^j\Sigma_p$ such that $\mu(\Sigma\backslash \Omega_J)\geq  (j+1)^{-1}m^{-j}\mu(A)$. Then
\begin{align*}
\mu\big([I]\cap (\Sigma\backslash A)\big)&\geq \mu\left( [IJ]\cap \sigma^{-n-|J|}(\Sigma\backslash \Omega_J)\right)\;\quad\qquad\mbox{ (by \eqref{e-ZZ4})} \\
&\geq C\mu( [IJ])\mu(\Sigma\backslash \Omega_J)\;\;\qquad\qquad\qquad\mbox{ (by \eqref{e-ZZ2})} \\
&\geq C^2\mu( [I])\mu( [J])\mu(\Sigma\backslash \Omega_J)\;\qquad\qquad\mbox{ (by \eqref{e-ZZ2})} \\
&\geq C^2\mu( [I])\mu( [J])(j+1)^{-1}m^{-j}\mu(A)\\
&\geq \epsilon \mu( [I])\;\;\quad\qquad\qquad\qquad\qquad\qquad\mbox{ (by \eqref{e-epsilon})},
\end{align*}
which contradicts \eqref{e-ZZ3}. This proves \eqref{e-ZZ1}.

Notice that $\#\left(\bigcup_{p=0}^j\Sigma_p\right)< (j+1)m^j$. By \eqref{e-ZZ1},
\begin{align*}
\mu(\Sigma\backslash A)&+\sum_{J\in \Sigma_*\colon |J|\leq j}\mu(\Sigma\backslash \Omega_J)\\
&<1-\mu(A)+\#(\{J\in \Sigma_*\colon |J|\leq j\})\cdot(j+1)^{-1}m^{-j}\mu(A)\\
&< 1.
\end{align*}
It follows that
\begin{align*}
\mu\left(A\cap \left(\bigcap_{J\in \Sigma_*\colon |J|\leq j}\Omega_J\right)\right)&=1-\mu\left((\Sigma\backslash A)\cup \left(\bigcup_{J\in \Sigma_*\colon |J|\leq j}(\Sigma\backslash \Omega_J)\right)\right)\\
 &\geq 1-\mu(\Sigma\backslash A)-\sum_{J\in \Sigma_*\colon |J|\leq j}\mu(\Sigma\backslash \Omega_J)>0.
\end{align*}
Set $A'=A\cap \left(\bigcap_{J\in \Sigma_*\colon |J|\leq j}\Omega_J\right)$.  Then $\mu(A')>0$ and
$IJy\in A$ for all $J\in \Sigma_*$ with $|J|\leq j$ and $y\in A'$.
\end{proof}

For a set $E\subset \R^d$, let  ${\rm span}(E)$ denote the smallest linear subspace of $\R^d$ that contains
$E$.

\begin{lem}
\label{lem-finite}
Let $W$ be a nonzero linear subspace of $\R^d$. Let $M_1,\ldots, M_m$ be real $d\times d$ matrices.
Set $V={\rm span}\left(\bigcup_{I\in \Sigma_*}M_I(W)\right)$, where $M_I:=M_{i_1}\cdots M_{i_n}$ for $I=i_1\ldots i_n$, and we take the convention that $M_{\emptyset}={\rm Id}$, where $\emptyset$ stands for the empty word.    Then $$
V={\rm span}\left(\bigcup_{I\in \Sigma_*\colon |I|\leq j}M_I(W)\right)
$$
for all $j\geq d-1$.
Moreover, $V=\R^d$ if $\{M_i\}_{i=1}^m$ is irreducible.
\end{lem}
\begin{proof}
Define $W_0=W$ and $$W_j={\rm span}\left(\bigcup_{I\in \Sigma_*\colon |I|\leq j}M_I(W)\right)$$
for $j\geq 1$.  Clearly
\begin{equation}
\label{e-eww1}
W_{j+1}={\rm span} \left(W_j\cup\left(\bigcup_{i=1}^m M_i(W_j)\right)\right) \quad \mbox{ for }j\geq 0,
\end{equation}
and $W_0\subset W_1\subset W_2\subset \cdots\subset V$.
 Thus, $$\dim(W_0)\leq \dim (W_1)\leq  \dim (W_2)\leq  \cdots \leq \dim (V).$$
Since $\dim W_j\leq d$ for all $j\geq 0$, there exists  $j_0\in \{0,1,\ldots, d-1\}$ such that
\begin{equation}
\label{e-eww}
\dim(W_{j_0+1})=\dim (W_{j_0});
\end{equation}
otherwise, $\dim W_d>\dim W_{d-1}>\cdots >\dim W_0\geq 1$, and consequently,  $\dim W_d>d$, leading to a contradiction.  Since $W_{j_0+1}\supset W_{j_{0}}$, by \eqref{e-eww}, we have $W_{j_0+1}=W_{j_0}$. Then applying \eqref{e-eww1}, we see that $W_{j}=W_{j_0}$ for all $j\geq j_0$.
Since $$\bigcup_{I\in \Sigma_*}M_I(W)\subset \bigcup_{j= j_0}^\infty W_{j}=W_{j_0},$$
 it follows that  $W_{j_0}\subset V\subset {\rm span}(W_{j_0})=W_{j_0}$, and consequently, $V=W_{j_0}=W_{d-1}$.  From the definition of $V$, we see that $M_iV\subset V$ for all $1\leq i\leq m$. Hence $V=\R^d$ if  $\{M_i\}_{i=1}^m$ is irreducible.
\end{proof}
\bigskip

\begin{lem}
\label{lem-irreducible}
Let $A_1,\ldots, A_m\in {\rm Mat}_d(\R)$. Then
$$
\{A_i\}_{i=1}^m \mbox{ is reducible}\Longleftrightarrow  \{A_i^*\}_{i=1}^m \mbox{ is reducible}.
$$
\end{lem}
\begin{proof}
The result is standard. For the convenience of the reader, we include a proof.

Since $A_i^{**}=A_i$, by symmetry it is enough to prove the direction ``$\Longrightarrow$''. To this end, suppose that $\{A_i\}_{i=1}^m$ is reducible, i.e., there exists a proper nonzero subspace $W$ of $\R^d$  such that $A_i(W)\subset W$ for all $1\leq i\leq m$. Let $W^\perp$ denote the orthogonal complement of $W$. Clearly, $W^\perp$ is also a proper nonzero subspace of $\R^d$. Let $v\in W^\perp$ and $1\leq i\leq m$.  For any $u\in W$, since $A_iu\in W$, it follows that
$$
\langle u, A_i^*v\rangle=\langle A_iu, v\rangle=0.
$$
Hence $ A_i^*v\in W^\perp$.  This proves $A_i^*(W^\perp)\subset W^\perp$. So $\{A_i^*\}_{i=1}^m \mbox{ is reducible}$.
\end{proof}

Now we are ready to prove Theorem \ref{thm-1.2}(iii).

\begin{proof}[Proof of Theorem \ref{thm-1.2}(iii)]
Let $\mu$ be a fully supported, ergodic, and super-multiplicative measure on $\Sigma$. We first show that
\begin{equation}
\label{e-Z0}
\underline{S}(\mu, {\bf T},W)=\overline{S}(\mu,{\bf T}, W)\qquad \mbox{ for all }W\in G(d,k).
\end{equation}
To prove this, by definition, it is enough to show that for each $W\in G(d,k)$ and $s\in [0,k]$, the sequence
$$
\frac{1}{n}\log \varphi^s(T_{x|n}^*P_W)
$$
 converges pointwisely to a constant  for $\mu$-a.e.~$x$. Since
$$
\varphi^s(M)=\left(\varphi^{\lfloor s\rfloor}(M)\right)^{1+\lfloor s\rfloor -s}\left(\varphi^{\lfloor s\rfloor+1}(M)\right)^{s-\lfloor s\rfloor}
$$
for each $0\leq s\leq d$ (see Lemma \ref{lem-phis}),
it is sufficient to prove the aforementioned statement in the case where  $s$ is an integer.
For this purpose, fix $W\in G(d,k)$ and $q\in \{1,\ldots,k\}$. We will show that there exists $\lambda\in \R$ such that
\begin{equation}
\label{e-Z1}
\lim_{n\to\infty}\frac{1}{n} \log \|(T^*_{x|n})^{\wedge q}P_{W^{\wedge q}}\|=\lambda\qquad \mbox{ for $\mu$-a.e.~$x\in \Sigma$}.
\end{equation}

To see \eqref{e-Z1}, we define a linear subspace $X$ of $(\R^d)^{\wedge q}$ by
\begin{equation}
\label{e-X1}
X={\rm span}\left(\bigcup_{I\in \Sigma_*}(T_I^*)^{\wedge q}(W^{\wedge q})\right).
\end{equation}
It is easy to see that $(T_i^*)^{\wedge q}(X)=X$ for all $1\leq i\leq m$. Set
$$
\tau=\dim \left((\R^d)^{\w q}\right)=\binom{d}{q}.
$$
By Lemma \ref{lem-finite} (where we replace $\R^d$ with $(\R^d)^{\w q}$),
$$
X={\rm span}\left(\bigcup_{J\in \Sigma_*\colon |J|\leq \tau}(T_J^*)^{\wedge q}(W^{\wedge q})\right).
$$

For $i=1,\ldots, m$, let $M_i=(T_i^*)^{\wedge q}|_X$ be the restriction of $(T_i^*)^{\wedge q}$ on $X$.  Then $M_i$ are invertible linear transformations on $X$. Let $\widetilde{\lambda}_1>\widetilde{\lambda}_2>\cdots>\widetilde{\lambda}_{\widetilde{r}}$ be the Lyapunov exponents  of the matrix cocycle $x\mapsto M_{x_1}$ with respect to $\mu$, and let
\begin{equation}
\label{e-xfiltration}
X=\widetilde{V}_0(x)\supsetneqq\widetilde{V}_1(x)\supsetneqq \cdots \supsetneqq \widetilde{V}_{\widetilde{r}}(x)=\{0\},\quad x\in \Sigma'',
\end{equation}
be the corresponding Oseledets filtration, where $\Sigma''$ is a $\sigma$-invariant Borel subset of $\Sigma$ with $\mu(\Sigma'')=1$. By Remark \ref{rem-3.2}(ii),
\begin{equation}
\label{e-einvariant}
M_{x_1}\widetilde{V}_1(x)=\widetilde{V}_1(\sigma x) \quad \mbox{ for every }x\in \Sigma''.
\end{equation}
 We claim that
\begin{equation}
\label{e-Z2}
W^{\wedge q}\cap \left(\widetilde{V}_0(x)\backslash \widetilde{V}_1(x)\right)\neq \emptyset \qquad \mbox{ for $\mu$-a.e.~$x\in \Sigma''$},
\end{equation}
which implies \eqref{e-Z1} (in which we take $\lambda=\widetilde{\lambda}_1$).

Suppose on the contrary that \eqref{e-Z2} does not hold. Then there exists a measurable set $E\subset \Sigma''$ with $\mu(E)>0$ such that
\begin{equation}
\label{e-wwedge}
W^{\wedge q}\subset \widetilde{V}_1(x) \qquad \mbox{ for all } x\in E.
\end{equation}
Define
\begin{equation}
\ell=\sup\left\{\dim Y\colon Y \mbox{ is a subspace of $X$ such that }\mu\{x\in E\colon Y\subset \widetilde{V}_1(x)\}>0\right\}.
\end{equation}
Clearly, $\ell\geq \dim W^{\wedge q}$, and the supremum is attained at a subspace $Y_0$ of $X$. Set
$$
A=\left\{x\in E\colon Y_0\subset \widetilde{V}_1(x)\right\}.
$$
Since $\mu$ is fully supported and supermultiplicative, by Lemma \ref{lem-density},  there exist $I\in \Sigma_*$ and $A'\subset A$ with $\mu(A')>0$    such that
$$
IJy\in A
$$
for all $J\in \Sigma_*$ with $|J|\leq \tau$ and $y\in A'$.

Suppose that $|I|=n$ and $I=i_1\ldots i_n$. Let $y\in A'$.  Then, for every $J=j_1\ldots j_p\in \Sigma_p$ with $p\leq \tau$,  since $IJy\in A$, it follows that
$$
Y_0\subset \widetilde{V}_1(IJy)=\widetilde{V}_1(i_1\ldots i_n j_1\ldots j_py).
$$
 Since $M_{j_p\ldots j_1i_n\ldots i_1}\widetilde{V}_1(IJy)=\widetilde{V}_1(y)$ by \eqref{e-einvariant}, it follows that
$$
\widetilde{V}_1(y)\supset M_{j_p\ldots j_1i_n\ldots i_1}(Y_0)=M_{j_p\ldots j_1}(M_{i_n\ldots i_1}(Y_0)).
$$
Hence, let $Y_1:=M_{i_n\ldots i_1}(Y_0)$ and by taking the union over all $J\in \Sigma_*$ with $|J|\leq \tau$, we obtain that
$$
\widetilde{V}_1(y)\supset \bigcup_{J\in \Sigma_*\colon |J|\leq \tau}M_{J}(Y_1).
$$
Since $\widetilde{V}_1(y)$ is a  subspace of $X$, we have
$$
\widetilde{V}_1(y)\supset {\rm span}\left(\bigcup_{J\in \Sigma_*\colon |J|\leq \tau} M_J(Y_1)\right)=: Y_2,
$$
By Lemma  \ref{lem-finite}, $Y_2={\rm span}\left(\bigcup_{J\in \Sigma_* } M_J(Y_1)\right)$ and hence $M_i(Y_2)=Y_2$ for all $1\leq i\leq m$.
Meanwhile, since $y\in A'\subset A\subset E$, by \eqref{e-wwedge}, we have $\widetilde{V}_1(y)\supset W^{\w q}$ as well. Hence
$$
\widetilde{V}_1(y)\supset {\rm span}(Y_2 \cup W^{\w q}) \quad \mbox{ for all }y\in A'.
$$
Since $\mu(A')>0$, by the maximality of $\ell$, we have $\dim\left({\rm span}(Y_2 \cup W^{\w q})\right)\leq \ell$. However, as $Y_2\supset Y_1=M_{i_n\ldots i_1}(Y_0)$, we have
$$\dim Y_2\geq \dim Y_1=\dim Y_0=\ell.$$
 Hence
 $$
 \dim\left({\rm span}(Y_2 \cup W^{\w q})\right)= \dim Y_2.
 $$
Since $Y_2$ is a subspace of $X$, the above equality implies that $W^{\w q}\subset Y_2$. Recall that $Y_2$ is $M_i$-invariant for all $1\leq i\leq m$. It follows that
$$
Y_2\supset {\rm span}\left(\bigcup_{J\in \Sigma_*} M_J(W^{\w q})\right)=X.
$$
Hence for all $y\in A'$,
$$
\widetilde{V}_1(y)\supset Y_2\supset X=\widetilde{V}_0(y),
$$
leading to a contradiction.
 This proves \eqref{e-Z2}, and consequently, \eqref{e-Z1}.

 Next we assume that $\{T_i^{\wedge q}\}_{i=1}^m$ is irreducible for some integer $q$ with $\ell'\leq q\leq k$, where  $\ell'$ is the smallest integer not less than $\min\{k, \dim_{\rm LY}(\mu, {\bf T})\}$. By Lemmas \ref{lem-F-3.1}(iii) and  \ref{lem-irreducible},  $\{(T_i^*)^{\wedge q}\}_{i=1}^m$ is also irreducible.  Let $W\in G(d,k)$ and let $X$ be defined as in \eqref{e-X1}.
 Since $X$ is $(T_i^*)^{\wedge q}$-invariant, it follows that $X=(\R^d)^{\wedge q}$.  As was proved above, we have
 \begin{equation}
 \label{e-Y1}
 \lim_{n\to\infty}\frac{1}{n} \log \|(T^*_{x|n})^{\wedge q}P_{W^{\wedge q}}\|=\widetilde{\lambda}_1\qquad \mbox{ for $\mu$-a.e.~$x\in \Sigma$},
 \end{equation}
 where $\widetilde{\lambda}_1$ is the largest Lyapunov exponent of the matrix cocycle $x\mapsto (T_{x_1}^*)^{\wedge q}$ with respect to $\mu$.

 Recall that by Remark \ref{rem-3.2}(iii),
 \begin{equation}
 \label{e-Y2}
 \widetilde{\lambda}_1=\Lambda_1+\cdots+\Lambda_q,
 \end{equation}
 where $\Lambda_1\geq \cdots\geq \Lambda_m$ are the Lyapunov exponents (counting multiplicity) of the matrix cocycle $x\mapsto (T_{x_1}^*)$ with respect to $\mu$.
 Meanwhile by Proposition \ref{pro-key},
  \begin{equation}
  \label{e-Y3}
 \lim_{n\to\infty}\frac{1}{n} \log \|(T^*_{x|n})^{\wedge q}P_{W^{\wedge q}}\|=\Lambda_{p_1(x)}+\cdots+\Lambda_{p_q(x)},  \mbox{ for $\mu$-a.e.~$x\in \Sigma$},
 \end{equation}
 where $(p_1(x),\ldots, p_k(x))$ is the pivot position vector of $W$ with respect to the basis ${\bf v}(x)=\{v_i(x)\}_{i=1}^d$, where  ${\bf v}(x)$ is defined as in Theorem \ref{thm:oseledets}. Since $p_i(x)\geq i$ for $1\leq i\leq q$, combining \eqref{e-Y1}, \eqref{e-Y2} and \eqref{e-Y3} yields that for $\mu$-a.e.~$x\in \Sigma$,
 \begin{equation}
  \label{e-Y4}
  \Lambda_{p_i(x)}=\Lambda_i \qquad\mbox{ for }i=1,\ldots, q.
 \end{equation}

 %By \eqref{e-Y4} and Proposition \ref{pro-key}, for every $0\leq s\leq q$,
 % \begin{equation}
  %\label{e-Y5}
 %\lim_{n\to \infty} \frac{1}{n}\log \varphi^s(T_{x|n}^*P_W)=\left(\sum_{i=1}^{\lfloor s\rfloor} \Lambda_i\right) + (s-\lfloor s\rfloor)\Lambda_{\lfloor s\rfloor+1} \qquad\mbox{ for $\mu$-a.e.~}x\in \Sigma'.
%\end{equation}

 Now set $s=\min\{k, \dim_{\rm LY}(\mu, {\bf T})\}$. Then $s\leq q\leq k$. By the definition of $\dim_{\rm LY}(\mu, {\bf T})$,
either
 $$
 s=k\quad \mbox{ and } \quad h_\mu(\sigma)+\sum_{i=1}^{k} \Lambda_i\geq 0,
 $$
or
 $$
 0\leq s<k\quad \mbox{ and } \quad h_\mu(\sigma)+\sum_{i=1}^{\lfloor s\rfloor} \Lambda_i + (s-\lfloor s\rfloor)\Lambda_{\lfloor s\rfloor+1}= 0.
 $$
 By \eqref{e-Y4} and Lemma \ref{lem-smuw}, we see that in both cases, $S(\mu,{\bf T}, W,x)=s$ for $\mu$-a.e.~$x\in \Sigma'$, and consequently,
 $$\underline{S}(\mu,{\bf T}, W)=\overline{S}(\mu, {\bf T},W)=s=\min\{k, \dim_{\rm LY}(\mu, {\bf T})\}.
 $$
 This completes the proof of Theorem \ref{thm-1.2}(iii).
\end{proof}

\begin{rem}
\label{rem-5.8}
It is worth pointing out that the assumption of $\mu$ being fully supported in Theorem \ref{thm-1.2}(iii) can be dropped. More precisely, part (iii) of Theorem \ref{thm-1.2} can be strengthened  as follows:
\begin{itemize}
\item[(iii)']
Assume  additionally that $\mu$ is  supermultiplicative. Then $$\underline{S}(\mu,{\bf T}, W)=\overline{S}(\mu, {\bf T},W)$$ for all $W\in G(d,k)$. Let ${\mathcal A}:=\{1\leq i\leq m\colon \mu([i])>0\}$. If furthermore $\{T_i^{\wedge q}\}_{i\in \mathcal A}$ is irreducible for some integer $q$ such that $\ell'\leq q\leq k$, where  $\ell'$ is the smallest integer not less than $\min\{k, \dim_{\rm LY}(\mu, {\bf T})\}$, then
\begin{equation*}
\underline{S}(\mu, {\bf T},W)=\overline{S}(\mu, {\bf T},W)=\min\{k,\dim_{\rm LY}(\mu, {\bf T})\}
\end{equation*}
 for all $W\in G(d,k)$.
 \end{itemize}
The proof remains unchanged if we consider the IFS $\{T_ix+a_i\}_{i\in \mathcal A}$ instead.
\end{rem}

\section{Proof of  Theorem \ref{thm-1.1}}
\label{S-6}

We prove parts (i), (iii) and (ii) of Theorem \ref{thm-1.1} separately.

\begin{proof}[Proof of  Theorem \ref{thm-1.1}(i)]
Let $\ell$ be the smallest integer not less than $\min\{k, \dim_{\rm AFF}({\bf T})\}$. Then $\ell\leq k$.   By Lemma \ref{lem-Affd}(i), $\dim_{\rm AFF} ({\bf T}, W)\leq \ell$ for all $W\in G(d,k)$.  Below we divide the remaining proof into two steps.

{\it Step 1. For every integer $q$ with $\ell\leq q\leq k$,
\begin{equation}
\label{e-afftw}
\#\{\dim_{\rm AFF}({\bf T}, W)\colon W\in G(d,k)\}\leq \binom{d}{q}-\binom{k}{q}+1.
\end{equation}
}

To prove this inequality, we fix an integer $q$ with $\ell\leq q\leq k$.  Suppose on the contrary that there exist $W_1,\ldots, W_\tau\in G(d,k)$, with $\tau=\binom{d}{q}-\binom{k}{q}+2$,  such that
$$
\dim_{\rm AFF}({\bf T}, W_1)>\dim_{\rm AFF}({\bf T}, W_2)>\cdots>\dim_{\rm AFF}({\bf T}, W_\tau).
$$
Write $s_i:=\dim_{\rm AFF}({\bf T}, W_i)$ for $1\leq i\leq \tau$. Then
\begin{equation}
k\geq \ell\geq s_1>s_2>\cdots>s_\tau.
\end{equation}

 By Lemma \ref{lem-Affd}(iii),
$$
P({\bf T}, W_1,s_1)\geq 0, \quad \mbox{ and } \quad P({\bf T}, W_i, s_i)=0  \quad \mbox{ for all }2\leq i\leq \tau.
$$
It follows from Lemma \ref{lem-Pws}(ii) that
\begin{equation}
\label{e-3.3a}
P({\bf T}, W_j, s_i)<P({\bf T}, W_j, s_j)= 0\leq P({\bf T}, W_i, s_i)\quad  \mbox{ for all }1\leq i<j\leq \tau.
\end{equation}
Below we first make the following claim.
\medskip

\noindent {\bf Claim}. {\it For each $i\in \{1,\ldots, \tau-1\}$, there exist a word $K_i\in \Sigma_*$, and a linear subspace $H_i$ of $(\R^d)^{\wedge q}$ with $\dim H_i=\binom{d}{q}-1$, such that
\begin{equation}
\label{e-3.4aa}
(T_{K_i}^*W_i)^{\wedge q} \not\subset H_i
\end{equation}
and
\begin{equation}
\label{e-3.5a}
(T^*_KW_j)^{\wedge q} \subset H_i\quad  \mbox{ for all } i<j\leq \tau \mbox{ and }K\in \Sigma_*.
\end{equation}
}

Before proving the above claim, we first use it to derive a contradiction. Write for brevity
$$V_i=(T_{K_i}^*W_i)^{\wedge q}, \quad i=1,\ldots, \tau,$$
and
$$G_i=\bigcap_{p=1}^iH_p, \quad i=1,\ldots, \tau-1.$$
 Clearly, $G_i$,  $i=1,\ldots, \tau-1$, are linear subspaces of $(\R^d)^{\wedge q}$ so that
\begin{equation}
\label{e-H1}
H_1=G_1\supset G_2\supset \cdots  \supset G_{\tau-1}.
\end{equation}
By \eqref{e-3.4aa}-\eqref{e-3.5a}, $V_i\not\subset H_i$ and $V_j\subset H_i$ for all $1\leq i<j\leq \tau$. It follows that for each $1\leq i\leq \tau-1$,
$$V_{i+1}\subset G_i\quad \mbox{   but }\quad  V_{i+1}\not\subset G_{i+1},$$
which implies that  $G_i\neq G_{i+1}$.   Combining this with \eqref{e-H1} yields that
$$
\dim H_1=\dim G_1>\dim G_2>\cdots >\dim G_{\tau-1},
$$
and thus $\dim G_{\tau-1}\leq \dim H_1-(\tau-2)=\binom{d}{q}-\tau+1$. However, since $ G_{\tau-1}\supset V_\tau$, one has
$$
\dim G_{\tau-1}\geq \dim V_\tau=\binom{k}{q}.
$$
It follows that $\binom{k}{q}\leq \binom{d}{q}-\tau+1$, that is, $\tau \leq \binom{d}{q}-\binom{k}{q}+1$. It leads to a contradiction.

Now we turn to the proof of the claim. Fix $i\in \{1,\ldots, \tau-1\}$, and   let $\mu$ be an ergodic equilibrium measure for the subadditive potential $\{\log \psi_{W_i}^{s_i}(\cdot|n)\}_{n=1}^\infty$. That is, $\mu$ is an ergodic $\sigma$-invariant measure such that
\begin{equation}
\label{e-3.4a}
h_{\mu}(\sigma)+\Theta(\psi_{W_i}^{s_i}, \mu)=P({\bf T}, W_i, s_i),
\end{equation}
where
$$
\Theta\left(\psi_{W_i}^{s_i},\mu\right)=\lim_{n\to \infty} \frac1n\int \log \psi_{W_i}^{s_i}(x|n)\; d\mu(x).
$$
For each $j$ with $i<j\leq \tau$, applying Theorem \ref{thm:subadditive-VP} to the subadditive potential $\{\log \psi_{W_j}^{s_i}(\cdot|n)\}_{n=1}^\infty$ gives
 $$P({\bf T}, W_j, s_i)\geq  h_{\mu}(\sigma)+\Theta(\psi_{W_j}^{s_i}, \mu),$$
where
$$
\Theta(\psi_{W_j}^{s_i},\mu)=\lim_{n\to \infty} \frac1n\int \log \psi_{W_j}^{s_i}(x|n)\; d\mu(x).
$$
This together with \eqref{e-3.3a} and \eqref{e-3.4a} yields that
\begin{equation}
\label{e-thetaij}
\Theta(\psi_{W_i}^{s_i}, \mu)>\Theta(\psi_{W_j}^{s_i}, \mu)\quad \mbox{ for all }j\in \{i+1,\ldots,  \tau\}.
\end{equation}

Let $\Sigma'$, $r$, $\Lambda_1,\ldots, \Lambda_d$, $\bigoplus_{j=1}^r E_j(x)$ ($x\in \Sigma'$) be given as in Theorem \ref{thm:oseledets}, and also let  ${\bf v}(x)=\{v_1(x),\ldots, v_d(x)\}$ be a measurable ordered basis adapted to the splitting $\bigoplus_{j=1}^r E_j(x)$, $x \in \Sigma'$.

By  Proposition \ref{pro-key},  Corollary \ref{cor-2.8} and \eqref{e-thetaij},  there exist a measurable $A\subset \Sigma'$ with $\mu(A)>0$ and $K_i\in \Sigma_*$ such that
for each $x\in A$,
\begin{align}
&\lim_{n\to \infty} \frac{1}{n}\log \varphi^{s_i}(T_{x|n}^*P_{T_{K_i}^*W_i})=\Theta(\psi_{W_i}^{s_i}, \mu), \label{e-3.7'}\\
&\lim_{n\to \infty} \frac{1}{n}\log \varphi^{s_i}(T_{x|n}^*P_{T_{K}^*W_j})\leq \Theta(\psi_{W_j}^{s_i}, \mu)<\Theta(\psi_{W_i}^{s_i}, \mu), \label{e-3.8'}
\end{align}
 for all $K\in \Sigma_*$ and   $j\in\{i+1,\ldots,  \tau\}$, and moreover, for each $W\in G(d,k)$,
\begin{equation}
\label{e-3.7}
\lim_{n\to \infty} \frac{1}{n}\log \varphi^{s_i}(T_{x|n}^*P_W)=\sum_{j=1}^{\lfloor s_i \rfloor} \Lambda_{p_j(W,x)}+(s_i-\lfloor s_i \rfloor)\Lambda_{p_{\lfloor s_i \rfloor+1}(W,x)},
\end{equation}
where $\left(p_1(W,x),\ldots, p_k(W,x)\right)$ is the pivot position vector of $W$ with respect to the ordered basis ${\bf v}(x)=\{v_i(x)\}_{i=1}^d$.

Fix $x\in A$. Write for brevity $v_j=v_j(x)$ and $p_j=  p_j(T_{K_i}^*W_i,x)$ for $j=1,\ldots, k$.  Define
\begin{align*}
H_i=\; & {\rm span}\left\{v_{j_1}\wedge \cdots \wedge v_{j_q}\colon 1\leq j_1<\cdots<j_q\leq d \mbox{ and }\right.\\
&\qquad \qquad  \qquad \qquad \qquad \left. (j_1,\ldots, j_q)\neq (p_1,\ldots, p_q) \right\}.
\end{align*}
Clearly, $\dim H_i=\binom{d}{q}-1$. By Lemma \ref{lem-2.5'}(i),  $(T_{K_i}^*W_i)^{\wedge q}\not\subset H_i$.  It remains to show that $(T_{K}^*W_j)^{\wedge q}\subset H_i$ for each $K\in \Sigma_*$ and  $j$ with $i<j\leq \tau$.  Suppose on the contrary that there exist $K\in \Sigma_*$ and  $j$ with $i<j\leq \tau$ such that $(T_{K}^*W_j)^{\wedge q}\not\subset H_i$.
Then by Lemma \ref{lem-2.5'}(ii),
$$p_1(T_{K}^*W_j,x)\leq p_1,\quad \ldots, \quad p_q(T_{K}^*W_j,x)\leq p_q.$$
Keep in mind that $q\geq \ell\geq s_i$. Applying \eqref{e-3.7},
\begin{align*}
\lim_{n\to \infty} \frac{1}{n}\log \varphi^{s_i}(T_{x|n}^* P_{T^*_KW_j})&=\sum_{u=1}^{\lfloor s_i \rfloor} \Lambda_{p_u(T^*_KW_j,x)}+(s_i-\lfloor s_i \rfloor)\Lambda_{p_{\lfloor s_i \rfloor+1}(T^*_KW_j,x)}\\
&\geq \sum_{u=1}^{\lfloor s_i \rfloor} \Lambda_{p_u}+(s_i-\lfloor s_i \rfloor)\Lambda_{p_{\lfloor s_i \rfloor+1}}\\
&=\lim_{n\to \infty} \frac{1}{n}\log \varphi^{s_i}(T_{x|n}^* P_{T^*_{K_i}W_i}).
\end{align*}
However, by  \eqref{e-3.7'} and \eqref{e-3.8'},
$$
\lim_{n\to \infty} \frac{1}{n}\log \varphi^{s_i}(T_{x|n}^* P_{T^*_KW_j})<\lim_{n\to \infty} \frac{1}{n}\log \varphi^{s_i}(T_{x|n}^* P_{T^*_{K_i}W_i}),
$$
leading to a contradiction. This proves  the claim, and consequently, inequality \eqref{e-afftw}.

{\it Step 2. Assume that  $\{T_i^{\wedge q}\colon i=1,\ldots, m\}$ is irreducible for some $\ell\leq q\leq k$.  Then
     $$
     \dim_{\rm AFF}({\bf T}, W)=\min\{k, \dim_{\rm AFF}({\bf T})\}
     $$
    for all $W\in G(d,k)$.
    }

Write $s=\min\{k, \dim_{\rm AFF}({\bf T})\}$. Then $k\geq q\geq s$.  Define
$$
P({\bf T}, s)=\lim_{n\to \infty} \frac1n\log \sum_{I\in \Sigma_n}\varphi^s(T_I^*).
$$
Clearly, $P({\bf T}, s)$ is the topological pressure of the subadditive potential $\{\log \varphi^s(T^*_{\cdot|n})\}_{n=1}^\infty$; see Lemma \ref{lem-singinequality}(i) and Section \ref{S-subadditive}.
Recall that $\dim_{\rm AFF}({\bf T})$ is defined as in \eqref{e-aff}. Since $\dim_{\rm AFF}({\bf T})\geq s$, it follows that
 $P({\bf T}, s)\geq 0$. Let $\mu$ be an ergodic equilibrium measure for the potential $\{\log \varphi^s(T^*_{\cdot|n})\}_{n=1}^\infty$.   Then
\begin{equation}
\label{e-hmu}
h_\mu(\sigma)+\theta=P({\bf T}, s)\geq 0,
\end{equation}
where $\theta=\lim_{n\to \infty}\frac{1}{n}\int \log \varphi^s(T^*_{x|n})\; d\mu(x)$.

Consider the matrix cocycle $x\mapsto T_{x_1}^*$ with respect to $\mu$ and let $\Sigma'$, $r$, $\Lambda_1,\ldots, \Lambda_d$, $\bigoplus_{j=1}^r E_j(x)$ ($x\in \Sigma'$) be given as in Theorem \ref{thm:oseledets}, and also let  ${\bf v}(x)=\{v_1(x),\ldots, v_d(x)\}$ be a measurable ordered basis adapted to the splitting $\bigoplus_{j=1}^r E_j(x)$, $x \in \Sigma'$. Since $s\leq k\leq d$,   by \eqref{e-Gs},
\begin{equation}
\label{e-theta}
\theta=\sum_{j=1}^{\lfloor s\rfloor}\Lambda_j +(s-\lfloor s\rfloor)\Lambda_{\lfloor s\rfloor+1}.
\end{equation}

Let $W\in G(d,k)$. Fix $x\in \Sigma'$. Define
\begin{align*}
H=\; & {\rm span}\left\{v_{j_1}(x)\wedge \cdots \wedge v_{j_q}(x)\colon 1\leq j_1<\cdots<j_q\leq d \mbox{ and }\right.\\
&\qquad \qquad  \qquad \qquad \qquad\qquad \qquad \left. (j_1,\ldots, j_q)\neq (1,\ldots, q) \right\}.
\end{align*}
Then $H$ is a proper linear subspace of $(\R^d)^{\wedge q}$. Since $\{T_i^{\wedge q}\colon i=1,\ldots, m\}$ is irreducible, by Lemmas \ref{lem-F-3.1}(iii) and  \ref{lem-irreducible},
$\{(T_i^*)^{\wedge q}\colon i=1,\ldots, m\}$ is also irreducible.  Therefore there exists $J\in \Sigma_*$ such that
\begin{equation}
\label{e-tjw}
(T_J^*W)^{\wedge q}=(T_J^*)^{\wedge q} (W^{\wedge q})\not\subset H.
\end{equation}
Let $(p_1,\ldots, p_k)$ be the pivot position vector of $T_J^*W$ with respect to the ordered basis ${\bf v}(x)$ of $\R^d$.  By \eqref{e-tjw} and Lemma \ref{lem-2.5'}(ii),  $p_1\leq 1, \ldots, p_q\leq q$.  However it always holds that $1\leq p_1<\cdots <p_q$,  implying that $p_i\geq i$ for $1\leq i\leq q$. Hence $p_i=i$ for  $1\leq i\leq q$. By Proposition \ref{pro-key} and \eqref{e-theta},
\begin{align*}
\lim_{n\to \infty}\frac{1}{n}\log \varphi^s(T_{x|n}^*P_{T^*_JW})&=\sum_{j=1}^{\lfloor s\rfloor}\Lambda_{p_j} +(s-\lfloor s\rfloor)\Lambda_{p_{\lfloor s\rfloor+1}}\\
&= \sum_{j=1}^{\lfloor s\rfloor}\Lambda_j +(s-\lfloor s\rfloor)\Lambda_{\lfloor s\rfloor+1}\\
&=\theta.
\end{align*}
It follows that for every $x\in \Sigma'$,
$$\lim_{n\to \infty}\frac{1}{n}\log \psi^s_W(x|n)\geq \theta.$$
Hence applying Theorem \ref{thm:subadditive-VP} to the subadditive potential $\{\log \psi_W^s(\cdot|n)\}_{n=1}^\infty$ gives
$$
P({\bf T}, W,s)\geq h_{\mu}(\sigma)+\lim_{n\to \infty}\frac{1}{n}\int \log \psi^s_W(y|n)\; d\mu(y)\geq h_{\mu}(\sigma)+\theta\geq 0,
$$
where the last inequality follows from \eqref{e-hmu}. Hence by Lemma \ref{lem-Affd}(ii),
 $$\dim_{\rm AFF}({\bf T}, W)\geq s=\min\{k, \dim_{\rm AFF}({\bf T})\}.$$
 This, combined with  Lemma \ref{lem-Affd}(i), implies that $\dim_{\rm AFF}({\bf T}, W)=s$.
\end{proof}

To prove Theorem \ref{thm-1.1}(iii), we need the following elementary result.

\begin{lem}
\label{lem-4.1}
Let $W\in G(d,k)$ and $M\in {\rm GL}_d(\R)$. Then
\begin{equation}
\label{e-PWM}
P_WM=P_WM P_{M^*W},
\end{equation}
and $P_WM(M^*(W))=W$.
 Moreover, the mapping $P_WM\colon M^*(W)\to W$ is surjective and hence bi-Lipschitz. Consequently, for every Borel set $E\subset \R^d$,
\begin{equation}
\label{e-rmH}
\dim_{\rm H} P_W(M(E))=\dim_{\rm H} P_{M^*W}(E).
\end{equation}
\end{lem}
\begin{proof}
Let $I_d$ denote the identity map from $\R^d$ to itself. Clearly, $$P_{M^*W}+P_{(M^*W)^{\perp}}=I_d.$$
 Hence to prove \eqref{e-PWM}, it suffices to prove that
\begin{equation}
\label{e-perp}
P_WM P_{(M^*W)^{\perp}}=0.
\end{equation}
To see the above identity, let $x\in (M^*W)^{\perp}$. Then for any $y\in W$,
$$\langle M x, y\rangle=\langle x, M^*y\rangle=0.$$  It follows that
  $P_W(Mx)=0$. This proves \eqref{e-perp} and thus \eqref{e-PWM}.

Now, the equality $P_WM(M^*(W))=W$ follows directly from \eqref{e-PWM}. Indeed,
$$
P_WM(M^*(W))=P_WM P_{M^*W}(\R^d)=P_WM(\R^d)=P_W(\R^d)=W,
$$
where we use \eqref{e-PWM} in the second equality.
Since  $M^*(W)$ and $W$ have the same dimension, the mapping
$P_WM\colon M^*(W)\to W$ is linear and invertible,  so it is bi-Lipschitz.

%Clearly the mapping $P_WM$ from $M^*(W)$ to $W$ is linear. Since $M^*(W)$ and $W$ have the same dimension, to show that the mapping is invertible, it is enough to show that the kernel of the mapping is the null space.

%Suppose $P_WM(x)=0$ for some $x\in M^*(W)$. Then $x=M^*y$ for some $y\in W$.  Hence $P_WMM^*y=0$. That is,  $MM^*y\perp W$ and so $MM^*y\perp y$. It follows that $\langle y, MM^*y\rangle=0$, i.e. $y^*MM^*y=0$.  Thus $x^*x=0$, implying that $x=0$.

Finally, let $E$ be a Borel subset of $\R^d$. By \eqref{e-PWM}, $P_W(M(E))=P_WM (P_{M^*W}(E))$.
Since $P_{M^*W}(E)\subset M^*(W)$ and the mapping $P_WM\colon M^*(W)\to W$ is bi-Lipschitz,
$$
\dim_{\rm H} P_W(M(E))=\dim_{\rm H} P_WM (P_{M^*W}(E))=\dim_{\rm H} P_{M^*W}(E),
$$
as desired.
\end{proof}

The following transversality result is also needed in the proof of  Theorem \ref{thm-1.1}(iii).

\begin{lem}
\label{lem-transversality}
Assume that $\|T_i\|<1/2$ for $1\leq i\leq m$. Let $\rho>0$. Then there exists a constant $c=c(\rho, T_1,\ldots, T_m)$ such that for each $r>0$, $W\in G(d,k)$, and distinct $x,y\in \Sigma$,
\begin{equation*}
\mathcal L^{md}\{\ba\in B_\rho\colon \|P_W\pi^\ba x-P_W\pi^\ba y\|<r\}\leq c \prod_{i=1}^k \min\left\{\frac{r}{\alpha_i(P_WT_{x\w y})},1\right\},
\end{equation*}
where $x\w y$ denotes the common initial segment of $x$ and $y$. In particular,
\begin{equation}
\label{e-transversality}
\mathcal L^{md}\{\ba\in B_\rho\colon \|P_W\pi^\ba x-P_W\pi^\ba y\|<r\}\leq  \frac{cr^k}{\varphi^k(P_WT_{x\w y})}.
\end{equation}

\end{lem}
\begin{proof}
It is a slight and  trivial modification of  the proof of \cite[Lemma 5.2]{JPS07}.
\end{proof}

\begin{proof}[Proof of Theorem \ref{thm-1.1}(iii)]
Let $W\in G(d,k)$. It was  proved by Morris \cite[Theorem 1]{Morris2023} that $$\overline{\dim}_{\rm B}P_W(K^\ba)\leq \dim_{\rm AFF}({\bf T}, W)$$
for every $\ba\in \R^{md}$. Below we assume that $\|T_i\|<1/2$ for all $1\leq i\leq m$. We first show that
$$\dim_{\rm H}P_W(K^\ba)\geq  \dim_{\rm AFF}({\bf T}, W)$$
for ${\mathcal L}^{md}$-a.e.~${\bf a}\in \R^{md}$.

Write $s=\dim_{\rm AFF}({\bf T}, W)$ and let $\mu$ be an ergodic equilibrium measure for the subadditive potential $\{\log \psi^s_W(\cdot|n)\}_{n=1}^\infty$.  By  Lemma \ref{lem-Affd}(iii),
\begin{equation}
\label{e-hmutheta}
h_\mu(\sigma)+\theta=P({\bf T}, W,s)\geq 0,
\end{equation}
where
$$
\theta:=\lim_{n\to \infty} \int \log \psi^s_W(x|n)\; d\mu(x).
$$

By Corollary \ref{cor-2.8}, there exists a measurable $A\subset \Sigma'$ with $\mu(A)>0$ and $J\in \Sigma_*$ such that for each $x\in A$,
$$
\lim_{n\to \infty}\frac{1}{n}\log \varphi^s(T_{x|n}^*P_{T_J^*W})=\lim_{n\to \infty}\frac{1}{n}\log \psi_W^s(x|n)=\theta.
$$
This, together with Proposition \ref{pro-key}, yields that for each $x\in A$,
$$
\sum_{j=1}^{\lfloor s \rfloor}\Lambda_{p_j(T^*_{J}W,x)}+(s-\lfloor s \rfloor)\Lambda_{p_{\lfloor s \rfloor+1}(T^*_{J}W,x)}=\theta,
$$
where $(p_1(T^*_{J}W,x), \ldots, p_k(T^*_{J}W,x))$ is the pivot position vector of $T_J^*W$ with respect to the ordered basis ${\bf v}(x)$ (see Theorem
 \ref{thm:oseledets} for the definition of   ${\bf v}(x)$).  Combining this with \eqref{e-hmutheta} and Lemma  \ref{lem-smuw} yields that for
 for each $x\in A$,
$$
S(\mu, T_J^*W, x)\geq s,
$$
where $S(\mu, T_J^*W, x)$ is defined as in \eqref{e-Smux}; see also \eqref{e-equi}.  Hence
$$\overline{S}(\mu, T_J^*W):=\esssup_{x\in {\rm spt}\mu} S(\mu,T_J^*W, x)\geq s.
$$
By Theorem \ref{thm-1.2}(iv) (in which we replace $W$ by $T_J^*W$), for  ${\mathcal L}^{md}$-a.e.~${\bf a}\in \R^{md}$.
$$
\udim{H} (P_{T_J^*W}\pi^\ba)_*\mu  = \overline{S}(\mu, T_J^*W)\geq s.
$$
Since $(P_{T_J^*W}\pi^\ba)_*\mu$ is supported on $P_{T^*_JW}(K^\ba)$, it follows that
\begin{equation}
\label{e-rmHP}
\dim_{\rm H}  P_{T^*_JW}(K^\ba)\geq s \quad \mbox{for  ${\mathcal L}^{md}$-a.e.~${\bf a}\in \R^{md}$.}
\end{equation}
Meanwhile,  by the self-affinity of $K^\ba$, we have  $K^\ba\supset f^{\ba}_J(K^{\ba})$.  Hence for each ${\bf a}\in \R^{md}$,
\begin{align*}
\dim_{\rm H}P_W(K^\ba)&\geq \dim_{\rm H} P_W(f^{\ba}_J(K^{\ba}))\\
&=\dim_{\rm H}P_W(T_J(K^{\ba}))\\
&=\dim_{\rm H}P_{T^*_JW}(K^{\ba}),
\end{align*}
where the last equality follows from \eqref{e-rmH} (in which we take $M=T_J$ and $E=K^\ba$).  Combining it with \eqref{e-rmHP} yields that
$$\dim_{\rm H}  P_{W}(K^\ba)\geq s \quad \mbox{for  ${\mathcal L}^{md}$-a.e.~${\bf a}\in \R^{md}$.}
$$

Next assume that \eqref{e-lebcondition} holds, that is, $P({\bf T}, W,k)>0$ by Proposition \ref{pro-equivalence}. Below, we will show that $\mathcal H^k(P_{W}(K^\ba))>0$ for ${\mathcal L}^{md}$-a.e.~${\bf a}\in \R^{md}$, by using Corollary \ref{cor-2.8} and adapting the proof of \cite[Proposition 4.4(b)]{JPS07}.

Let $\mu$ be an ergodic equilibrium measure for the  potential $\{\log \psi^k_W(\cdot|n)\}_{n=1}^\infty$. Then
$$
h_\mu(\sigma)+\Theta=P({\bf T}, W,k)>0,
$$
where $\Theta:=\lim_{n\to \infty}(1/n)\int \log \psi^k_W(x|n)\; d\mu(x)$.
By the Shannon-McMillan-Breiman theorem,
\begin{equation}
\label{e-SMB'}
\lim_{n\to \infty}\frac{1}{n}\log \mu([x|n])=-h_\mu(\sigma) \quad \mbox{ for $\mu$-a.e.~$x\in \Sigma$}.
\end{equation}
Meanwhile by Corollary \ref{cor-2.8}, there exist a measurable set $A\subset \Sigma$ with $\mu(A)>0$ and a word $J\in \Sigma_*$ such that
\begin{equation}
\label{e-psiw1'}
\lim_{n\to \infty}\frac{1}{n}\log \varphi^k(T_{x|n}^*P_{T_J^*W})=\lim_{n\to \infty}\frac{1}{n}\log \psi_W^k(x|n)=\Theta\quad \mbox{ for }x\in A.
\end{equation}

Let $0<\epsilon<(h_\mu(\sigma)+\Theta)/3$. By \eqref{e-SMB'} and \eqref{e-psiw1'}, there exist $N\in \N$ and $A_1\subset A$ with $\mu(A_1)>0$ such that for all $x\in A_1$ and $n\geq N$,
\begin{equation*}
\varphi^k(T_{x|n}^*P_{T_J^*W})\geq e^{n(\Theta-\epsilon)},\qquad \mu([x|n])\leq e^{-n(h_\mu(\sigma)-\epsilon)}.
\end{equation*}
Consequently,
\begin{equation}
\label{e-inequality}
\mu([x|n])\leq e^{-n\epsilon} \varphi^k(T_{x|n}^*P_{T_J^*W}) \quad \mbox{ for all $x\in A_1$ and $n\geq N$}.
\end{equation}
Define a Borel probability measure $\widetilde{\mu}$ on $\Sigma$ by
$$
\widetilde{\mu} (E)=\frac{\mu(E\cap A_1)}{\mu(A_1)} \quad \mbox{ for any Borel set $E\subset \Sigma$}.
$$
By \eqref{e-inequality}, there exists $C>0$ such that
\begin{equation}
\label{e-inequality'}
\widetilde{\mu}([I])\leq C e^{-n\epsilon} \varphi^k(T_{I}^*P_{T_J^*W}) \quad \mbox{ for all $I\in \Sigma_*$}.
\end{equation}

Write $\widetilde{W}=T_J^*(W)$.   Next we prove the absolute continuity of $\eta^\ba:=(P_{\widetilde{W}}\pi^\ba)_*\widetilde{\mu}$ with respect to the $k$-dimensional Lebesgue measure on $\widetilde{W}$ for ${\mathcal L}^{md}$-a.e.~$\ba$, by following the standard approaches in \cite{PeresSolomyak1996, JPS07}. Let $\rho>0$, and let $B_\rho$ denote the closed ball in $\R^{md}$ of radius $\rho$ centred at the origin.  It suffices to show that
$$
I_\rho:=\int_{B_\rho}\int_{\R^{md}}\liminf_{r\to 0}\frac{\eta^\ba(B(z,r))}{r^k}\;d\eta^\ba(z)d\ba<\infty.
$$
Applying Fatou's lemma and Fubini's theorem,
\begin{align*}
I_\rho &\leq \liminf_{r\to 0}\frac{1}{r^k} \int_{B_\rho}\int_\Sigma\int_\Sigma  {\bf 1}_{\{(x,y)\colon \|P_{\widetilde{W}}\pi^\ba(x)-P_{\widetilde{W}}\pi^\ba(y)\|\leq r\}}\; d\widetilde{\mu}(x)d\widetilde{\mu}(y)d\ba\\
&\leq \liminf_{r\to 0}\frac{1}{r^k} \int_\Sigma\int_\Sigma \mathcal L^{md}\left\{\ba\in B_\rho\colon \|P_{\widetilde{W}}\pi^\ba(x)-P_{\widetilde{W}}\pi^\ba(y)\|\leq r\right\}\; d\widetilde{\mu}(x)d\widetilde{\mu}(y)\\
&\leq c\int_\Sigma\int_\Sigma \frac{1}{\varphi^k(P_{\widetilde{W}}T_{x\w y})} \; d\widetilde{\mu}(x)d\widetilde{\mu}(y)
\qquad\qquad\mbox{(by \eqref{e-transversality})}\\
&\leq c\sum_{n=0}^\infty \sum_{I\in \Sigma_n} \frac{\widetilde{\mu}([I])^2} { \varphi^k(P_{\widetilde{W}} T_I)}\\
&\leq cC\sum_{n=0}^\infty e^{-n\epsilon}\qquad\qquad \mbox{(by \eqref{e-inequality'})}\\
&<\infty.
\end{align*}
Hence,  $\eta^\ba$ is absolutely continuous with respect to the Lebesgue measure on $\widetilde{W}=T_J^*(W)$ for ${\mathcal L}^{md}$-a.e.~$\ba\in \R^{md}$.  Since $\eta^\ba$ is supported on $P_{T^*_JW}(K^\ba)$, it follows that $\mathcal H^k(P_{T^*_JW}(K^\ba))>0$ for ${\mathcal L}^{md}$-a.e.~$\ba\in \R^{md}$.

Meanwhile,  by the self-affinity of $K^\ba$, we have  $K^\ba\supset f^{\ba}_J(K^{\ba})$.  It follows that for each ${\bf a}\in \R^{md}$,
\begin{align*}
\mathcal H^k(P_W(K^\ba))&\geq  \mathcal H^k( P_W(f^{\ba}_J(K^{\ba})))\\
&=\mathcal H^k( P_WT_J(K^{\ba}))\\
&=\mathcal H^k(P_WT_J P_{T^*_JW}(K^{\ba}))\qquad \qquad \mbox{(by \eqref{e-PWM})}\\
&=D\mathcal H^k(P_{T^*_JW}(K^{\ba})),
\end{align*}
where $D>0$ is a positive constant which depends on $P_WT_J$, and this follows from the fact that $P_WT_J\colon T^*_J(W)\to W$ is a bijective linear map (see Lemma \ref{lem-4.1}).
Hence $\mathcal H^k(P_{W}(K^\ba))>0$ for ${\mathcal L}^{md}$-a.e.~$\ba\in \R^{md}$.
This completes the proof of Theorem \ref{thm-1.1}(iii).
\end{proof}

Finally, we prove part (ii) of Theorem \ref{thm-1.1}.

\begin{proof}[Proof of Theorem \ref{thm-1.1}(ii)] Let $k\in \{1,\ldots, d-1\}$. We first prove that under an additional assumption that $\|T_i\|<1/2$ for all $1\leq i\leq m$,
\begin{equation}
     \label{e-AFFW*}
     \dim_{\rm AFF}({\bf T}, W)=\min\{k, \dim_{\rm AFF}({\bf T})\} \qquad \mbox{ for $\gamma_{d,k}$-a.e.~$W\in G(d,k)$}.
     \end{equation}

To see this, assume that $\|T_i\|<1/2$ for all $1\leq i\leq m$. By \cite[Theorem~5.3]{Fal88} and \cite[Proposition~3.1]{Sol98},
 $$
 \dim_{\rm H} K^\ba=\min\{d, \dim_{\rm AFF}({\bf T})\} \qquad \mbox{ for $\mathcal L^{md}$-a.e.~$\ba\in \R^{md}$}.
$$
This, together with the higher dimensional analog of Marstrand's projection theorem proved by Mattila \cite{Mat75}, implies that for $\mathcal L^{md}$-a.e.~$\ba\in \R^{md}$,
$$
\dim_{\rm H} P_W(K^\ba)=\min\{k, \dim_{\rm H} K^\ba\}=\min\{k, \dim_{\rm AFF}({\bf T})\}$$
for $\gamma_{d,k}$-a.e.~$W\in G(d,k)$.
Applying the Fubini theorem, we see that for $\gamma_{d,k}$-a.e.~$W\in G(d,k)$,
$$
\dim_{\rm H} P_W(K^\ba)=\min\{k, \dim_{\rm AFF}({\bf T})\} \qquad \mbox{ for $\mathcal L^{md}$-a.e.~$\ba\in \R^{md}$}.
$$
Combining this with Theorem \ref{thm-1.1}(iii) yields  \eqref{e-AFFW*}.

Next we consider the general case when $\|T_i\|<1$ for all $1\leq i\leq m$.  Take a large integer $n$ such that
$$
\|T_{I}\|<1/2 \qquad \mbox{for  all $I\in \Sigma_n$}.
$$
Define a tuple ${\bf T}^{(n)}$ of $d\times d$ matrices by ${\bf T}^{(n)}=(T_I)_{I\in \Sigma_n}$. As was proved above, \eqref{e-AFFW*} holds when ${\bf T}$ is replaced by ${\bf T}^{(n)}$. However, by performing a routine check using the definition (which we leave as an exercise for the reader), one finds that
$$
\dim_{\rm AFF}({\bf T}^{(n)})=\dim_{\rm AFF}({\bf T}) \quad \mbox{ and }\quad \dim_{\rm AFF}({\bf T}^{(n)}, W)=\dim_{\rm AFF}({\bf T}, W).
$$
This proves \eqref{e-AFFW*} for ${\bf T}$ in the general case.
\end{proof}

\section{More about $\dim_{\rm AFF}({\bf T}, W)$, $\overline{S}(\mu,{\bf T}, W)$ and $\underline{S}(\mu,{\bf T}, W)$ in some special cases}
\label{S-7}

In this section, we provide several results (Propositions \ref{pro-planar case}, \ref{pro-one-dimensional} and  \ref{pro-d=3}) on $\dim_{\rm AFF}({\bf T}, W)$ in the cases where $d=2$ or $d=3$,  or where $\dim W=1$. Additionally, we present one result (Proposition \ref{pro-planar case-measure}) concerning  $\overline{S}(\mu,{\bf T}, W)$ and $\underline{S}(\mu,{\bf T}, W)$ in the case where $d=2$.

Our first result provides a simple verifiable criterion for $\dim_{\rm AFF}({\bf T}, W)$ to be strictly less than  $\min\{1, \dim_{\rm AFF}({\bf T})\}$ in the case where $d=2$.
\begin{pro}
\label{pro-planar case}
Assume that $d=2$. Let $W\in G(2,1)$.
Then
\begin{equation*}
\label{e-cond1}
\dim_{\rm AFF}({\bf T}, W)<\min\{1, \dim_{\rm AFF}({\bf T})\}
\end{equation*}
 if and only if the following two properties hold:
\begin{itemize}
\item[(1)] $T_i^*W=W$ for all $1\leq i\leq m$;
\item[(2)]  Letting $a_i$ be the eigenvalue of $T_i^*$ corresponding to $W$, and setting $b_i=\det(T_i)/a_i$ for $1\leq i\leq m$, one has
$t<\min\{1,s\}$, where $s, t$ are the unique positive numbers so that
$$
\sum_{i=1}^m |a_i|^t=1,\qquad  \sum_{i=1}^m |b_i|^s=1.
$$
\end{itemize}
\end{pro}
\begin{proof} We first prove the ``if'' part of the proposition. Assume that  both (1) and (2) hold.  Since $W$ is $T_i^*$-invariant for all $i$, there exists $G\in {\rm GL}_2(\R)$ such that $GT_i^*G^{-1}$ is upper triangular and of the form
$$
\left(\begin{array}{cc} a_i & *\\
0 & b_i
\end{array}
\right)
$$
for each $1\leq i\leq m$.
A simple calculation shows that $\dim_{\rm AFF}({\bf T}, W)=\min\{1, t\}$.  Moreover, there is a closed formula for $\dim_{\rm AFF}({\bf T})$ (see \cite[Corollary 2.6]{FalconerMiao2007}), from which one can easily show that  $$\min\{1,\dim_{\rm AFF}({\bf T})\}= \min\{1,\max\{s, t\}\}.$$
Since $t<\min\{1,s\}$, we obtain that    $\dim_{\rm AFF}({\bf T}, W)<\min\{1, \dim_{\rm AFF}({\bf T})\}$.

Next we turn to the proof of the ``only if'' part. Assume that
\begin{equation}
\label{e-cond1}\dim_{\rm AFF} ({\bf T}, W)<\min\{1, \dim_{\rm AFF}({\bf T})\}.
\end{equation}
  Write $s_0=\dim_{\rm AFF} ({\bf T}, W)$. Let $\mu$ be an ergodic equilibrium measure for the subadditive potential $\{\log \varphi^{s_0}(T^*_{\cdot|n})\}_{n=1}^\infty$, and   let $\Lambda_1\geq \Lambda_2$ be the Lyapunov exponents of the cocycle $x\mapsto T_{x_1}^*$ with respect to $\mu$. Then
\begin{equation}
\label{e-hmus0}
h_\mu(\sigma)+s_0\Lambda_1=P({\bf T}, s_0)> 0.
\end{equation}
where $P({\bf T}, s_0):=\lim_{n\to \infty}(1/n)\log \left(\sum_{I\in \Sigma_n} \varphi^{s_0}(T_I^*)\right)$, and the second inequality follows from the assumption that $s_0< \min\{1, \dim_{\rm AFF}({\bf T})\}$.  Since $s_0<1$,  by
Lemma \ref{lem-Affd}(iii),
\begin{equation}
\label{e-hmus1}
P({\bf T}, W, s_0)=0,
\end{equation}
where $P({\bf T}, W, s_0)$ stands for the topological pressure of the subadditive potential $\{\log \psi_W^{s_0}(\cdot|n)\}_{n=1}^\infty$; see \eqref{e-psiw} and \eqref{e-ptws}. By the subadditive variational principle (see Theorem \ref{thm:subadditive-VP}),
$P({\bf T}, W, s_0)\geq h_\mu(\sigma)+\Theta$, where $\Theta:=\lim_{n\to \infty}(1/n)\int \log \psi_W^{s_0}(x|n)\;d\mu(x)$. Combining this with \eqref{e-hmus0} and \eqref{e-hmus1} yields that
$$
\Theta<s_0\Lambda_1.
$$
Meanwhile, since $\psi^s_W$ is submultiplicative (see  Lemma \ref{lem-subm}),  by the subadditive ergodic theorem,
$$
\lim_{n\to \infty}\frac1n\log \psi_W^{s_0}(x|n)=\Theta \quad \mbox{ for $\mu$-a.e.~$x\in \Sigma$}.
$$
 This, combined with  Proposition \ref{pro-key}, yields that
$$
\sup_{J\in \Sigma_*}s_0\Lambda_{p_1(T_J^*W,x)}=\Theta\quad \mbox{ for $\mu$-a.e.~$x\in \Sigma$},
$$
where $p_1(T_J^*W,x)$ denotes the pivot position vector of $T_J^*W$ with respect to the ordered basis ${\bf v}(x)=\{v_i(x)\}_{i=1}^2$ defined in Theorem \ref{thm:oseledets}.   Since $\Theta<s_0\Lambda_1$, it follows that $\Lambda_2<\Lambda_1$ and that for $\mu$-a.e.~$x\in \Sigma$,
$$p_1(T_J^*W,x)=2 \quad \mbox{ for all }J\in \Sigma_*,
$$
and consequently,  $T_J^*W={\rm span}\{v_2(x)\}$ for all $J\in \Sigma_*$.  This implies that $T_i^*W=W$ for all $1\leq i\leq m$.  Hence (1) holds. As was pointed out in the beginning of our proof, in this case, $\dim_{\rm AFF}({\bf T}, W)=\min\{1,t\}$ and $\min\{1, \dim_{\rm AFF}({\bf T})\}=\min\{1, \max\{s,t\}\}$. So the condition \eqref{e-cond1} implies that $t<\min\{1, s\}$. Hence (2) also holds.
\end{proof}

Our next result characterizes, in the planar case,  the circumstances under which exceptional phenomena occur for the projections of ergodic stationary measures.
\begin{pro}
\label{pro-planar case-measure}
Assume that $d=2$. Let $W\in G(2,1)$, and let $\mu$ be an ergodic $\sigma$-invariant measure on $\Sigma$. Let $\Lambda_1\geq  \Lambda_2$ be the Lyapunov exponents (counting multiplicity) of the cocycle $x\mapsto T_{x_1}^*$ with respect to $\mu$ (see Theorem \ref{thm:oseledets}). Set ${\mathcal A}=\{i\colon 1\leq i\leq m\mbox{ and } \mu([i])>0\}$. Then the following statements hold.
\begin{enumerate}
\item
$\overline{S}(\mu, {\bf T},W)<\min\{1, \dim_{\rm LY}(\mu, {\bf T})\}$
 if and only if all the following conditions are fulfilled:
\begin{enumerate}
\item $T_i^*W=W$ for all $i\in \mathcal A$;
\item For $i\in \mathcal A$, let $a_i$ be the eigenvalue of $T_i^*$ corresponding to $W$, and set $b_i=\det(T_i)/a_i$. Then
$$
\Lambda_2=\sum_{i\in \mathcal A}\mu([i])\log |a_i|<\Lambda_1= \sum_{i\in \mathcal A}\mu([i])\log |b_i|.
$$
\item $h_\mu(\sigma)>0$ and $h_\mu(\sigma)+\Lambda_2<0$.
\end{enumerate}
\medskip
\item $\overline{S}(\mu, {\bf T},W)\neq \underline{S}(\mu, {\bf T},W)$ if and only if all the following conditions are fulfilled:
\begin{enumerate}
\item $\Lambda_1>\Lambda_2$.
\item Let $\R^2=\bigoplus_{i=1}^2E_i(x)$, $x\in \Sigma'$, be the corresponding Oseledets splittings for the cocycle $x\mapsto T_{x_1}^*$ and $\mu$. Then
    $$
    0<\mu(\{x\in \Sigma'\colon E_2(x)=W\})<1.
    $$
\item $h_\mu(\sigma)>0$ and $h_\mu(\sigma)+\Lambda_2<0$.
\end{enumerate}
\end{enumerate}
\end{pro}
\begin{proof} From Definition \ref{de-Lyapunov} and \eqref{e-Gs}, it is readily checked that
\begin{equation}
\label{e-8.41}
\min\{1, \dim_{\rm LY}(\mu, {\bf T})\}=\sup\{0\leq s\leq 1\colon h_\mu(\sigma)+s\Lambda_1\geq 0\}.
\end{equation}
Let ${\bf v}(x)=\{v_i(x)\}_{i=1}^2$, where $x\in \Sigma'$, be the ordered basis of $\R^2$ defined as in Theorem \ref{thm:oseledets}, and let $(p_i(W,x))_{i=1}^2$ denote the pivot position vector of $W$ with respect to $\{v_i(x)\}_{i=1}^2$. From Lemma \ref{lem-smuw}, we obtain that
\begin{equation}
\label{e-8.42}
\overline{S}(\mu,{\bf T}, W)=\esssup_{x\in \Sigma'}\;\sup\{0\leq s\leq 1\colon h_\mu(\sigma)+s\Lambda_{p_1(W,x)}\geq 0\}
\end{equation}
and
\begin{equation}
\label{e-8.43}
\underline{S}(\mu, {\bf T},W)=\essinf_{x\in \Sigma'}\;\sup\{0\leq s\leq 1\colon h_\mu(\sigma)+s\Lambda_{p_1(W,x)}\geq 0\}.
\end{equation}

By \eqref{e-8.41} and \eqref{e-8.42}, we see that $\overline{S}(\mu,{\bf T}, W)<\min\{1, \dim_{\rm LY}(\mu, {\bf T})\}$ if and only if
$p_1(W,x)=2$ for $\mu$-a.e.~$x\in \Sigma'$ and moreover
\begin{equation}
\label{e-8.44}\sup\{0\leq s\leq 1\colon h_\mu(\sigma)+s\Lambda_{2}\geq 0\}<\sup\{0\leq s\leq 1\colon h_\mu(\sigma)+s\Lambda_{1}\geq 0\}.
\end{equation}
Notice that \eqref{e-8.44} holds if and only if the following two conditions are satisfied: (i) $\Lambda_2<\Lambda_1$; (ii) $h_\mu(\sigma)>0$ and $h_\mu(\sigma)+\Lambda_2<0$. Meanwhile, the condition that $p_1(W,x)=2$ for $\mu$-a.e.~$x\in \Sigma'$ is equivalent to ${\rm span}\{v_2(x)\}=W$ for $\mu$-a.e.~$x\in \Sigma'$. Observe that if $\Lambda_1>\Lambda_2$, then ${\rm span}\{v_2(x)\}=E_2(x)$ for $\mu$-a.e.~$x\in \Sigma'$, where $\bigoplus_{i=1}^2E_i(x)$, with $x\in \Sigma'$, is  the associated Oseledets splitting for the cocycle $x\mapsto T_{x_1}^*$ with respect to $\mu$. Since $T_{x_1}E_2(x)=E_2(\sigma x)$ a.e., the condition that ${\rm span}\{v_2(x)\}=W$ for $\mu$-a.e.~$x\in \Sigma'$ implies that $T_i^*W=W$ for all $i\in \mathcal A$. Hence if $\overline{S}(\mu,{\bf T}, W)<\min\{1, \dim_{\rm LY}(\mu, {\bf T})\}$, then all the statements (a), (b) and (c) in part (1) hold. Conversely, if all the statements (a), (b) and (c) in part (1) hold, then it is direct to apply \eqref{e-8.41}-\eqref{e-8.42} to conclude  that $\overline{S}(\mu,{\bf T}, W)<\min\{1, \dim_{\rm LY}(\mu, {\bf T})\}$. This proves (1).

To see (2), by \eqref{e-8.42}-\eqref{e-8.43}, we see that $\overline{S}(\mu, {\bf T},W)>\underline{S}(\mu,{\bf T}, W)$ if and only if the following three conditions are satisfied: \begin{itemize}
\item[(i)] $\Lambda_1>\Lambda_2$;
\item[(ii)] $0<\mu\{x\in \Sigma'\colon p_1(W,x)=2\}<1$; and
\item[(iii)] \eqref{e-8.44} holds.
\end{itemize}
This is enough to conclude (2), since $p_1(W,x)=2$ is equivalent to ${\rm span}\{v_2(x)\}=W$, and in the case when $\Lambda_1>\Lambda_2$, one has $E_2(x)={\rm span}\{v_2(x)\}$ a.e.
\end{proof}

The following result provides a formula for $\dim_{\rm AFF}({\bf T}, W)$ in the case where $d\geq 2$ and $\dim W=1$.
\begin{pro}
\label{pro-one-dimensional}
Let $d\geq 2$ and  $W\in G(d,1)$. Then
$$
\dim_{\rm AFF}({\bf T}, W)=\min\{1, \dim_{\rm AFF} ({\bf T}_X)\},
$$
where $X={\rm span}\left(\bigcup_{J\in \Sigma_*}T^*_J(W)\right)$ and ${\bf T}_X:=(T_1^*|_X,\ldots, T_m^*|_X)$, in which $T_i^*|_X$ stands for the restriction of $T_i^*$ on $X$.
\end{pro}
\begin{proof}
Clearly, the subspace $X$ is $T_i^*$-invariant for all $1\leq i\leq m$. Write $$s_0=\min\{1, \dim_{\rm AFF} ({\bf T}_X)\}.$$
Then $P({\bf T}_X, s_0)\geq 0$, where $P({\bf T}_X, s_0)$ stands for the topological pressure of the subadditive potential $\{\log \varphi^{s_0}(T_{\cdot|n}^*|_X)\}_{n=1}^\infty$.

Let $\mu$ be an ergodic equilibrium measure for the potential $\{\log \varphi^{s_0}(T_{\cdot|n}^*|_X)\}_{n=1}^\infty$. Moreover let $\lambda_1>\cdots>\lambda_r$ be the distinct Lyapunov exponents of the cocycle $x\mapsto T_{x_1}^*|_X$ with respect to $\mu$ and let
$$
X=V_0(x)\supsetneqq \cdots \supsetneqq V_r(x),\quad x\in \Sigma',
$$
be the associated Oseledets filtration, where $\Sigma'$ is a $\sigma$-invariant Borel subset of $\Sigma$ with $\mu(\Sigma')=1$. By Theorem \ref{thm:subadditive-VP}, $$h_\mu(\sigma)+s_0\lambda_1= P({\bf T}_X, s_0)\geq 0.$$

For each $x\in \Sigma'$, since $X=V_0(x)\supsetneqq V_1(x)$, it follows that
\begin{equation}
\label{e-Tsupset}
\left(\bigcup_{J\in \Sigma_*} T_J^*(W)\right)\cap \left(V_0(x)\backslash V_1(x)\right)\neq \emptyset,
\end{equation}
or equivalently, $\bigcup_{J\in \Sigma_*} T_J^*(W)\not\subset V_1(x)$; otherwise, $$V_0(x)= {\rm span}\left(\bigcup_{J\in \Sigma_*}T^*_J(W)\right)\subset V_1(x),$$
leading to a contradiction. Hence by \eqref{e-Tsupset}, for each $x\in \Sigma'$,
\begin{align*}
\lim_{n\to \infty}\frac{1}{n}\log \psi_W^{s_0}(x|n)&\geq \sup_{J\in \Sigma_*} \lim_{n\to \infty} \frac1n \log \varphi^{s_0}(T_{x|n}^*P_{T_J^*W})\\
&=s_0\sup_{J\in \Sigma_*} \lim_{n\to \infty} \frac1n \log \|T_{x|n}^*P_{T_J^*W}\|\\
&=s_0\lambda_1,
\end{align*}
where $\psi_W^{s_0}$ is defined as in \eqref{e-psiw}.   Applying Theorem \ref{thm:subadditive-VP} to the subadditive potential
$\{\log \psi_W^{s_0}(\cdot|n)\}_{n=1}^\infty$
gives
$$
P({\bf T}, W, s_0)\geq h_\mu(\sigma)+ \lim_{n\to \infty} \frac{1}{n}\int \log \psi_W^{s_0}(x|n)\; d\mu(x)\geq  h_\mu(\sigma)+s_0\lambda_1\geq 0,
$$
where $P({\bf T}, W, s_0)$ is defined as in \eqref{e-ptws}. By Lemma \ref{lem-Affd}(ii), $\dim_{\rm AFF}({\bf T}, W)\geq s_0$. Meanwhile,  since $X$ is $T_i^*$-invariant for all $1\leq i\leq m$ and $W\subset X$, it follows that $$\varphi^s(T_I^*P_W)=\varphi^s(T_I^*|_XP_W) \quad \mbox{ for all $s\geq 0$ and $I\in \Sigma_*$},
$$ from which we obtain that  $\dim_{\rm AFF}({\bf T}, W)=\dim_{\rm AFF}({\bf T}|_X, W)$. Then applying Lemma \ref{lem-Affd}(i) to ${\bf T}|_X$,  we obtain the reverse direction $\dim_{\rm AFF}({\bf T}, W)\leq s_0$.
\end{proof}

Our last result in this section provides some necessary conditions for $\dim_{\rm AFF}({\bf T}, W)$ to be strictly less than  $\min\{\dim W, \dim_{\rm AFF}({\bf T})\}$ in the case where $d=3$.

\begin{pro}
\label{pro-d=3}
Let $d=3$ and $W\in G(3, k)$, where $k=1$ or $2$. Suppose that $\dim_{\rm AFF}({\bf T}, W)<\min\{k, \dim_{\rm AFF}({\bf T})\}$.  Then ${\bf T}$ is reducible. More precisely, one of the following scenarios occurs.
\begin{itemize}
\item[(i)] $k=1$,  and either $T_i^*W=W$ for all $1\leq i\leq m$, or $W$ is contained in a $2$-dimensional subspace $V$ of $\R^3$ such that $T_i^*V=V$ for all $1\leq i\leq m$;
\item [(ii)] $k=2$, and either $T_i^*W=W$ for all $1\leq i\leq m$, or $W$ contains a $1$-dimensional subspace $V$ of $\R^3$ such that $T_i^*V=V$ for all $1\leq i\leq m$.
\end{itemize}
\end{pro}
\begin{proof}
We first consider the case where $k=1$.  Then part (i) follows directly from Proposition \ref{pro-one-dimensional}. Indeed, letting $X={\rm span}\left(\bigcup_{J\in \Sigma_*}T_J^*W\right)$, and given that $k=1$ and $\dim_{\rm AFF}({\bf T}, W)<\min\{1, \dim_{\rm AFF}({\bf T})\},$ it follows from  Proposition \ref{pro-one-dimensional} that  $X\neq \R^3$. Since $X$ is $T_i^*$-invariant, the conclusion in (i) follows.

In the remaining part of the proof,  we consider the case where $k=2$.  Write $$s_0=\min\{2, \dim_{\rm AFF}({\bf T})\}.$$
 Then  $P({\bf T}, s_0)\geq 0$, where $P({\bf T}, s_0)$ stands for the topological pressure of the sub-additive potential $\{\log \varphi^{s_0}(T_{\cdot|n}^*)\}_{n=1}^\infty$.

Let $\mu$ be an ergodic equilibrium measure for the potential $\{\log \varphi^{s_0}(T_{\cdot|n}^*)\}_{n=1}^\infty$. Let
$\Lambda_1\geq \Lambda_2\geq \Lambda_3$ be the Lyapunov exponents (counting multiplicity)  of the cocycle $x\mapsto T_{x_1}^*$ with respect to $\mu$, and let ${\bf v}(x)=\{v_i(x)\}_{i=1}^3$, where $x\in \Sigma'$, be the corresponding ordered basis of $\R^3$ given  in Theorem \ref{thm:oseledets}.
Below, we consider the following two cases separately: (a) $s_0\in [0,1]$;  (b) $s_0\in (1,2]$.

First assume $s_0\in [0,1]$. By Theorem \ref{thm:subadditive-VP},
\begin{equation}
\label{e-mulambda}
h_\mu(\sigma)+s_0\Lambda_1=P({\bf T}, s_0)\geq 0.
\end{equation}
 Meanwhile, by Proposition \ref{pro-key},
\begin{equation}
\label{e-mulambda1}
\lim_{n\to \infty} \frac{1}{n} \log \psi_W^{s_0}(x|n)=\Theta=s_0\sup_{J\in \Sigma_*} \Lambda_{p_1(T_J^*W,x)}
\end{equation}
for each $x\in \Sigma'$, where $\Theta=\lim_{n\to \infty} \frac{1}{n} \int \log \psi_W^{s_0}(x|n)\; d\mu(x)$. However, by assumption,  $\dim_{\rm AFF}({\bf T}, W)<s_0$. It follows from Lemma \ref{lem-Affd}(ii) that $P({\bf T}, W, s_0)<0$. Consequently, by Theorem \ref{thm:subadditive-VP}, $h_\mu(\sigma)+\Theta\leq P({\bf T}, W, s_0)<0$. This, combined with \eqref{e-mulambda} and \eqref{e-mulambda1}, yields that $\sup_{J\in \Sigma_*} \Lambda_{p_1(T_J^*W,x)}<\Lambda_1$ for $\mu$-a.e. $x\in \Sigma'$.  That is, for $\mu$-a.e.~$x\in \Sigma'$, $p_1(T_J^*W,x)\in \{2,3\}$ for all $J\in \Sigma_*$.  This implies that for $\mu$-a.e.~$x\in \Sigma'$,
$$
\bigcup_{J\in \Sigma_*} T_J^*W\subset {\rm span}\{v_2(x), v_3(x)\}.
$$
Since $\dim W=2$, it follows that $T_J^*W=W$ for all $J\in \Sigma_*$, and thus $W$ is $T_i^*$-invariant for all $1\leq i\leq m$.

Next assume that $s_0\in (1,2]$. Then, correspondingly, by  Theorem \ref{thm:subadditive-VP} and Proposition \ref{pro-key}, $$h_\mu(\sigma)+\Lambda_1+(s_0-1)\Lambda_2=P({\bf T}, s_0)\geq 0,$$
and
$$
\lim_{n\to \infty} \frac{1}{n} \log \psi_W^{s_0}(x|n)=\Theta=\sup_{J\in \Sigma_*} \Lambda_{p_1(T_J^*W,x)}+(s_0-1)\Lambda_{p_2(T_J^*W,x)},\quad x\in \Sigma'.
$$
Similarly, the assumption of $\dim_{\rm AFF}({\bf T}, W)<s_0$ implies that $$h_\mu(\sigma)+\Theta\leq P({\bf T}, W, s_0)<0,$$ and consequently,
$$
\sup_{J\in \Sigma_*} \Lambda_{p_1(T_J^*W,x)}+(s_0-1)\Lambda_{p_2(T_J^*W,x)}<\Lambda_1+(s_0-1)\Lambda_2
$$
for $\mu$-a.e.~$x\in \Sigma'$. Hence for $\mu$-a.e.~$x\in \Sigma'$,
$$
\left(p_1(T_J^*W,x), p_2(T_J^*W,x)\right)\neq (1,2) \quad \mbox{ for all }J\in \Sigma_*.
$$
This implies that $p_2(T_J^*W,x)=3$, and thus $v_3(x)\in T_J^*W$ for $\mu$-a.e.~$x\in \Sigma'$.  Consequently, for $\mu$-a.e.~$x\in \Sigma'$, $(T_J^*)^{-1}v_3(x)\in W$ for  all  $J\in \Sigma_*$.  Fix such a point $x$ and set
$$V={\rm span}\left(\bigcup_{J\in \Sigma_*}\left\{(T_J^*)^{-1}v_3(x)\right\}\right).
$$
It is easy to verify that $V\subset W$, and that $(T_J^*)^{-1}V=V$ for all $J\in \Sigma_*$;  equivalently, $T_J^*V=V$ for all $J\in \Sigma_*$. This completes the proof of (ii).
\end{proof}

\begin{rem}
In Proposition \ref{pro-d=3},  the conclusion that ${\bf T}$ is reducible follows alternatively  from Theorem \ref{thm-1.1}(i),  using the additional fact that when $d=3$, the tuple $\{T_i\}_{i=1}^m$ is irreducible if and only if $\{T_i^{\w 2}\}_{i=1}^m$ is irreducible (see, e.g., \cite[Lemma 3.3]{KaenmakiMorris2018} for a more general statement about this fact).
\end{rem}

\section{Examples for which $\overline{S}(\mu,{\bf T}, W)\neq \underline{S}(\mu,{\bf T}, W)$}
\label{S-8}

  In this section, we first provide an example (see Example \ref{not exact dimensional}) in which we construct a tuple ${\bf T}=(T_1,T_2, T_3)$ of $2\times 2$ antidiagonal matrices with norm $<1/2$,  a one-dimensional subspace $W$ of $\R^2$ and an ergodic $\sigma$-invariant measure $\mu$ on $\Sigma=\{1,2,3\}^\N$ such that $\overline{S}(\mu, {\bf T},W)\neq \underline{S}(\mu,{\bf T}, W)$. By Theorem \ref{thm-1.2}(iv), this implies that the projected measure $(P_W\pi^\ba)_*\mu$ is not exact dimensional for $\mathcal L^{6}$-a.e.~${\bf a}=(a_1, a_2, a_3)\in \R^6$.  This example was adapted from one previously constructed by the first author and Caiyun Ma \cite{FengMa2025}, who demonstrated that the orthogonal projection of an ergodic stationary measure, associated with a planar iterated function system (IFS) of similarities having a finite rotation group, may fail to be exact dimensional. Then, for a given finite tuple ${\bf T}$ of contracting antidiagonal  matrices,  we provide a criterion to determine  whether there exist an ergodic measure $\mu$ and a subspace $W\in G(2,1)$ such that $\overline{S}(\mu,{\bf T}, W)\neq \underline{S}(\mu, {\bf T},W)$; see Proposition \ref{pro-antidiagonal}.

   \begin{ex} \label{not exact dimensional}
         Define three $2\times 2$ matrices $T_1,T_2, T_3$ by

         $$T_{1} = \begin{pmatrix}
            0 & \frac{2}{5} \\[5pt]
            \frac{1}{5} & 0
        \end{pmatrix}, \qquad T_{2} = \begin{pmatrix}
            0 & \frac{2}{5} \\[5pt]
            \frac{1}{5} & 0
        \end{pmatrix}, \qquad T_{3} = \begin{pmatrix}
            0 & \frac{1}{5} \\[5pt]
            \frac{2}{5} & 0
        \end{pmatrix}. $$
               Let $W\subset \R^2$ be the $x$-axis, i.e. $ W = \{ (a,0): a\in \mathbb{R} \} $.
        Then $ P_{W} = \begin{pmatrix}
            1 & 0 \\ 0 & 0
        \end{pmatrix} $. Let $ \mu $ be the $({\bf p}, P)$-Markov measure on $ \Sigma $, where  $$ {\bf p}= (p_{1}, p_{2}, p_{3}) = \left(\frac{1}{4}, \frac{1}{4}, \frac{1}{2}\right) $$ and  $$ P = (p_{i,j})_{1\leq i, j \leq 3} =
        \begin{pmatrix}
            0 & 0 & 1 \\
            0 & 0 & 1 \\
            \frac{1}{2} & \frac{1}{2} & 0
        \end{pmatrix}.
        $$
        That is, $$\mu([x_1\ldots x_n])=p_{x_1}p_{x_1,x_2}\ldots p_{x_{n-1},x_n}$$ for each $n\geq 2$ and $x_1\ldots x_n\in \{1,2,3\}^n$.
        Then $\overline{S}(\mu,{\bf T}, W)\neq \underline{S}(\mu, {\bf T},W)$.
        \end{ex}

        \begin{proof}[{\it Justification}] It is easily checked that ${\bf p}P={\bf p}$ and that $P$ is positively irreducible (i.e.~for any $i,j$, there exists $n>0$ such that the $(i,j)$-entry of $P^n$ is positive). By \cite[Theorem 1.13]{Walters1982},  $\mu$ is an ergodic Markov measure.  It is well known (see e.g.~\cite[p.~103]{Walters1982} or \cite[p.~246]{Petersen1989}) that
        $$ h_\mu(\sigma) = - \sum_{ i,j  }p_{i}p_{i,j} \log p_{i,j} = \frac{\log 2}{2},$$
        and $\mu$ is supported on the Markov shift space
        $$
        \Omega=\left\{(x_n)_{n=1}^\infty\in \{1,2,3\}^\N\colon p_{x_n,x_{n+1}}>0 \mbox{ for all }n\geq 1\right\}.
        $$
                It is easy to check that $\Omega=\Omega_1\cup \Omega_2$, where
         \begin{align*}
         \Omega_{1} &  = \left\{ (x_{i})_{i=1}^{\infty}  :  x_{2i+1} \in \{1,2\} \mbox{ and }       x_{2i+2} = 3  \mbox{ for all }i \ge 0  \right\}, \\
        \Omega_{2} &  = \{ (x_{i})_{i=1}^{\infty}:  x_{2i+1} = 3 \mbox{ and } x_{2i+2}\in \{1, 2\}  \mbox{ for all }i\ge 0\}.
        \end{align*}
       Moreover,
       $$ \mu(\Omega_{1}) =\mu([13])+\mu([23])= \frac{1}{2}, \quad \mu(\Omega_{2})=\mu([31])+\mu([32]) =  \frac{1}{2}.$$

       Notice that
                        \begin{equation}
        \label{e-example}
            T_{1}T_{3} = T_{2} T_{3}
            = \begin{pmatrix}
               \frac{4}{25} & 0 \\[5pt]
                0 & \frac{1}{25}
            \end{pmatrix}
        \quad\mbox{
        and }\quad
            T_{3}T_{1} = T_{3} T_{2} = \begin{pmatrix}
                \frac{1}{25} & 0 \\[5pt]
                 0 & \frac{4}{25}
            \end{pmatrix}.
        \end{equation}
      A simple calculation using
    \eqref{e-example} yields that for each  $0 \leq s \leq 1 $,
               \begin{equation*}
            \lim_{n \rightarrow \infty} \frac{1}{n} \log \varphi^{s}( P_{W} T_{x|n} ) = \left\{
            \begin{array}{ll}
             s \log \left(2/5 \right) & \mbox{ if }x\in \Omega_1,\\
                                  s \log \left( 1/5 \right) & \mbox{ if }x\in \Omega_2.
                                \end{array}
                \right.
        \end{equation*}
    Moreover, it is direct to check that
    $$
   \lim_{n \rightarrow \infty} \frac{1}{n}\log \mu([x|n])=-\frac{\log 2}{2}=-h_\mu(\sigma)\qquad \mbox{ for all }x\in \Omega.
    $$
   By  the definitions \eqref{e-equi}-\eqref{e-Smux}, we have
    $$
    S(\mu,{\bf T}, W,x)=\left\{
            \begin{array}{cl}
              \dfrac{\log 2}{2\log (5/2)} & \mbox{ if }x\in \Omega_1,\\
            [10pt]
                   \dfrac{\log 2}{2\log (5)} & \mbox{ if }x\in \Omega_2.
                                \end{array}
                \right.
        $$
   Since $\mu$ is supported on $\Omega=\Omega_1\cup \Omega_2$, it follows that $$\overline{S}(\mu,{\bf T}, W)=\dfrac{\log 2}{2\log (5/2)}
   \quad \mbox{ and }\quad \underline{S}(\mu, {\bf T},W)=\dfrac{\log 2}{2\log (5)},
  $$
  so $\underline{S}(\mu,{\bf T}, W)\neq \overline{S}(\mu,{\bf T}, W)$.
\end{proof}

\begin{rem}
 One can check that the measure  $\mu$ constructed in Example \ref{not exact dimensional} is not ergodic with respect to $\sigma^2$. Actually in that example (or more generally, in the case that ${\bf T}$ is an arbitrary finite tuple of contracting antidiagonal $2\times 2$ real matrices), if  $\eta$ is a $\sigma$-invariant measure  that  is  ergodic with respect to $\sigma^2$, then $$\underline{S}(\eta, {\bf T},W')=\overline{S}(\eta,{\bf T}, W')\quad \mbox{ for all $W'\in G(2,1)$}.$$
 To see this, notice that $T_i$ is of the form
$$
\begin{pmatrix}
0 & c_i\\
d_i & 0
\end{pmatrix}
 $$
 for $i=1,2, 3$. Hence for $x\in \Sigma$ and $n\in \N$,
 $$
T_{x_1}T_{x_2}\cdots T_{x_{2n-1}}T_{x_{2n}}=\begin{pmatrix}
u_{2n}(x) & 0\\
0 & v_{2n}(x)
\end{pmatrix},
$$
with
$$
u_{2n}(x)=c_{x_1}d_{x_2}c_{x_3}d_{x_4}\cdots c_{x_{2n-1}}d_{x_{2n}},\; \quad v_{2n(x)}=d_{x_1}c_{x_2}d_{x_3}c_{x_4}\cdots d_{x_{2n-1}}c_{x_{2n}}.
$$
Since $\eta$ is ergodic with respect to $\sigma^2$, by the Birkhoff ergodic theorem, for $\eta$-a.e.~$x\in \Sigma$,
\begin{align*}
\lim_{n\to \infty}\frac{1}{n}\log u_{2n}(x)&=\sum_{i=1}^3\sum_{j=1}^3(\log (c_id_j))\eta([ij])\\
&=\sum_{i=1}^3\sum_{j=1}^3(\log (c_i)+\log (d_j))\eta([ij])\\
&=\sum_{i=1}^3(\log (c_id_i))\eta([i]).
\end{align*}
 Similarly, for $\eta$-a.e.~$x\in \Sigma$,  $\lim_{n\to \infty}\frac{1}{n}\log v_{2n}(x)=\sum_{i=1}^3(\log (c_id_i))\eta([i])$.
  This is,  the two Lyapunov exponents of the cocycle $x\to T_{x_1}^*$ with respect to $\eta$ are the same. Thus, by Proposition \ref{pro-planar case-measure}(2),  we have $\underline{S}(\eta, {\bf T},W')=\overline{S}(\eta,{\bf T}, W')$ for all $W'\in G(2,1)$.
\end{rem}

\begin{pro}
\label{pro-antidiagonal}
Let $m\geq 2$ and let ${\bf T}=(T_1,\ldots, T_m)$ be a tuple of $2\times 2$ real matrices of the form
$$
T_i=
\left(\begin{array}{cc}
0 & c_i\\
d_i & 0
\end{array}
\right),
 $$
 where $0<|c_i|, |d_i|<1$ for $i=1,\ldots, m$. Then the following two statements are equivalent.
 \begin{itemize}
\item[(i)] There exist distinct $i,j\in \{1,\ldots, m\}$ such that  $|c_i/d_i|\neq |c_j/d_j|$.
\item[(ii)] There exist an ergodic measure $\mu$ and a subspace $W\in G(2,1)$ such that $\overline{S}(\mu,{\bf T}, W)\neq \underline{S}(\mu,{\bf T}, W)$.
 \end{itemize}

\end{pro}
\begin{proof}
The proof is based on Proposition \ref{pro-planar case-measure}(2).  Notice that for each $x\in \Sigma$ and $n\in \N$,
\begin{equation}
\label{e-e8.1}
T_{x|2n}={\rm diag}(u_{2n}(x), v_{2n}(x)),
\end{equation}
where
\begin{equation}
\label{e-e8.2}
\begin{split}
u_{2n}(x)&=c_{x_1}d_{x_2}c_{x_3}d_{x_4}\cdots c_{x_{2n-1}}d_{x_{2n}},\\
v_{2n}(x)&=d_{x_1}c_{x_2}d_{x_3}c_{x_4}\cdots d_{x_{2n-1}}c_{x_{2n}}.
\end{split}
\end{equation}

We first assume that $|c_1/d_1|=|c_2/d_2|=\ldots=|c_m/d_m|$.  Then $|u_{2n}(x)|=|v_{2n}(x)|$ for all $x\in \Sigma$ and $n\in \N$. Consequently, for every ergodic $\sigma$-invariant measure $\mu$ on $\Sigma$, the two Lyapunov exponents $\Lambda_1, \Lambda_2$ of the cocycle $x\mapsto T_{x_1}^*$ with respect to $\mu$ are equal.  By Proposition \ref{pro-planar case-measure}(2), $\overline{S}(\mu,{\bf T}, W)= \underline{S}(\mu, {\bf T},W)$ for every subspace $W\in G(2,1)$. This proves the direction (ii) $\Longrightarrow$ (i).

Next we assume that there exist distinct $i,j\in \{1,\ldots, m\}$ such that  $$|c_i/d_i|\neq |c_j/d_j|.$$ Without loss of generality, we may assume  $i=1$, $j=2$ and  $|c_1d_2|>|c_2d_1|$. Let $N$ be a positive integer. Define two words $A,B\in \{1,2\}^{2N}$ by
$$A=(12)^{N},\qquad B=(12)^{N-1}21.$$
Then define a $\sigma^{2N}$-invariant compact subset $Y$ of $\Sigma$ by
$$
Y=\{y=w_1w_2\ldots w_n\ldots\colon w_i\in \{A, B\} \mbox { and } w_iw_{i+1}\neq BB \mbox{ for all } i\geq 1\}.
$$
It is direct to check that the topological entropy of $Y$ with respect to $\sigma^{2N}$ satisfies
$$
h_{\rm top}(Y, \sigma^{2N})=\log \left((\sqrt{5}+1)/2\right);
$$
see e.g.~\cite[Theorem 7.13(ii)]{Walters1982}. Define
$$
X=\bigcup_{i=0}^{2N-1}\sigma^i(Y).
$$
It is easy to check that $\sigma(X)\subset X$ and $$h_{\rm top}(X,\sigma)=\frac{1}{2N}h_{\rm top}(Y, \sigma^{2N})=\frac{1}{2N}\log\left((\sqrt{5}+1)/2\right).$$
Below, we show that when $N$ is large enough, for every ergodic $\sigma$-invariant measure $\mu$ supported on $X$ with positive entropy,
it holds that $\overline{S}(\mu, W)\neq \underline{S}(\mu, W)$ when $W$ is either the $x$-axis or the $y$-axis.

From the constructions of $Y$ and $X$, we see that for each $x=(x_n)_{n=1}^\infty\in X$, either
\begin{align*}
&\liminf_{n\to \infty}\frac{1}{n}\#\{1\leq k\leq n\colon x_{2k-1}x_{2k}=12\}\geq \frac{N-1}{N} \quad \mbox{ and }\\
&\liminf_{n\to \infty}\frac{1}{n}\#\{1\leq k\leq n\colon x_{2k}x_{2k+1}=21\}\geq \frac{N-3}{N},
\end{align*}
or
\begin{align*}
&\liminf_{n\to \infty}\frac{1}{n}\#\{1\leq k\leq n\colon x_{2k-1}x_{2k}=21\}\geq \frac{N-3}{N} \quad \mbox{ and }\\
&\liminf_{n\to \infty}\frac{1}{n}\#\{1\leq k\leq n\colon x_{2k}x_{2k+1}=12\}\geq \frac{N-1}{N}.\end{align*}
Since $|c_1d_2|\neq |c_2d_1|$, it follows that when $N$ is large enough,
\begin{equation}
\label{e-twolimits}
\limsup_{n\to \infty} \frac{1}{n}\log |u_{2n}(x)|\neq    \limsup_{n\to \infty} \frac{1}{n}\log |v_{2n}(x)| \quad \mbox { for all }x\in X;
\end{equation}
as one of these two limits is close to $\log |c_1d_2|$, and the other one is close to $\log |c_2d_1|$.

Now suppose that $N$ is large enough so that \eqref{e-twolimits} holds. Let $\mu$ be an ergodic $\sigma$-invariant measure supported on $X$ such that $h_\mu(\sigma)>0$ (for instance, we may choose $\mu$ as an ergodic invariant measure on $X$ with maximal entropy).  Let  $\Lambda_1\geq \Lambda_2$ be the Lyapunov exponents of the cocycle $x\mapsto T_{x_1}^*$ with respect to $\mu$. By \eqref{e-e8.1} and \eqref{e-e8.2}, for $\mu$-a.e.~$x\in \Sigma$,
\begin{equation}
\label{e-either}
\begin{split}
\mbox{ either }\quad  \lim_{n\to \infty} \frac{1}{2n}\log |u_{2n}(x)|&=\Lambda_1,  \quad   \lim_{n\to \infty} \frac{1}{2n}\log |v_{2n}(x)|=\Lambda_2,\\
\mbox{ or }\quad
\lim_{n\to \infty} \frac{1}{2n}\log |u_{2n}(x)|&=\Lambda_2,  \quad   \lim_{n\to \infty} \frac{1}{2n}\log |v_{2n}(x)|=\Lambda_1.
\end{split}
\end{equation}
This, combined with \eqref{e-twolimits}, implies that  $\Lambda_1>\Lambda_2$.  Now write
$$\Omega=\left\{x\in \Sigma\colon \lim_{n\to \infty} \frac{1}{2n}\log |v_{2n}(x)|=\Lambda_2\right\}.$$
By \eqref{e-e8.2}, $\Omega=\sigma^{-2}\Omega$, and moreover, $v_{2n}(x)\approx u_{2n}(\sigma x)$ for each $x\in \Sigma$.  It follows from \eqref{e-either} that
$$\mu(\Omega\cap \sigma(\Omega))=0\quad \mbox{ and } \quad \mu(\Omega\cup \sigma(\Omega))=1,$$
which implies that $\mu(\Omega)=1/2$. Let $\R^2=\bigoplus_{i=1}^2E_i(x)$, $x\in \Sigma'$, be the associated Oseledets splittings for the cocycle $x\mapsto T_{x_1}^*$ and $\mu$. Let $W$ be the $y$-axis, i.e., $W=\{(0,y)\colon y\in \R\}$. Since $\mu(\Omega)=1/2$, we have
\begin{equation}
\label{e-e8.3}
\mu\left\{x\in \Sigma'\colon E_2(x)=W\right\}=\mu\left\{x\in \Sigma'\colon E_2(x)=W^\perp\right\}=\frac12.
\end{equation}

 Finally, notice that
$$
\Lambda_1\leq \log \left(\max\{|c_1|, \ldots, |c_m|, |d_1|, \ldots, |d_m|\}\right)=:\lambda<0.
$$
We may also require that $N$ is large enough so that
$$h_{\rm top}(X,\sigma)=\frac{1}{2N}\log\left((\sqrt{5}+1)/2\right)<-\lambda\leq -\Lambda_1<-\Lambda_2.
$$
Since $\mu$ is supported on $X$, we have $h_\mu(\sigma)\leq  h_{\rm top}(X,\sigma)$ (see e.g.~\cite[Theorem 8.6]{Walters1982}). It follows that
\begin{equation}
\label{e-e8.4}
h_\mu(\sigma)>0,\quad h_\mu(\sigma)+\Lambda_2<0.
\end{equation}
Since $\Lambda_1>\Lambda_2$, by \eqref{e-e8.3}, \eqref{e-e8.4} and Proposition \ref{pro-planar case-measure}(2), we conclude that
$$\overline{S}(\mu, {\bf T},W)\neq \underline{S}(\mu,{\bf T}, W),\quad \overline{S}(\mu, {\bf T},W^\perp)\neq \underline{S}(\mu,{\bf T}, W^\perp).$$
  This proves the direction (i) $\Longrightarrow$ (ii).
\end{proof}

\section{Final remarks}
\label{S-9}

In the section we give a few remarks.

In our main theorems, the assumption that $\|T_i\|<1/2$ for $1\leq i\leq m$ can be weaken to $\max_{i\neq j} (\|T_i\|+\|T_j\|)<1$. Indeed the first assumption is only used to guarantee the self-affine transversality condition (see Lemmas \ref{lem-2.5} and \ref{lem-transversality}). As pointed in \cite[Proposition 10.4.1]{BSS2023},  the second assumption is sufficient for the self-affine transversality condition.

In the special case where $T_i=\rho_iO_i$ for all $1\leq i\leq m$, with $0<\rho_i<1$ and $O_i$ being orthogonal, it is straightforward to verify that for each $W\in G(d,k)$ and every $s\in [0,k]$,  we have
$$
\varphi^s(P_WT_I)=(\rho_I)^s
$$
for $I\in \Sigma_*$. Furthermore,  for each ergodic $\sigma$-invariant measure $\mu$, the Lyapunov exponents $\Lambda_1,\ldots, \Lambda_m$ for the cocycle $x\mapsto T_{x_1}^*$ with respect to $\mu$ are all equal. Thus, by the definition of $\dim_{\rm AFF}({\bf T}, W)$ (see \eqref{e-aff}) and Lemma \ref{lem-smuw}, we have
$$\dim_{\rm AFF}({\bf T}, W)=\min\{\dim W, \dim_{\rm AFF}({\bf T})\},$$
and
$$\overline{S}(\mu,{\bf T}, W)=\underline{S}(\mu,{\bf T}, W)=\min\{\dim W, \dim_{\rm LY}(\mu, {\bf T})\}.$$
Therefore, in this case, there is no dimension drop regarding the projections of $K^\ba$ and $\pi^\ba_*\mu$ for almost all $\ba$ if $\rho_i<1/2$ for all $i$.

We remark that Example \ref{not exact dimensional} provides a negative answer to a question posed in \cite[p.~709]{Feng2023} whether every ergodic stationary measure associated with an affine IFS is dimension conserving with respect to $P_{W^\perp}$ for almost every subspace $W$ (with respect to the so-called Furstenberg measures); see the remark after \cite[Theorem 1.6]{Feng2023}.
%{\color{red}
%\begin{itemize}
%\item Self-similar case.
%\item Existence of local dimension in the typical case. It is expected that this might be true in the deterministic case.
%\item Results on the existence of interior of the projected sets.
%\item What else?
%\end{itemize}
%}
\appendix

\section{Main notation and conventions}
\label{B}
For the reader's convenience, we summarize in Table \ref{table-1} the main notation and typographical conventions used in this paper.
\begin{table} [H]
\centering
\caption{Main notation and conventions}
\vspace{0.05 in}
\begin{footnotesize}
\begin{raggedright}
\begin{tabular}{p{1.5 in} p{4 in} }
\hline \rule{0pt}{3ex}
%& Effect of increasing \\
%Definition & substrate stiffness \\[3pt]
%\hline \rule{0pt}{3ex}
${\bf T}$ & A tuple $(T_1,\ldots, T_m)$ of invertible $d\times d$ real matrices with $\|T_i\|<1$\\
$T_{I}$ & $T_{i_1}\cdots T_{i_n}$ for $I=i_1\ldots i_n$\\
$T_{I}^*$ & $(T_I)^*$, where $*$ stands for transpose \\
$\{f_i^\ba\}_{i=1}^m$ & An IFS $\{T_ix+a_i\}_{i=1}^m$ on $\R^d$ with $\ba=(a_1,\ldots, a_m)$ (cf.~Section \ref{S1})\\
$K^\ba$ & The attractor of $\{f_i^\ba\}_{i=1}^m$\\
$(\Sigma, \sigma)$ & One-sided full shift over the alphabet $\{1,\ldots,m\}$\\
$\pi^\ba\colon \Sigma\to K^\ba$ & Coding map associated with  $\{f_i^\ba\}_{i=1}^m$ (cf.~Section \ref{S1})\\
$g_*\mu$ & Push-forward of $\mu$ by $g$, i.e.  $g_*\mu=\mu\circ g^{-1}$\\
$P_W$ & Orthogonal projection onto $W$\\
$\varphi^s$ & Singular value function (cf.~\eqref{e-singular})\\
$\lfloor s\rfloor$ & Integral part of $s$\\
$\dim_{\rm AFF}({\bf T})$ & Affinity dimension of ${\bf T}$ (cf.~Definition \eqref{e-aff})\\
$\dim_{\rm AFF}({\bf T}, W)$ &  (cf.~\eqref{e-affine})\\
${\rm Mat}_{d}(\R)$ & The set of $d\times d$ real matrices\\
$\mathcal H^s$ & $s$-dimensional Hausdorff measure\\
${\rm GL}_{d}(\R)$ & The set of invertible $d\times d$ real matrices\\
$G(d,k)$ & Grassmann manifold  of $k$-dimensional linear subspaces of $\R^d$\\
$\gamma_{d,k}$& The natural invariant measure on $G(d,k)$\\
$S_n(\mu,{\bf T},W,x)$ & (cf~\eqref{e-2.2}, \eqref{e-equi})\\
$S(\mu, {\bf T}, W,x)$ & (cf~\eqref{e-Smux})\\
$\overline{S}(\mu, {\bf T}, W)$, $\underline{S}(\mu, {\bf T}, W)$ & (cf.~\eqref{e-Gammamu})\\
$\dim_{\rm LY}(\mu, {\bf T})$ & Lyapunov dimension of $\mu$ w.r.t.~${\bf T}$  (cf.~Definition \ref{de-Lyapunov})\\
$P(\sigma, \{f_n\}_{n=1}^\infty)$ & Topological pressure of a subadditive potential $\{f_n\}_{n=1}^\infty$ on $(\Sigma,\sigma)$ (cf.~Section \ref{S-subadditive})\\
$h_\mu(\sigma)$ & Measure-theoretic entropy of $\mu$ w.r.t.~$\sigma$\\
$\alpha_i(A)$, $i=1,\ldots, d$ & The $i$-th singular value of a $d\times d$ matrix $A$ (cf.~Section \ref{S1})\\
$\alpha({\bf v})$ & Smallest angle generated by an ordered basis ${\bf v}$ of an ambient space (cf. Definition \ref{de-2.1})\\
$\Lambda_i$, $i=1,\ldots,d$ & The $i$-th Lyapunov exponent of the cocycle $x\mapsto T_{x_1}^*$ w.r.t.~$\mu$ (cf.~\eqref{e-Lambdai} or Theorem \ref{thm:oseledets})\\
$\psi_{W}^s$ & (cf. \eqref{e-psiw})\\
$P({\bf T}, W, s)$ & Topological pressure of the subadditive potential $\{\log \psi_W^s(\cdot|n)\}_{n=1}^\infty$\\
$P({\bf T}, s)$ & Topological pressure of the subadditive potential $\{\log \varphi^s(T^*_{\cdot|n})\}_{n=1}^\infty$\\
${\rm span}(E)$ & Smallest linear subspace of the ambient space  that contains
$E$\\
${\bf p}(W, {\bf v})$ & Pivot position vector of $W\in G(d,k)$ with respect to an ordered basis ${\bf v}$ of $\R^d$
 (cf.~Definition \ref{de-pivot})\\
\hline
\end{tabular}
\label{table-1}
\end{raggedright}
\end{footnotesize}
\end{table}

{\noindent \bf Acknowledgements}.
     The authors are grateful to Amir Algom and Xiong Jin for pointing out several related references, and to Julien Barral for helpful comments.  This research was partially supported by the General Research Fund grant (projects CUHK14303524 and CUHK14308423) from the
Hong Kong Research Grant Council, and by a direct grant for research from the Chinese University
of Hong Kong.

\end{document}